\documentclass[11pt, reqno]{amsart}

\usepackage{caption, amssymb,amsfonts,amsmath, graphicx, geometry}
\usepackage{pgf,tikz}
\usepackage{pgfplots}
\usepackage[nice]{nicefrac}
\usepackage{mathrsfs}
\usepackage[colorlinks=true,bookmarksnumbered=true, pdfauthor={Jacob, M\"orters}]{hyperref}
\usepackage{diagbox}
\usepackage{multirow}
\usepackage{colortbl}
\usepackage{comment}
\usepackage{hhline}

\usepackage{enumerate}
\usepackage{enumitem}
\usepackage[normalem]{ulem}

\definecolor{mygreen}{rgb}{0.13,0.55,0.13}

\definecolor{LightGreen}{HTML}{34FF80}

\definecolor{LightCyan}{rgb}{0.88,1,1}
\definecolor{LightYellow}{HTML}{FFFF00}

\newcolumntype{a}{>{\columncolor{LightYellow}}c}
\newcolumntype{b}{>{\columncolor{LightCyan}}c}
\newcolumntype{d}{>{\columncolor{LightGreen}}c}

\def \AL#1{{\color{red}#1}}

\def \margEJ#1{\marginpar{\raggedright\parindent=0pt\tiny {\color{magenta}#1}}}

\newcommand{\cl}{\mathfrak{c}_1}
\newcommand{\cu}{\mathfrak{c}_2}

\usetikzlibrary{arrows,decorations.markings}

\definecolor{miazul}{RGB}{0,197,253}
\definecolor{miyellow}{RGB}{190,190,30}
\definecolor{salmon}{RGB}{235,126,78}
\definecolor{naranja}{RGB}{250,83,0}
\definecolor{migreen}{RGB}{49,156,85}

\newlength{\hatchspread}
\newlength{\hatchthickness}
\newlength{\hatchshift}
\newcommand{\hatchcolor}{}
\tikzset{hatchspread/.code={\setlength{\hatchspread}{#1}},
	hatchthickness/.code={\setlength{\hatchthickness}{#1}},
	hatchshift/.code={\setlength{\hatchshift}{#1}},
	hatchcolor/.code={\renewcommand{\hatchcolor}{#1}}}
\tikzset{hatchspread=3pt,
	hatchthickness=0.4pt,
	hatchshift=0pt,
	hatchcolor=black}
\pgfdeclarepatternformonly[\hatchspread,\hatchthickness,\hatchshift,\hatchcolor]
{custom north west lines}
{\pgfqpoint{\dimexpr-2\hatchthickness}{\dimexpr-2\hatchthickness}}
{\pgfqpoint{\dimexpr\hatchspread+2\hatchthickness}{\dimexpr\hatchspread+2\hatchthickness}}
{\pgfqpoint{\dimexpr\hatchspread}{\dimexpr\hatchspread}}
{
	\pgfsetlinewidth{\hatchthickness}
	\pgfpathmoveto{\pgfqpoint{0pt}{\dimexpr\hatchspread+\hatchshift}}
	\pgfpathlineto{\pgfqpoint{\dimexpr\hatchspread+0.15pt+\hatchshift}{-0.15pt}}
	\ifdim \hatchshift > 0pt
	\pgfpathmoveto{\pgfqpoint{0pt}{\hatchshift}}
	\pgfpathlineto{\pgfqpoint{\dimexpr0.15pt+\hatchshift}{-0.15pt}}
	\fi
	\pgfsetstrokecolor{\hatchcolor}
	\pgfusepath{stroke}
}

\usetikzlibrary{calc}
\usepackage{relsize}

\tikzset{fontscale/.style = {font=\relsize{#1}}
}
\usepackage{smartdiagram}

\newcommand{\eps}{\varepsilon}
\newcommand{\ssup}[1] {{\scriptscriptstyle{({#1}})}}
\newcommand{\one}{{\mathsf 1}}

\newcommand{\R}{\mathbb R}

\newcommand{\1}{{\bf{1}}}

\newcommand{\E}{\mathbb E}
\renewcommand{\P}{\mathbb P}

\newcommand{\Vo}{\mathscr V}
\newcommand{\Md}{M^{\rm disc}}

\newcommand{\Npr}{N^{\rm{pr}}}
\newcommand{\Ncl}{N^{\rm{cl}}}
\newcommand{\Nsq}{N^{\rm{sq}}}
\newcommand{\Nss}{N^{\rm{ss}}}

\renewcommand{\phi}{\varphi}
\newcommand{\TL}{T_{loc}}

\newcommand{\N}{\mathbb N}

\renewcommand{\P}{\mathbb P}
\newcommand{\vertiii}[1]{{\left\vert\kern-0.25ex\left\vert\kern-0.25ex\left\vert #1 
    \right\vert\kern-0.25ex\right\vert\kern-0.25ex\right\vert}}

\newcommand{\de}{\mathrm{d}} 

\newcommand{\F}{\mathfrak{F}}

\newcommand{\Tloc}{T^{\rm{loc}}}
\newcommand{\astr}{a_{\rm{str}}} 
\newcommand{\asl}{a_{\rm{slow}}} 
\newcommand{\assl}{a_{\rm{stsl}}} 
\newcommand{\CMI}{ \mathfrak{c}}
\newcommand{\CmMI}{ \mathfrak{c}}
\newcommand{\rhoMI}{\rho} 
\newcommand{\ratiost}{r} 

\newcommand{\Tex} {T_{\rm ext}}
\newcommand{\tree}{\mathfrak{t}}

\newcommand{\U}{\mathcal U}
\newcommand{\Rec}{\mathcal R}
\newcommand{\I}{\mathcal I}
\newcommand{\St}{\mathscr S}
\newcommand{\Co}{\mathscr C}

\newcommand{\heap}[2]  {\genfrac{}{}{0pt}{}{#1}{#2}}
\newcommand{\sfrac}[2] {\mbox{$\frac{#1}{#2}$}}

\newtheorem{theorem}{Theorem}
\newtheorem{definition}{Definition}
\newtheorem{corollary}[theorem]{Corollary}
\newtheorem{lemma}{Lemma}
\newtheorem{proposition}{Proposition}
\newtheorem{remark}{Remark}

\title[The contact process on evolving scale-free networks]
{Metastability of the contact process\\ on slowly evolving scale-free networks} 

\author[Emmanuel Jacob, Amitai Linker and Peter M\"orters]{Emmanuel Jacob, Amitai Linker and Peter M\"orters}

\begin{document}
\maketitle

\vspace{-0.6cm}

\begin{quote}
{\small {\bf Abstract:} }
We investigate the contact process on scale-free networks evolving by a stationary dynamics whereby each vertex independently updates its connections with a rate depending on its 
power.  
This rate can be slowed down or speeded up by virtue of decreasing or increasing a parameter~$\eta$, with  $\eta\downarrow-\infty$  approaching the static and $\eta\uparrow\infty$  the mean-field case. We identify the regimes of slow, fast and ultra-fast extinction of the contact process. Slow extinction occurs in the form of metastability, when the contact process maintains a certain density of infected states for a 
time exponential in the network size. 
In our main result we identify the metastability exponents, which describe 
the decay of these densities as the infection rate goes to zero, in dependence on $\eta$ and the power-law exponent $\tau$. 
While the fast evolution cases 
have been treated in a companion paper, Jacob, Linker, M\"orters~(2019), the present paper looks at the significantly more difficult 
cases of slow network evolution. We describe various effects, like degradation, 
regeneration and depletion, which lead to a rich picture featuring numerous first-order 
phase transitions for the  metastable exponents. To capture these effects in our upper bounds we develop a new  martingale based proof technique combining a local and global analysis of the process. 
\end{quote}

\vspace{0.4cm}

\noindent
{\emph{MSc Classification:} Primary 05C82; Secondary 82C22.

}



\pagebreak[3]

\section{Introduction}

The aim of the present paper and its companion~\cite{JLM19} 
is to provide a paradigmatic study of the combined effects of temporal and spatial variability of a graph on the spread of diffusions on this graph, and to provide the mathematical techniques for such a study if the graph process is stationary and autonomous.
There has been considerable interest in the probability literature on the behaviour of diffusions on a variety of evolving graph models, some recent examples are \cite{JLM22, SS23, SV23, FMO24, CO24}. 
Our particular interest in this work is in the phase transitions that occur when we tune parameters controlling the speed and inhomogeneity of the graph evolution.
%
For this study our choice is to look at 
\begin{itemize}[leftmargin=*]
\item finite graphs, which as their size $N$ goes to infinity have a heavy tailed asymptotic degree distribution. These are often called \emph{scale-free networks} in the literature \cite{BJR, vdH17, vdH23}. The tail exponent $\tau$ is the parameter controlling the spatial inhomogeneity of the graph.\pagebreak[3]\smallskip
\item a stationary dynamics based on \emph{updating vertices} at a rate possibly depending on the power of the vertex, see e.g.~\cite{JM15, ZMN}. Upon updating a vertex resamples all its connections independently with a probability that depends on its power
and the power of the potential connecting vertex.
This updating mechanism loosely speaking mimics movement of the individuals represented by the vertices. The rate at which a vertex of expected degree $k$ is updated is set to be proportional to $k^\eta$, where $\eta\in\R$ is the parameter controlling the speed of the graph process.\smallskip
 \item a diffusion modelled by the \emph{contact process}, see e.g.~\cite{DL88, DS88, L99, MV16}. In this epidemic process healthy vertices recover at rate one, and infected vertices infect their neighbours at rate~$\lambda$. Recovered vertices can be reinfected, so that the infection can use edges that remain in the graph long enough several times for an infection. This creates interactions between the contact process and the autonomous graph evolution, which are at the heart of this paper.%
 \medskip%
 \end{itemize}
 
  We now heuristically describe the main effects behind the results of this paper. 
 \medskip
 
 In the case
 of a static scale-free network if an infected vertex of degree $k\gg \lambda^{-2}$ recovers, it typically has of order $\lambda k$ infected neighbours. The probability that none of these neighbours reinfects the vertex before it recovers is roughly \smash{$(1-\frac{\lambda}{\lambda+1})^{\lambda k}\sim e^{-k\lambda^2}$}. 
Hence the vertex can
hold the infection for a time which is exponential in the vertex degree  before a sustained recovery. 
In a static scale-free network there are sufficiently many well-connected vertices of high degree for this {local survival} mechanism to keep an established infection alive for a time exponential in the graph size, see~\cite{BB+, CD09, GMT05}.
\medskip
 
 If the graph is evolving in time, even when the degree of the powerful vertex remains high throughout, its capacity to use the neighbourhood for local survival is significantly reduced compared to the static case and we observe a substantial \emph{degradation} in the ability of the vertex to sustain an infection locally.
 Already for slow evolutions, if $\eta<0$, the degradation effect is significant. On recovery of a vertex of degree $k$ the number of infected neighbours is still of order \smash{$\lambda k$} but the probability that none of them reinfects the vertex before it updates is now roughly 
 $$\frac{k^\eta}{\lambda^2 k +k^\eta} \sim \lambda^{-2} k^{\eta-1}
 \quad \text{ if $k \gg \lambda^{-\frac2{1-\eta}}$.} $$
 Hence the time that a vertex can survive locally is
 of order $\lambda^{2}k^{1-\eta}$. Although this
is just polynomial in the vertex degree and hence much 
smaller than in the static case, this local survival mechanism 
together with a spreading strategy can keep an established 
infection alive for a time exponential in the graph size, we then speak of \emph{slow extinction}.
\medskip

Updating of vertices has a second effect, which is beneficial to the survival of the infection, namely that 
in the time up to a sustained recovery updates of a powerful vertex will regenerate its neighbourhood and the overall connectivity of the graph is improved allowing it to spread the infection more efficiently. When a vertex consistently has degree~$k$ over an infection period of length  $\lambda^{2}k^{1-\eta}$ we see updates at a rate $k^{\eta}$ and hence if $\lambda^2k\gg 1$ we have of order $\lambda^2 k$ updates before a sustained recovery. 
This effect of \emph{regeneration} gives the vertex an increased effective degree. 
More precisely, in a period of length $k^{-\eta}$ between two updates
the probability that one of the $\lambda k$ neighbours of this vertex passes the infection to
a second vertex of degree $k$ is of order \smash{$\lambda^2 k^{1-\eta}\frac{k}N$} where $\frac{k}N$ is the order of the probability that there is an edge between a given neighbour of the first and the second vertex. As there are 
$\lambda^2 k$ updates before a sustained recovery the overall probability that an infected vertex with
degree $k\gg \lambda^{-2}$ passes the infection to a vertex of the same degree via an intermediary vertex is \smash{$\lambda^4 k^{3-\eta}\frac{1}N$} and if the proportion
$a(k)$ of vertices with degree at least $k$ satisfies
 $a(k)\gg \lambda^{-4}k^{\eta-3}$
 the strategy of 
 \emph{delayed indirect spreading}, where the infection is passed
 from one powerful vertex to another via an intermediary vertex, yields slow 
 extinction. 
 \pagebreak[3]
 \medskip%

 A similar argument suggests that, as an infected vertex of degree $k$ has of order $\lambda^2 k$ updates before a sustained recovery, if \smash{$\frac{p(k)}N$} denotes the probability that two vertices of degree~$k$ are connected by an edge, then the probability that by the time of its sustained recovery an infected vertex of degree $k$ directly  infects another vertex  of degree $k$ is \smash{$\lambda^{3}k^{1-\eta}\frac{p(k)}{N}$}. Hence 
 if the proportion $a(k)$ of vertices with degree at least $k$ satisfies
 $$a(k)\gg \lambda^{-3}k^{\eta-1} p(k)^{-1}$$
 the strategy of 
 \emph{delayed direct spreading}, where the infection is passed
 from one powerful vertex to another directly, should yield slow 
 extinction of the infection. However,	this is only correct if the update rate $k^\eta$ of a vertex of degree $k$  is at least of the order of the infection rate $\lambda$. If 	$k^\eta\ll\lambda$, then the dynamics of the subnetwork of powerful vertices is much slower than the spread of the infection. The infections passed from a powerful vertex will then typically reach vertices that are already infected, an effect we call \emph{depletion}.
	In those cases the infection will spread at rate $k^\eta$ instead of $\lambda$,  so that overall delayed direct spreading can only give slow extinction if  $a(k)\gg k^{-\eta}\lambda^{-2}k^{\eta-1} p(k)^{-1}$. Note that whether the delayed direct spreading mechanism is possible 
depends not only on the tail exponent $\tau$ of the degree distribution but also on the finer network geometry through the quantity~$p(k)$.%
	\medskip%
	
	When the graph is very densely connected both spreading mechanisms are equally effective and we speak of  \emph{delayed concurrent spreading}. In this case the infection
retains a density of infected vertices of the same order if either all edges between high degree vertices
were removed or all low degree neighbouring vertices of high degree vertices were leaves.%
	 \medskip%

 The degradation and regeneration effects are even stronger for fast evolutions, i.e.\ when $\eta\ge0$. On recovery of a vertex of degree $k$ the number of infected neighbours is now of order \smash{$k \frac{\lambda}{\lambda+1+k^\eta} \sim \lambda k^{1-\eta}$} and hence the probability that none of them reinfects the vertex before it updates is roughly 
 $$\frac{k^\eta}{\lambda^2 k^{1-\eta} +k^\eta} \sim \lambda^{-2} k^{2\eta-1}
 \quad \text{ if $\lambda^{2}k^{1-\eta}\gg k^\eta$.} $$
  Now only if $\eta\leq\frac12$ the local survival mechanism persists and 
vertices with degree $k$ satisfying \smash{$k^{1-2\eta}\gg \lambda^{-2}$} can hold the infection 
for a time of order \smash{$\lambda^{2}k^{1-2\eta}$}. In this case we can also benefit from the effect of regeneration without depletion. As a result delayed indirect spreading is possible 
on the set of vertices with expected degree at least $k$ 
if \smash{$a(k)\gg\lambda^{-4} k^{2\eta-3}$} and 
delayed indirect spreading is possible if 
\smash{$a(k)\gg k^{2\eta-1}\lambda^{-3}p(k)^{-1}$}.\pagebreak[3]\medskip

Heuristically speaking, the four strategies above, 
delayed indirect  and delayed direct spreading, together with the strategies of quick indirect and direct spreading, which spread the infection without using a local survival mechanism, compete for domination. Each strategy sustains the infection on a set of powerful vertices called the \emph{stars} and the strategy that operates with the smallest threshold degree $k$, or equivalently the largest proportion $a(k)$ of stars among the vertices, dominates. The infected neighbours of the stars form a set comprising an approximate  proportion $\lambda ka(k)$ of vertices which remain infected for a time exponential in the graph size. During that time the proportion of infected vertices remains essentially constant, this effect is called \emph{metastability}. The constant proportion of infected vertices is  called the \emph{metastable density}. \medskip

In this project we establish \emph{metastable exponents} describing the decay of the metastable densities as $\lambda\downarrow0$ for various network models. The exponents are a function of  the power-law exponent $\tau$ and the updating exponent $\eta$. They characterise the underlying survival strategies of the infection and thereby rigorously underpin the heuristics. Proofs of the upper and lower bounds require novel techniques, which are developed in this project:
\begin{itemize}[leftmargin=*]
\item \emph{Lower bounds} for metastable densities are based on the identification of the optimal survival strategies for the infection. To show that these strategies are successful a technically demanding coarse graining technique has to be used. This is done in \cite{JLM19} for fast network evolutions, but is getting much harder for slow evolutions as there is much less independence in the system and additional effects like depletion have to be handled. We develop the necessary novel techniques for slow evolutions in Section~\ref{sec:lower_bounds} of this paper.\smallskip

\item \emph{Upper bounds} for metastable densities \smallskip
\begin{itemize}
\item in the fast evolution case are based on model simplification in conjunction with supermartingale arguments, a basic version of this argument is developed in~\cite{JLM19}. The argument can be substantially refined 
    using a further distinction of vertices and leads to Theorem~\ref{teoupper_vertex_improved}, which is 
    proved in Section~\ref{sec_upper_new}.
\smallskip
\item upper bounds for very slow evolutions are based on local approximation and results for the contact process on trees. This is done in \cite{JLM22} for the same stationary graph model but running with a different stationary dynamics. Arguments from~\cite{JLM22} can be adapted to our case, see Section~\ref{sec_upper_overall}.\smallskip
\item As becomes apparent in Section~\ref{sec_upper_overall} the techniques above do not suffice to give matching upper bounds for the entire slow evolution domain.  In order to close this gap a completely new tool has to be developed combining global and local analysis into a  single argument. This is the main technical innovation of this paper. It  will be presented 
as Theorem~\ref{teoupper_optimal}, which is proved in Section~\ref{sec_upper_new}.\end{itemize}
\end{itemize}\pagebreak[3]
In the next section we give full details of the network models we consider and state our main result for the networks of primary interest based on the factor and preferential attachment kernel. Further results for more general networks are deferred to Sections~\ref{sec:lower_bounds} to~\ref{sec_upper_new}.
\pagebreak[3]
\medskip

\section{Main result}
\label{sec:main}

We now define a stationary evolving graph or network~\smash{$(\mathscr G^{\ssup N}_t \colon t\geq 0)_{N\in\N}$}. Take a function

$${\kappa\colon (0,1] \to (0,\infty)},$$
and a kernel
$$p\colon (0,1] \times (0,1] \to (0,\infty).$$
\smallskip

\noindent
The vertex set of the graph $\mathscr G^{\ssup N}_t$ is $\{1,\ldots,N\}$ for any $t\geq0$. The graph is evolving by \emph{vertex updating}: Each vertex~$i$ has an independent Poisson clock with rate~$\kappa_i:=\kappa(i/N)$.
When it strikes, the  vertex updates, which means:
\begin{itemize}
\item All adjacent edges are removed, and 
\item new edges $i\leftrightarrow j$ are formed with probability 
$$p_{i,j}:= \frac{1}{N} \,  
p\Big(\frac{i}{N}, \frac{j}{N}\Big)\wedge 1.$$
independently for every $j\in\{1,\dots,N\}\setminus\{i\}$.
\end{itemize}
We denote by $(\mathscr G^{\ssup N}_t \colon t\geq 0)$
the stationary graph process under this dynamics.
\medskip

For the kernel $p\colon (0,1]\times(0,1]\rightarrow (0,\infty)$ we make the following assumptions:
\begin{enumerate}
\item $p$ is symmetric, continuous and decreasing in both parameters,
\item there is some $\gamma\in(0,1)$ and constants $0<\cl<\cu$ such that for all $a\in(0,1)$,
\begin{equation}\label{condp}
\cl a^{-\gamma}\le p(a,1) \le \int_0^1 p(a,s) \, \mathrm ds<\cu a^{-\gamma}.
\end{equation}
\end{enumerate}
These properties are satisfied by the 
\begin{itemize}
\item  \emph{factor kernel}, defined by $p(x,y)=\beta x^{-\gamma} y^{-\gamma}$, or the 
\smallskip
\item \emph{preferential attachment kernel}, defined by 
$p(x,y)=\beta (x \wedge y)^{-\gamma} (x \vee y)^{\gamma-1},$ 
\end{itemize}
for~$\beta>0$. These kernels correspond to the connection probabilities of the Chung-Lu or configuration models in the former, and preferential attachment models in the latter case, when the vertices are ranked by decreasing power. While we state the main results below for these principal kernels, our results are by no means limited to these examples.
\medskip

 For a sequence $(i_N)$ of vertices $i_N\in \mathscr G_t^{\ssup N}$
such that $i_N\sim xN$ for some $x\in (0,1]$  the degree distribution of the vertex $i_N$ converges to a Poisson distribution with parameter \smash{$\int_0^1 p(x,y) dy$}, so its typical degree is of order $(i_N/N)^{-\gamma}$. Moreover, the empirical degree distribution of the network converges in probability to a limiting degree distribution $\mu$, which is a mixed Poisson distribution obtained by taking $x$ uniform in $(0,1)$, then a Poisson distribution with parameter \smash{$\int_0^1 p(x,y) dy$}, 
see~\cite[Theorem 3.4]{vdH23}. 
This distribution satisfies 
$$\mu(k)=k^{-\tau+o(1)}\qquad
\text{ as $k\to\infty$, with $\tau=1+\tfrac1\gamma$,}$$ 
i.e., at any time $t$ the network \smash{$(\mathscr G^{\ssup N}_t)_{N\in\N}$} is scale-free with
power-law exponent~$\tau>2$. 
\medskip
\pagebreak[3]

For fixed parameters \smash{$\eta\in\R$} and $\kappa_0>0$ we look at 
\[
\kappa(x)\;=\;\kappa_0x^{-\gamma\eta}
\qquad \mbox{ for } x\in(0,1].
\]
With this choice the update rate of a vertex is approximately proportional to its expected degree to the power $\eta$. 
The parameter $\eta$ determines the speed of the network evolution.  When $\eta=0$ network evolution and contact process operate on the same time-scale. 
If $\eta>0$ the network evolution is faster, but note that now not all vertices update at the same rate. We let powerful vertices update faster in order to `zoom into the window' where the qualitative behaviour of the contact process changes.
If $\eta<0$ the network evolution is slower and also here it is particularly slow for the more powerful vertices. 
\medskip

{The infection is now described by a process $(X_t(i), i\in \{1,\ldots,N\}  \colon t\geq 0)$ with values 
in~\smash{$\{0,1\}^N$}, such that $X_t(i)=1$ if $i$ is 
infected at time~$t$, and $X_t(i)=0$ if $i$ is healthy at time $t$.  The infection process 
associated to a starting set $A_0$ of infected vertices, is the c\`adl\`ag process with $X_0(i)=\one_{A_0}(i)$ 
evolving according to the following rules:
\begin{itemize}
	\item to each vertex $i$ we associate an independent Poisson process $\mathcal R^i$ with intensity one, which represents recovery times,  i.e.\  if $t\in \mathcal R^i$, then $X_t(i)=0$ whatever~$X_{t-}(i)$. \smallskip
	\item to every unordered pair $\{i,j\}$ of distinct vertices we associate an independent Poisson process $\mathcal I_0^{ij}$ with intensity
	$\lambda$. If $t  \in \mathcal I_0^{ij}$ and 
	$\{i,j\}$ is an edge in $\mathscr G^{\ssup N}_t$, then 
	$$
	(X_t(i),X_t(j))=
	\left\{
	\begin{array}{rl}
	(0,0) & \mbox{ if } (X_{t-}(i),X_{t-}(j))=(0,0). \\
	(1,1) & \mbox{ otherwise.} \\
	\end{array}
	\right.
	$$
\end{itemize}
The process $({\mathscr G}^{\ssup N}_t, X_t \colon t\geq 0)$ is a Markov process describing the simultaneous evolution of the network and of the infection. We denote by  $(\F_t\colon t\ge 0)$ its canonical filtration.}
\medskip

We start the process with the stationary distribution of the graph and all vertices infected. Just like in the
static case there is a finite, random extinction time~$T_{\rm ext}$ and we say that there is
\begin{itemize}
\item \emph{ultra-fast extinction}, if there exists $\lambda_c>0$ such that for all infection rates $0< \lambda< \lambda_c$ the expected extinction time is bounded by 
{a subpolynomial function of $N$;}\smallskip
\item \emph{fast extinction}, if there exists $\lambda_c>0$ such that for all infection rates $0< \lambda< \lambda_c$ the expected extinction time is bounded by 
a polynomial function of $N$;\smallskip
\item \emph{slow extinction} if, for all $\lambda>0$, there exists some $\eps>0$ such that
$T_{\rm ext}\geq  e^{\eps N}$ with high probability.\smallskip
\end{itemize}
Our first interest is in characterising phases of ultra-fast, fast or slow extinction. Slow extinction is indicative of metastable behaviour of the process, and in this case our interest focuses on the exponent of decay of the metastable density when $\lambda\downarrow0$. More precisely, just like in \cite{JLM19}, we let
$$I_N(t)=\frac1N \, \E\Big[ \sum_{i=1}^N X_t(i)\Big] 
= \frac1N \sum_{i=1}^N \P_i \big( T_{\rm ext}>t\big),$$
where $\P_i$ refers to the process started with only vertex~$i$ infected and the last equality holds by the
self-duality of the process. We say the contact process features
\begin{itemize}
\item \emph{metastability} if there there exists $\eps>0$ such that
\smallskip
\begin{itemize}[leftmargin=*]
\item whenever $t_N$ is going to infinity slower than~$e^{\eps N}$, we have
$\displaystyle\liminf_{N\to\infty} I_N(t_N)>0.$
\item whenever $(s_N)$ and $(t_N)$ are two sequences going to infinity slower than~$e^{\eps N}$, we have 
$\displaystyle \lim_{N\to\infty} I_N(s_N)-I_N(t_N) = 0.$
\end{itemize}\smallskip
In that case, we can unambiguously define the \emph{lower metastable density} $\displaystyle\rho^-(\lambda)=\liminf_{N\to\infty} I_N(t_N)>0$ and the \emph{upper metastable density} $\displaystyle\rho^+(\lambda)  =\limsup_{N\to\infty} I_N(t_N)$.\\[-2mm]
\item a \emph{metastable exponent} $\xi$ if, for sufficiently small $\lambda>0$ there is metastability and 
$$\xi= \lim_{\lambda\downarrow 0}\frac{\log \rho^-(\lambda)}{\log \lambda} = \lim_{\lambda\downarrow 0}\frac{\log \rho^+(\lambda)}{\log \lambda}.$$
\end{itemize}

We are now ready to state our main result for the featured network kernels, the factor and preferential attachment kernels defined after~\eqref{condp}.

\pagebreak[3]

\begin{theorem}
\label{teofinal}
Consider the Markov process $({\mathscr G}^{\ssup N}_t, X_t \colon t\geq 0)$ describing the simultaneous evolution of a network $({\mathscr G}^{\ssup N}_t \colon t\geq 0)$ with power law exponent $\tau>2$ and update speed $\eta\in\R$ and of the contact process $(X_t \colon t\geq 0)$ on it. 
\begin{itemize}

\item[(a)]
Consider the \textbf{factor kernel}. 
\begin{itemize}
\item[(i)] If $\eta > \frac 12$ and $\tau>3$, there is ultra-fast extinction.
\item[(ii)] If $\eta\le 0$ and $\tau>4-\eta$, or if $0\le \eta\le \frac12$ and $\tau>4- 2\eta$,  there is fast extinction.
\item[(iii)]
If $\eta\le 0$ and $\tau<4-\eta$, or if $0\le \eta\le \frac12$ and $\tau<4- 2\eta$, or if $\eta\ge \frac 12$ and $\tau<3$, there is slow extinction and metastability. Moreover, the metastability exponent satisfies
\begin{equation}
\label{dens1}
\xi \:=\;\left\{\begin{array}{ccl}\frac1 {3-\tau} &\mbox{ if }&\; 
\left\{\begin{array}{rl}
&\eta\le 0 \mbox{ and }\tau\le\frac 5 2,\\
\mbox{ or }&0\le \eta\le \frac12\mbox{ and }\tau\le\frac 5 2+\eta, \\
\mbox{ or }&\eta\ge \frac 12\mbox{ and }\tau<3,
\end{array}\right.
\\[-2mm]
\\2\tau - 3&\mbox{ if }&\;
\left\{\begin{array}{rl}
&\eta\le \frac 5 2 - \tau \mbox{ and }\frac 5 2\le \tau \le 3,\\
\mbox{ or }& \eta\le 2-\tau \mbox{ and }\tau\ge 3,
\end{array}\right.
\\[-2mm]
\\\frac {2\tau-2-\eta}{4-\eta-\tau} &\mbox{ if }&\;\eta\le 0 \mbox{ and }\frac 52-\eta\le \tau<4+2\eta,\\
[-2mm]\\\frac {2\tau-2-2\eta}{4-2\eta-\tau} &\mbox{ if }&\;0\le \eta\le \frac 12 \mbox{ and } \frac 52+\eta\le \tau<4-2\eta
,\\[-2mm]
\\
\frac \tau {4-\tau} &\mbox{ if }&\; 3 \vee (4+2\eta) \le \tau \le 4+\eta,
\\[-2mm]
\\
\frac {3\tau - 4 -\eta}{4-\tau-\eta} &\mbox{ if }&\;(2-\eta)\vee (4+\eta)\le \tau <4-\eta.
\end{array}\right.
\end{equation}\end{itemize}
\medskip

\pagebreak[3]

\item[(b)] 
Consider the \textbf{preferential attachment kernel}.
\begin{itemize}
\item[(i)] 
If $\eta\ge \frac12$ and $\tau>3$, there is ultra-fast extinction.
\item[(ii)]
If $\eta<\frac12$, or if $\eta\ge \frac12$ and $\tau<3$, there is slow extinction and metastability, and the metastability exponent satisfies
\pagebreak[3]

\begin{equation}
\label{dens2}
\xi\:=\;\left\{\begin{array}{ccl}
2\tau-3&\mbox{ if }&
\left\{\begin{array}{rl}
&\eta\le -\frac 12,\\
\mbox{ or }&-\frac 12 \le \eta\le 0 \mbox{ and }\tau\le 2-\eta,
\end{array}\right.\\
[-2mm]\\\frac{3\tau-4-\eta}{4-\eta-\tau}&\mbox{ if }& -\frac12<\eta\le0\ \mbox{ and }\ 2-\eta\le \tau\le \frac 83+\frac \eta 3, \\
[-2mm]\\ \frac{3\tau-5-\eta}{1-\eta}&\mbox{ if }& -\frac12<\eta\le0\ \mbox{ and }\ \tau\ge \frac 83+\frac \eta 3, \\
[-2mm]\\ \frac{3\tau-4-2\eta}{4-2\eta-\tau}&\mbox{ if }& 0\le \eta\le\frac 12\ \mbox{ and }\ 2+2\eta\le \tau\le \frac 83+\frac {2\eta} 3, \\
[-2mm]\\ \frac{3\tau-5-2\eta}{1-2\eta}&\mbox{ if }& 0\le \eta\le\frac 12\ \mbox{ and }\ \tau\ge \frac 83+\frac {2\eta} 3, \\
[-2mm]\\ \frac {\tau-1} {3-\tau} &\mbox{ if }& 
\tau<3  \mbox{ and }\eta>\frac \tau 2 - 1.
 \end{array}\right.
\end{equation}
\end{itemize}
\end{itemize}
\end{theorem}
\medskip

\begin{figure}[h!]
{\includegraphics[height=6.5cm]{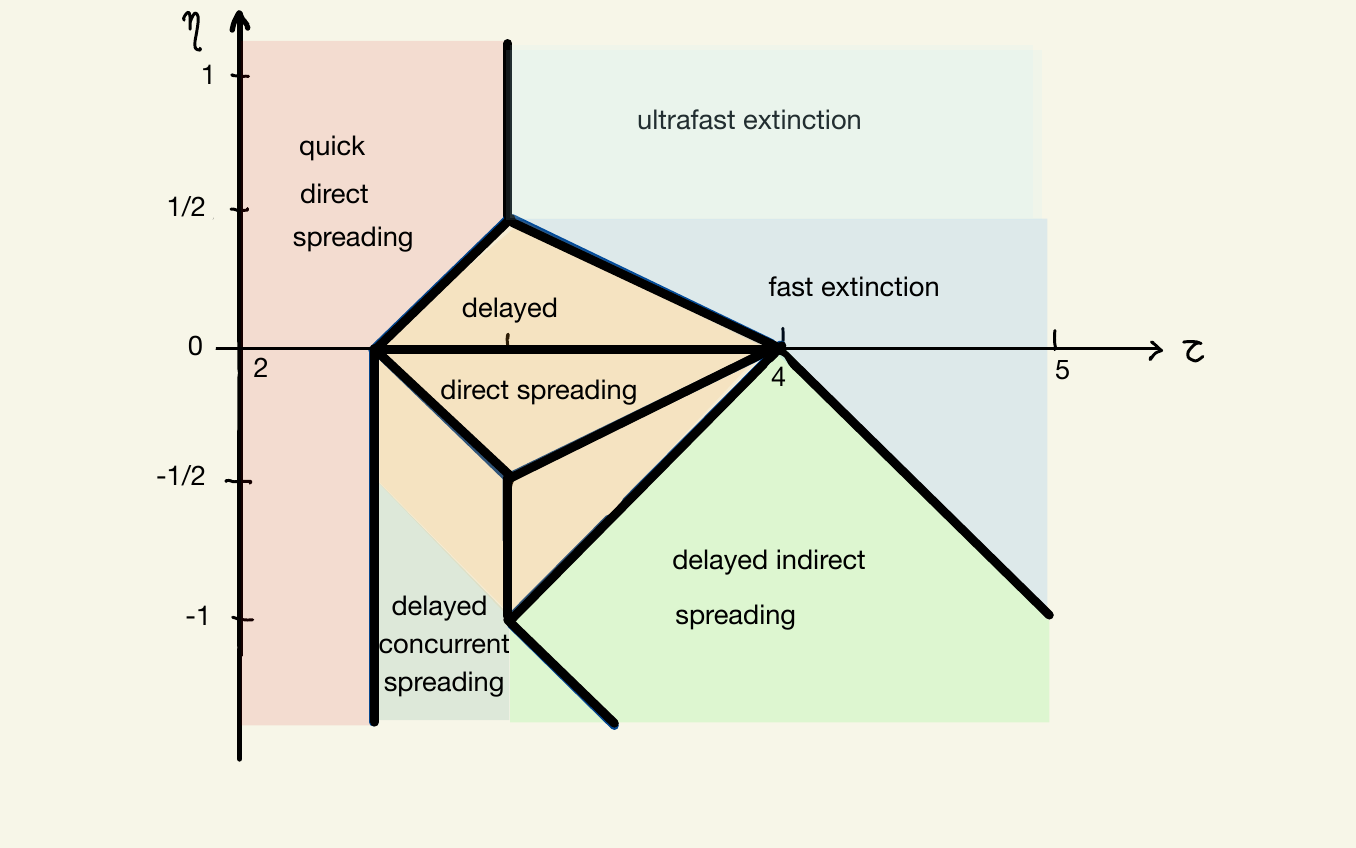}}
    {\includegraphics[height=6.5cm]{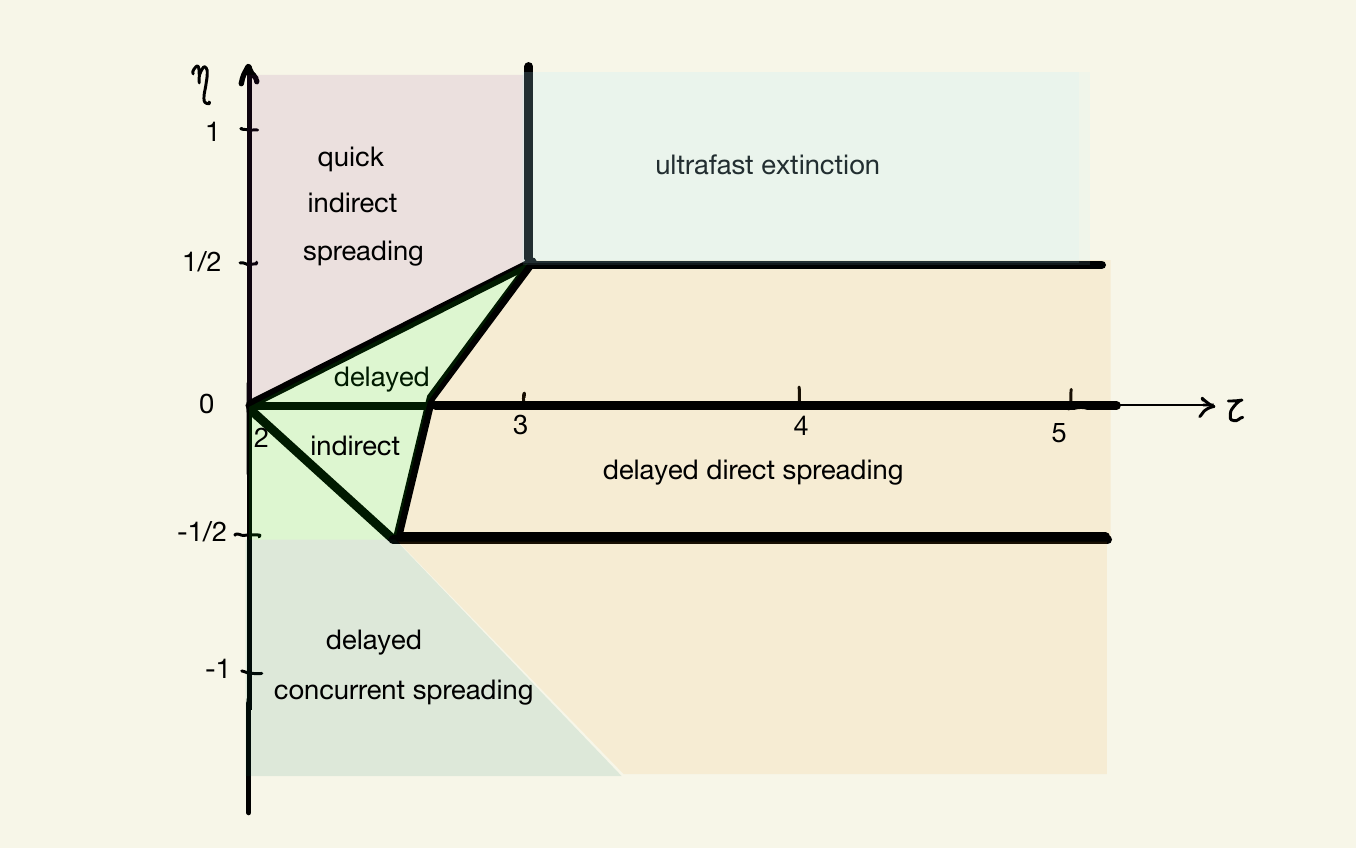}}
\caption{Phase diagram summarizing Theorem~\ref{teofinal} for factor kernel (top) and preferential attachment kernel (bottom). Bold lines indicate first order phase transitions, i.e. parameters where the derivative of $\xi$ is discontinuous. The colour code refers to the dominating strategy. Note that a crossover between strategies can occur without a phase transition, and a phase transition without a crossover between strategies.}
\label{fig1}
\end{figure}
\pagebreak[3]

\begin{remark}
The restriction of this theorem to $\eta\ge 0$ is the content of Theorem 3 of \cite{JLM19}, although the distinction of fast and ultra-fast extinction has not been made explicit there. 
\end{remark}

\begin{remark}
Theorem~\ref{teofinal} is illustrated in Figure~\ref{fig1} showing the phase boundaries where the metastable exponents are not differentiable by bold lines. In both models they divide the slow extinction region into six different phases. The parametrisation with $(\tau, \eta)$ leads to phase boundaries which are line segments, which is not the case if $(\gamma, \eta)$ is used as a parametrisation as in \cite[Figure~1]{JLM19}. Figure~\ref{fig1} also uses colours to show for which of the survival strategies the upper bounds for the metastable exponent matches the lower bounds. It is interesting to see that a crossover between strategies can occur without a phase transition, and a phase transition without a crossover between strategies.
\end{remark}

\begin{remark}
Theorem~\ref{teofinal} shows that our model interpolates nontrivially between the static case, when $\eta\to-\infty$ and the mean-field case, when $\eta\to+\infty$. Metastable exponents have been found for the static case for a model with factor kernel in~\cite{MVY13} and with preferential attachment kernel in~\cite{VHC17}, and in the mean-field case implicitly in~\cite{PV01}. 
\end{remark}

The rest of this paper is structured as follows. The most innovative results, which provide bounds under general assumptions on the kernel~$p$, are stated as Theorems~\ref{theoslow}, \ref{teoupper_vertex_improved}~and~\ref{teoupper_optimal}.
In Section~\ref{sec:lower_bounds} we prove the lower, and in Section~\ref{sec_upper_overall}
the upper bounds. The full proof of Theorem~\ref{teoupper_optimal}, which introduces a new `hybrid' proof technique, is given in Section~\ref{sec_upper_new}.
We state the major intermediate steps in the proofs and 
results which are similar to results in the papers~\cite{JLM19, JLM22} as propositions. In the latter case we  keep the arguments brief.
\pagebreak[3]

\subsubsection*{Notation}
This work involves many constants 
whose precise value is not too important to us, but with nontrivial dependencies between each other. For them 
we often use the following notational convention. The constant is written with a small $c$ (resp. a capital $C$) if it is a ``small positive constant'' (resp. a ``large constant''), typically introduced with a condition requiring it to be small (resp. large). 
Then the index indicates the number of the equation or theorem where the constant is introduced, so $c_{\eqref{lemma1r}}$ is introduced in $\eqref{lemma1r}$ and  $C_{\ref{theoslow}}$ in Theorem~\ref{theoslow}.
\pagebreak[3]

\section{Survival and lower bounds for the metastable density}\label{sec:lower_bounds}

{Our aim in this section is to provide sufficient conditions for the survival of the contact process, while at the same time giving lower bounds on the metastable density.}
	As in~\cite{JLM19,JLM22} the proofs for survival and the lower metastable densities rely on a graphical construction of the process which can be found in \cite{JLM19} but which we summarize here for the sake of completeness. The evolving network $({\mathscr G}^{\ssup N}_t \colon t\geq 0, N\in \N)$ is represented 
with the help of the following independent random variables;
\begin{enumerate}
	\item[(1)] For each vertex $i\in\{1,\ldots,N\}$, a Poisson point process \smash{$\mathcal U^i:=(U^i_n)_{n\ge 1}$} of intensity $\kappa_i$, describing the updating times of the vertex~$i$.
	\smallskip
	\item[(2)] For each $\{i,j\}\subset\{1,\ldots,N\}$, $i\neq j$, a sequence $\mathcal{C}^{ij}:=(C^{ij}_n)_{n\ge0}$ of i.i.d. Bernoulli random variables with parameter $p_{ij}$, 
	where \smash{$C^{ij}_n=1$} if $\{i,j\}$ is an edge of the network between the $n$-th and $(n+1)$-th events in $\U^{ij}:=\U^i\cup \U^j$. 
\end{enumerate}
Given the network we represent the infection by means of the following set of
independent random variables,
\begin{enumerate}
	\item[(3)] For each $i\in \{1,\ldots,N\}$, a Poisson point process $\mathcal R^i=(R^i_n)_{n\ge 1}$ of intensity one describing the recovery times of~$i$.\smallskip
	\item[(4)] For each $\{i,j\}\subset\{1,\ldots,N\}$ with $i\not=j$, a Poisson point process $\mathcal {\mathcal I}_0^{ij}$ with intensity~$\lambda$ describing the infection  times along the edge $\{i,j\}$. Only the trace $\mathcal I^{ij}$ of this process on the~set
	\[\bigcup_{n=0}^\infty \{[U^{ij}_n, U^{ij}_{n+1}) \colon  C^{ij}_n=1\} \subset [0,\infty)\]
	can actually cause infections. If just before an event in $\mathcal I^{ij}$ one of the involved vertices is infected and the other is healthy, then the healthy one becomes infected. Otherwise, nothing happens.\smallskip
\end{enumerate}

Throughout this section we slightly change both the network and the parameters of the model as was done in {Section 3.1 in~\cite{JLM22}} in order to both simplify computations and formulate some of the results which will be used here.
	\smallskip
	
	Given a fixed $a=a(\lambda)\in (0,1/2)$ we reduce the vertex set to~${\!\St}\cup{\Co}^0\cup{\Co}^1\cup\Vo^{odd}$  where 
\begin{align*}
{\!\St}&:=\,\{i\in\{\lceil\sfrac{aN}{2}\rceil+1,\ldots, \lceil aN\rceil\}\colon\,i\text{ is even }\},\\[2pt]{\Co^0}&:=\,\{j\in\{\lceil \sfrac{N}{2}\rceil+1,\ldots, N\}\colon\,j=4k\text{ for some }k\in\N \},\\[2pt]{\Co^1}&:=\,\{j\in\{\{\lceil \sfrac{N}{2}\rceil+1,\ldots, N\}\colon\,j=4k+2\text{ for some }k\in\N \},
\\[2pt]{\Vo^{odd}}&:=\,\{i\in\{1,\ldots, N\}\colon\,i\text{ is odd }\}.
\end{align*}%
The vertices in $\St$ correspond to \emph{stars}, which are the key ingredients in the survival strategies. The vertices in $\Co=\Co^0\cup\Co^1$ correspond to \emph{connectors} (low power vertices), which are partitioned into $\Co^0$ and $\Co^1$ depending on whether we use them to survive locally or spread the infection, and vertices in $\Vo^{odd}$ will be used to provide lower bounds for the metastable density. 
{Observe that by assigning specific roles to vertices as above we neglect certain infection events and therefore obtain lower bounds for the probability of increasing events with respect to the process $(X_t \colon t\ge0)$. As for the edges of the dynamical network we replace the connection probabilities $p_{ij}$ by $\tilde{p}_{ij}$ where
\[
\widetilde{p}_{ij}\;=\left\{\begin{array}{cl}p_{\lfloor aN\rfloor,\lfloor aN\rfloor}&\text{ if }i,j\in{\St}\\p_{\lfloor aN\rfloor,N}&\text{ if }i\in{\St},j\in{\Co}
\text{ or }i\in{\Co},j\in{\St}\\0&\text{ if }i,j\in{\Co}.\end{array}\right.
\]
The idea behind simplifying the connection probabilities is to work on a model that uses only vertex quantity and not identity. Even though it still remains that the updating  rates~$\kappa$ are different for each edge, this idea becomes heuristically correct since $\kappa_i=\Theta(a^{-\gamma\eta})$ uniformly over all stars~$i$, and $\kappa_j=\Theta(1)$ uniformly over all connectors~$j$. Observe that we can construct both \smash{$\widetilde{{\mathscr G}}_t$} and  \smash{${\mathscr G}_t$} so that  \smash{$\widetilde{{\mathscr G}}_t\subseteq{\mathscr G}_t$} and hence the original process dominates the one running on the subgraph. In order to ease notation and computations we will henceforth replace $\tilde{p}_{ij}$ by $\frac{1}{N}p(a,1)$ or $\frac{1}{N}p(a,a)$ depending on the case. 
\medskip

Following the structure given in \cite{JLM22} we state the following lemma which provides a lower bound for the lower metastable density whenever metastability holds. We do not provide the proof since it is the same as in our previous work.
\begin{lemma}[\cite{JLM22}, Lemma 3.1]\label{lemmalower}
	Define  
	\smash{$(\mathfrak{F}^{_{0,1}}_{t})_{t\geq0}$} as the filtration given by all the $\Rec^{i}$, $\I_0^{ij}$, and $\U^{ij}$ up to time $t$ where $i,j\notin\Vo^{odd}$, as well as all the connections between such vertices up to $t$. For any $r>0$ and $t>0$  there is $c_{\eqref{lemma1r}}>0$ (independent of $\lambda$, $a$, $N$) such that
	\begin{equation}
	\label{lemma1r}
	\E\Big[\sum_{j\in\Vo^{odd}} X_{t+1}(j) \,\Big|\,\mathfrak{F}_t^{0,1}\, \Big] \;\geq\; c_{\eqref{lemma1r}} \,\Big(\int_0^1(\lambda a p(a,x)\wedge1) \, \mathrm d x\Big) N,
	\end{equation}
	on any event $A\in \mathfrak{F}_t^{0,1}$ implying $\big|\{i\in \St \colon X_t(i)=1\}\big|\,\geq\,raN$. 
\end{lemma}

The following result gives sufficient conditions for  \textit{quick direct spreading} and \textit{quick indirect spreading} to hold, leading to slow extinction. These survival strategies do not rely on local survival of the infection on powerful vertices. They were already studied in \cite{JLM19,JLM22} and the proof in
 \cite{JLM22} can be adapted from  the model with edge updating to the current setting the only relevant difference being  that the definition of \emph{$T_0$-stable} vertices should now include the requirement that vertices do not update within~$T_0$. 
\pagebreak[2]

\begin{proposition}[\cite{JLM22}, Theorem 3.2]\label{teolower} 
	There exist positive $M_{(i)}$ and $M_{(ii)}$ 
	(depending on~$\kappa_0$) such that slow extinction and metastability hold 
	for the contact process on the network if, for all $\lambda \in (0,1)$,
	there is $a=a(\lambda) \in (0,1/2)$ satisfying at least one of the following conditions:
	\medskip
		\begin{itemize}
		\item[(i)] {\bf (Quick direct spreading)}
		$\displaystyle \lambda ap(a,a)>M_{(i)}.$\\[-2mm]
		\item[(ii)] {\bf (Quick indirect spreading)}
		$\displaystyle \lambda^2a p^2(a,1)>M_{(ii)}$.\\[-2mm]
	\end{itemize}
	Moreover, in each of these cases we have 
	\begin{equation}
	\label{lowdensityetapos}
	\rho^-(\lambda)\;\geq\;c_{\eqref{lowdensityetapos}} \lambda a\int_0^1p(a,s)\, \mathrm ds,
	\end{equation}
	where $c_{\eqref{lowdensityetapos}}>0$ is a constant independent of $a$ and $\lambda$.
\end{proposition}
\begin{remark}\label{rem:direct-mechanims}
	We have included in the statement that under the conditions of the theorem metastability holds, which is not part of the original result but which follows from the arguments in Section~3.5 of \cite{JLM22} and the observation that on the slow extinction event the proportion of infected stars never goes below some $\rho(\lambda)>0$. Moreover, by condition \eqref{condp} we have \smash{$\int_0^1p(a,s)\, \mathrm ds\geq {\mathfrak c}_1 a^{-\gamma}$} so the lower bound in \eqref{lowdensityetapos} is always of the form 
	$c\lambda a^{1-\gamma}$.
\end{remark}\pagebreak[3]

We now turn to \emph{delayed direct spreading} and \emph{delayed indirect spreading}, which spreads as quick direct (resp.\ indirect) spreading while also relying on local survival of the infection around powerful vertices. This mechanism was also studied in \cite{JLM19}, but as discussed in the introduction, local survival in the case $\eta<0$ features a weaker degradation effect as compared to the case $\eta>0$, while the spread of the infection is subject to the additional depletion effect. The main purpose of this section is to deal correctly with these new features so as to obtain the following theorem.

\begin{theorem}\label{theoslow}
	Fix $\eta<0$, and define 	\begin{equation}\label{deftloc}\TL\;:=\; \TL(a,\lambda)\;:=\;
\lambda^2p(a,1)\kappa(a)^{-1}.
	\end{equation}
	There are constants $C_{\ref{theoslow}},M_{(iii)},M_{(iv)}>0$ large and {$c_{\ref{theoslow}}>0$} small, {depending on $\gamma$, $\eta$, $\kappa_0$, $\cl$ and $\cu$ alone}, 
	such that slow extinction and metastability hold for the contact process if, for all  $\lambda\in (0,c_{\ref{theoslow}})$, there is $a=a(\lambda)\in(0,1)$ satisfying the following two technical conditions
	\begin{itemize}
		\item[(i)] $p(a,1)\leq e^{1/\lambda}$
		\item[(ii)] $\lambda^2p(a,1)>-C_{\ref{theoslow}}\log(a\lambda)$
	\end{itemize}
	as well as at least one of the following two conditions:
	\begin{itemize}
		\item[(iii)]\ {\bf (Delayed direct spreading)}\ 
		$\displaystyle (\kappa(a)\wedge\lambda) ap(a,a)\TL>M_{(iii)}.$
		
		\item[(iv)]\ {\bf (Delayed indirect spreading)}\ 
		$\displaystyle\lambda^2 a p^2(a,1)\TL>M_{(iv)}.$
	\end{itemize}
	Moreover, in each of these cases we have 
	\begin{equation}
		\label{lowdensityetaposdelayed}
		\rho^-(\lambda)\;\geq\;c_{\eqref{lowdensityetaposdelayed}}\lambda a \int_0^1p(a,s)\, \mathrm ds,,
	\end{equation}
	where $c_{\eqref{lowdensityetaposdelayed}}>0$ is a constant independent of $a$ and $\lambda$.
\end{theorem}

\begin{remark} Again we do not address here the metastability, which follows from  a simple adaptation of the arguments given in Section~3.5 of \cite{JLM22}, and again the lower bound in~\eqref{lowdensityetaposdelayed} is always of the form $c \lambda a^{1-\gamma}$.
\end{remark}

\begin{remark}
	The condition
	\[\lambda^2 ap^2(a,1)\TL>M_{(iv)}\]
	required for delayed indirect spreading is similar to the one for quick indirect spreading, except for the $\TL$ factor which allows for the infection to `wait' around stars until spreading is possible. For delayed direct spreading, on the other hand, the condition reads
	\[(\lambda\wedge \kappa(a)) ap(a,a)\TL>M_{(iii)}\]
	so the same would be true if not for the factor $\lambda\wedge \kappa(a)$ replacing $\lambda$. This is due to the \emph{depletion} effect described in the introduction. In~\cite{JLM19} we have similar conditions for delayed direct/indirect spreading in the case $\eta\geq0$, but without the depletion effect, and with a different expression for $\TL$.
\end{remark}

\begin{remark}
	In order to address the right order of the local survival time we need to differentiate between short and long updating times. The short updating times will be handled with the use of Proposition~\ref{survshort} where we obtain lower bounds containing a logarithmic term, while long updating times will be handled with the use of Proposition \ref{survlong} which gives lower bounds without any such term. For $\eta<0$ most updating times are long, and Proposition~\ref{survlong} then gives rise to the definition of $\TL$ without logarithmic correction. We also discuss the case $\eta\geq0$ where we obtain a local survival time with a logarithmic correction (coming from typical updating times being short), which can be shown to replace the $\log^2(\lambda a)$ correction appearing in the local survival time in \cite{JLM19}, thus improving it.
\end{remark}

\begin{remark}
	In our proofs the spreading mechanisms are similar to the ones used in the fast evolving networks, which may come as a surprise since on these proofs the high rate of change in the graph played an important role in order to provide some degree of independence between infection events, while in this case, the slow evolving networks are closer in spirit to the static ones, for which there is a high level of dependency. The answer to this apparent contradiction is that even though updating times become large when $a\downarrow 0$, the local survival around stars become much larger, thus allowing for vertices to update many times on each unit interval, giving enough independence for our methods to work.
\end{remark}

The rest of this section is divided as follows: In Section \ref{localsurvival} we give a technical definition of the local survival event and show that upon reaching a given star its probability is bounded from below. In Sections \ref{DDS} and \ref{DIS} we show how the conditions given in Theorem~\ref{theoslow} give slow extinction of the process.
\pagebreak[3]

\subsection{Local survival}\label{localsurvival}

For static networks, the local survival mechanism has been studied many times, see for example~\cite{BB+, CD09, GMT05}, 
 and is generally understood as the event in which a highly connected vertex infects neighbours which in turn infect it back, thus creating a persistent loop which eventually breaks down after some event of small probability. In this subsection we look 
 at the local survival mechanism for evolving networks, if no assumption is specified results hold for arbitrary update parameters~$\eta\in\R$.\smallskip

 Since the local survival mechanism will be used  for different stars  within the proof of slow extinction, we want to avoid the undesirable dependencies amongst them which arise from sharing common neighbours. To do so fix a sequence $(J_{k})_{k\ge 1}$ of time intervals with $J_0=[0,4)$, $J_1=[0,5)$ and for $k\geq 2$, $J_{k}\,:=\,[(k-2),(k+4))$ and define for each $k\in\N$ the set of stable connectors
\[\Co_{k}:=\,\{j\in\Co^0 \colon (\Rec^j\cup\U^j)\cap J_k=\emptyset\}\]
which can be seen as the set of useful connectors for local survival across all stars. Notice that, as mentioned before, we construct $\Co_k$ using only connectors in $\Co^0$ since the ones in~$\Co^1$ will be used to spread the infection. Now, for any given star $i\in\St$ and $t\geq0$ define
\[\Co_{t,i}:=\,\{j\in\Co_{\lceil t\rceil} \colon \{i,j\}\in\mathscr G_t\}\]
which by definition of $\Co_{\lceil t\rceil}$ are neighbours of $i$ that do not update nor recover on $[t,t+4]$ (this will be relevant throughout the following sections). Using these sets we now define the processes $\bar{X}^i_t$ as in \cite{JLM22}.
\begin{definition}\label{modx}
	Fix $i\in\St$. For any realisation of the graphical construction, the process $(\bar{X}^i_t \colon t\ge0)$, is defined analogously to $(X_t\colon t\ge0)$ with the only exception that
	\begin{itemize}
		\item $\bar{X}^i_0= X_0$ on $\Co_{0,i}\cup\{i\}$ and $\bar{X}^i_0=0$ everywhere else.
		\item An infection event $s\in\I^{hj}$ is only valid if
		\[{h=i\text{ and }}j\in\Co_{r,i}\text{ for some }r\in[s-4,s+1].\]
	\end{itemize}
\end{definition}
Observe that for a given $i\in\St$ the process $(\bar{X}^i_t)$ defined above depends only on the processes \smash{$\{\I^{ij}\}_{j\in\Co^0}$, $\Rec^i$, $\U^i$, $\{\Rec^{j}\}_{j\in\Co^0}$}, and \smash{$\{\U^{j}\}_{j\in\Co^0}$}, as well as the respective connections $\{\mathcal{C}^{ij}\}_{j\in\Co^0}$. Since both the infections and the connections are defined on the edges of the graph, it follows that we can make local survival events independent across different stars as soon as we work on a fixed realisation of the $\{\Rec^{j}\}_{j\in\Co^0}$ and $\{\U^{j}\}_{j\in\Co^0}$ processes. {With this in mind observe that for any fixed $k$, connectors in $\Co^0$ belong to $\Co_k$ independently with probability $e^{-(\kappa(j)+1)|J_k|}$ which is lower bounded by a positive function of $\kappa_0$ and $\eta$}. It follows then from a large deviation argument that there is a constant $c_{\eqref{localstable}}>0$ {depending on $\kappa_0$ and $\eta$ alone} such that {for any $N$ large}
\begin{equation}\label{localstable}
\P(|\Co_{k}|> c_{\eqref{localstable}}N \, \mbox{for all } k\in\{0,\ldots,\lfloor e^{c_{\eqref{localstable}} N}\rfloor\})\,\geq\,1-e^{-c_{\eqref{localstable}} N},
\end{equation}
where we call $\mathcal{A}_0$ the event on the left. Since the lower bound on the right is already of the form required for slow extinction we will work on a fixed realisation of the \smash{$\{\Rec^{j}\}_{j\in\Co^0}$} and \smash{$\{\U^{j}\}_{j\in\Co^0}$} processes such that $\mathcal{A}_0$ holds. Now that we have addressed independence we can start constructing the local survival event for a given star $i\in\St$, which consists of `surviving' the updating events in $\U^i$ for a long time. Fix $s\geq0$ and let $U(s)$ be the first event in $\U^i\cap(s,\infty)$. Define the event
\[\mathcal{S}^{short}_s\;:=\;\{U(s)-s\leq 4\text{ and }\bar{X}^i_{U(s)}(i)=1\},\]
which can be thought as surviving an updating event that arrives quickly after time $s$. Our aim is to show that if at time $s$ the central node $i$ is infected and it has sufficiently many stable neighbours, then such survival is very likely to occur. To do so define for each $s\geq0$ and $i\in\St$ the $\sigma$-algebra $\mathfrak{F}^{i}_{s}$ given by the graphical construction throughout $[0,s]$ restricted to the subgraph generated by $\{i\}\cup\Co^0$. Also denote by $\mathfrak{F}$ the $\sigma$-algebra generated by the processes $\{\Rec^{j},\U^{j}\}_{j\in\Co^0}$.
\begin{proposition}\label{survshort}
	Fix $i\in\St$, $s\geq0$, and {$\lambda\leq\frac{1}{4}$}. 
	Then $\mathcal{S}^{short}_s$ is 
	\mbox{$\mathfrak{F}^i_{U(s)}$-measurable}, and for any $0<c<1$,
	\begin{equation}\label{propshort}\P(\mathcal{S}^{short}_s\,\big|\,\mathfrak{F}^i_s,\mathfrak{F})\geq1-\frac{C_{\eqref{propshort}}\kappa(a)}{c\lambda^2p(a,1)}\log\left(1+\frac{\lambda^2cp(a,1)}{(1+\kappa(a))^2}\right)-e^{-4\kappa(a)}\end{equation}
	on the event
	$\mathcal{A}_0\cap\{|\Co_{s,i}|>cp(a,1)\mbox{ and } \bar{X}^i_s(i)=1\}$, where {$C_{\eqref{propshort}}=20(1+\kappa_0+\kappa_0^{-1})^2$}. 
\end{proposition}

\begin{proof}
Abusing notation for the sake of clarity, we will assume that $\P$ stands for the law of the process conditioned on $\mathfrak{F}^i_s$ and $\mathfrak{F}$. The fact that \smash{$\mathcal{S}^{short}_s$} is \smash{$\mathfrak{F}^i_{U(s)}$}-measurable follows directly from the definition of \smash{$\mathcal{S}^{short}_s$} so we only prove the inequality. \smallskip

Let $(E_n)_{n\ge 1}$ be the ordered elements of $(\Rec^i\cup\U^i)\cap(s,\infty)$ so the sequence $(E_{n+1}-E_n)_{n\ge 1}$ are i.i.d. exponential random variables with rate $1+\kappa(a)$, and each $E_n$ is independently assigned to be a recovery event with probability $1/(1+\kappa(a))$ or an updating event otherwise. Note that the first updating event occurs at $U(s)=E_M$ where $M$ is a geometric random variable with parameter $\kappa(a)/(1+\kappa(a))$ independent of $(E_n)_{n\ge 1}$. Our analysis of $\P(\mathcal{S}^{short}_s)$ will depend on $M$ as follows: On the event $\{M=1\}$ there are no recovery events between~$s$ and $U(s)$ so that 
\begin{equation}\label{firstS}\P(\mathcal{S}^{short}_s\cap \{M=1\})=\P(U(s)-s\leq4\,\mbox{ and } M=1).\end{equation}
On the event $\{M\ge 2\}$, we write $T_l=E_1-s$ and $T_r=E_M - E_{M-1}$ to denote the length of the intervals before the first recovery and after the last. Now, since $|\Co_{s,i}|>c p(a,1)$ we know that $i$ has at least $c p(a,1)$ neighbours in $\Co^0$ which do not update or recover throughout $[s,s+4]$ so on the event $\{U(s)-s\leq 4\}$ it is enough for 
\smash{$\bar{X}^i_{U(s)}(i)=1$} to occur, to find some such neighbour that gets infected in $[s,E_1]$ and infects $i$ back in $[E_{M-1},E_M]$. Conditionally on $(E_n)_{n\ge 1}$ this occurs with probability
\[
1-\big(1-(1-e^{-\lambda T_l}) (1-e^{-\lambda T_r})\big) ^{|\Co_{s,i}|}\geq 1-(1-\lambda^2 T_l T_r/4)^{|\Co_{s,i}|} \ge 1-e^{-\lambda^2 c p(a,1) T_l T_r/4}
\]
on the event $\{U(s)-s\leq4 \mbox{ and } M\ge 2\}$, since these events are independent for different neighbours, and where we have used that both $T_l$ and $T_r$ are bounded by $4$ and $\lambda\leq\frac{1}{4}$ to apply the inequality $1-e^{-x}\geq x/2$. Taking expectations with respect to $(E_n)_{n\ge 1}$ and adding the result to \eqref{firstS} we obtain
\begin{equation} \label{secondS}
\P(\mathcal{S}^{short}_s)\ge \P(U(s)-s\leq 4)-\E[ e^{-\lambda^2 c p(a,1) T_l T_r/4} \1_{\{U(s)-s\leq4\mbox{ and } M\ge 2\} }].
\end{equation}
If $\eta\geq0$, then the event $\{U(s)-s\leq4\}$ is likely to occur and hence we do not lose much by further bounding \eqref{secondS} as
\[\P(\mathcal{S}^{short}_s)\ge \P(U(s)-s\leq 4)-\E[ e^{-\lambda^2 c p(a,1) T_l T_r/4} \1_{\{M\ge 2\} }]\]
but conditionally on $\{M\geq2\}$ the random variables $T_l$ and $T_r$ are just independent exponential random variables with parameter $1+\kappa_a$. Since for i.i.d.\ exponential random variables $Y,Y'$ with rate $1$ and $t\geq0$ we have
\begin{align}\label{auxY}
\E[e^{-t Y Y'}]&= \E[\frac 1 {1+ t Y}]= \int_0^{+\infty} \frac {e^{-s}}{1+st} \mathrm d s =\frac {e^{1/t}}t \int_{1/t}^{+\infty} \frac {e^{-s}}s \mathrm ds\le \frac {\log(1+t)}t
\end{align}
we conclude that for $\eta\geq0$ we have
\begin{align*}
\P(\mathcal{S}^{short}_s)&\ge \P(U(s)-s\leq 4)-\P(M\ge 2)\tfrac {4 (1+\kappa(a))^2}{\lambda^2 c p(a,1)} \log\left( 1 +\tfrac {\lambda^2 c p(a,1)}{4(1+\kappa(a))^2}\right)\\[1ex]&\ge 1-e^{-4\kappa(a)}-\tfrac {4 (1+\kappa(a))}{\lambda^2 c p(a,1)} \log\left( 1 +\tfrac {\lambda^2 c p(a,1)}{4(1+\kappa(a))^2}\right),
\end{align*}
where we used that \smash{$\P(M\ge2)=\frac{1}{1+\kappa(a)}$}. Inequality \eqref{propshort} follows as {$\eta\geq0$ and $a<1$ imply $4(1+\kappa(a))\leq 4(1+\kappa_0^{-1})\kappa(a)\leq C_{\eqref{propshort}}\kappa(a)$.} 
\pagebreak[3]\smallskip

Suppose now that $\eta<0$ so that $\{U(s)-s\leq 4\}$ is no longer a likely event. In this scenario we observe that
\[\{M\ge 2 \mbox{ and } U(s)-s\leq 4\}\subseteq\{M= 2\}\cup \{M\ge3 \mbox{ and } E_{M-1}-E_1\leq 4\}\]
and that both events $\{M=2\}$ and $\{M\ge 3\}\cap\{E_{M-1}- E_1\le 4\}$ are independent from $T_l, T_r$ conditionally on $M\ge 2$. To compute the probability of this event observe that $E_{M-1}-E_1$ corresponds to the sum of $M-2$ exponential random variables with rate $1+\kappa(a)$, and since conditionally on $\{M\geq3\}$, the random variable $M-2$ is  geometric with rate $\kappa(a)/(1+\kappa(a))$, it follows that $E_{M-1}-E_1$ is exponential with rate $\kappa(a)$ and hence
\[\P(M= 2 \mbox{ or } (M\ge3 \mbox{ and } E_{M-1}-E_1\leq 4))\;\le\;\tfrac{\kappa(a)}{1+\kappa(a)}+1-e^{-4\kappa(a)}\;\leq\;5\kappa(a).\]
Using the independence between the previous event and $T_l,T_r$ we obtain from \eqref{secondS} and~\eqref{auxY},
\begin{align*}
\P(\mathcal{S}^{short}_s)&\ge \P(U(s)-s\leq 4)-\E[ e^{-\lambda^2 c p(a,1) T_l T_r/4} \1_{\{M\ge 2\}\cup \{M\ge3\,\wedge\,E_{M-1}-E_1\leq 4\}}]\\[1ex]&\ge \P(U(s)-s\leq 4)-5\kappa(a)\tfrac {4 (1+\kappa(a))^2}{\lambda^2 c p(a,1)} \log\left( 1 +\tfrac {\lambda^2 c p(a,1)}{4(1+\kappa(a))^2}\right)\\[1ex]&\ge 1-\tfrac {C_{\eqref{propshort}} \kappa(a)}{\lambda^2 c p(a,1)} \log\left( 1 +\tfrac {\lambda^2 c p(a,1)}{4(1+\kappa(a))^2}\right)-e^{-4\kappa(a)},
\end{align*}
where in the last line we have used that if $\eta<0$ then {$20(1+\kappa(a))^2\leq20(1+\kappa_0)^2\leq C_{\eqref{propshort}}$}.
\end{proof}
Recall that $\mathcal{S}^{short}_s$ was thought as surviving an updating event that arrives quickly after time $s$. To address the case where updating events take a long time to occur we first define the set of \textit{infected stable neighbours} of $i$ as
\[\Co'_{t,i}\,:=\,\{y\in\Co_{t,i} \colon \bar{X}^i_{t}(y)=1\},\]
which are the stable connectors of $i$ that are infected at time $t$. {Similarly to} \cite{JLM22}, for any fixed $i\in\St$ and $t\geq0$ we introduce the `good' event $\mathcal{W}_i(t)$ as
	\begin{equation}\label{defW}\mathcal{W}_i(t)\,=\,\Big\{|\Co'_{t,i}|\geq {\tfrac{c_{\eqref{localstable}}}{80}}\lambda p(a,1)\,\text{ and  }\int_{t}^{t+1}\bar{X}^i_s(i)\, ds\geq \tfrac{1}{2}\Big\}.\end{equation}
In words, the event $\mathcal{W}_i(t)$ implies that on $[t,t+1]$ the central vertex is both infected a large proportion of the time, and has many infected neighbours. Again, for any $s\geq0$ call $U(s)$ the first event in $\U^i\cap(s,\infty)$ and define the random variable $\mathfrak{n}=\lfloor U(s)-s\rfloor$ so that $s+\mathfrak{n}\leq U(s)<s+\mathfrak{n}+1$. We can now define 
\smash{$\mathcal{S}^{long}_s$} as the event in which

\begin{enumerate}
	\item $U(s)-s\geq 4$,
	\item for each $2\leq j\leq \mathfrak{n}-2$, the event $\mathcal{W}_i(s+j)$ holds,
	\item $\bar{X}^i_{U(s)}(i)=1$.
\end{enumerate}
The following proposition shows that the event $\mathcal{S}^{long}_s$ is very likely to occur.

\begin{proposition}\label{survlong}	Fix  $i\in\St$, $s\geq0$, and $\eta<0$, and let $\mathfrak{F}^i_s$ and $\mathfrak{F}$ as defined before Proposition~\ref{survshort}. There are $0<c_{\lambda}<1/e$ small and $C_{\ref{survlong}}>0$ large depending on $c_{\eqref{localstable}}$, $\gamma$, $\eta$, $\kappa_0$, $\cl$ and $\cu$ alone, such that if $\lambda\in (0,c_\lambda)$ and $a\in(0,1)$ are chosen satisfying
	\begin{itemize}
		\item[(i)] $p(a,1)\leq e^{1/\lambda}$
		\item[(ii)] $\lambda^2p(a,1)>-C_{\ref{survlong}}\log(a\lambda)$
	\end{itemize}
	then $\mathcal{S}^{long}_s$ is $\mathfrak{F}^i_{U(s)}$-measurable and there is  $C_{\eqref{proplong}}
	>0$ depending on $c_\lambda$,  $c_{\eqref{localstable}}$, and the parameters  $\gamma$, $\eta$, $\kappa_0$, $\cl$ and $\cu$, such that, {for all $N$ sufficiently large},
		\begin{equation}\label{proplong}
	\P(\mathcal{S}^{long}_s\,\big|\,\mathfrak{F}^i_s,\mathfrak{F})\geq e^{-4\kappa(a)}-\frac{C_{\eqref{proplong}}}{\lambda^2p(a,1)}
	\end{equation}	
	on the event 	
	\begin{equation}\label{prop3}
	\mathcal{A}_1\;:=\;\mathcal{A}_0\cap\{|\Co_{s,i}|> \tfrac{c_{\eqref{localstable}}}{5} p(a,1)\mbox{ and } \bar{X}^i_s(i)=1\},
	\end{equation}
	for $\mathcal{A}_0$ as defined in \eqref{localstable}.\end{proposition}

\begin{proof}
	As in the previous result we abuse notation for the sake of clarity, and assume that $\P$ stands for the conditional law given \smash{$\mathfrak{F}^i_s$} and $\mathfrak{F}$. The fact that \smash{$\mathcal{S}^{long}_s$} is \smash{$\mathfrak{F}^i_{U(s)}$}-measurable follows from the definition of \smash{$\mathcal{S}^{long}_s$} so we only need to prove \eqref{proplong}.%
	\smallskip
	
	We first condition on $U(s)$ and assume that $\{U(s)-s\geq4\}$. Since $\bar{X}$ is defined on the star graph formed by $i$ and \smash{$\Co^0$}, and $i$ does not update throughout $[s,U(s)]$, we can think of the updating clocks placed on each connector $j$ as placed on the edge $\{i,j\}$ instead, and hence on this interval the model is equivalent to the one studied in~\cite{JLM22} (with the parameter $\mathfrak{t}$ being equal to $1$) so we can use some of the results therein. As in said work we will assume that the realization of the dynamical network around $i$ is typical throughout $[s,U(s)]$, meaning that the event
		\[\mathcal{A}_2:=\big\{|\Co_{s+k,i}|\geq \tfrac{c_{\eqref{localstable}}}{5} p(a,1) \;\forall 1\leq k\le \mathfrak{n}\big\}\]
		holds. As the initial distribution is the stationary measure, at any given time $t$ each connector \smash{$j\in\Co_{\lceil s\rceil}$} belongs to $\Co_{t,i}$ independently with probability $\frac{1}{N}p(a,1)$, and because on $\mathcal{A}_0$ we have $|\Co_{\lceil t\rceil}|>c_{\eqref{localstable}}N$, a Chernoff bound for binomial random variables gives
	\begin{equation}\label{ldevctx}\P\big(|\Co_{t,i}|\leq \tfrac{c_{\eqref{localstable}}}{5} p(a,1)\,\big|\,\mathcal{A}_0\big)\leq e^{-c_{\eqref{localstable}} p(a,1)/3}.\end{equation}
	It follows that 
	\begin{equation}\label{auxA1}\P\big(\mathcal{A}_2\,\big|\;\mathcal{A}_1,\,U(s)\big)\geq 1-(U(s)-s)e^{-c_{\eqref{localstable}} p(a,1)/3}.\end{equation}
	Fix now a realization of the updating events of the connectors on $[s,U(s)]$, and also on the connections between them and $i$ such that both $\mathcal{A}_1$ and $\mathcal{A}_2$ hold. On this event we have to bound from below the probability of $\mathcal{S}^{long}_s$, so in particular we need to prove that the star $i$ manages to infect several connectors. This is exactly the content of 
	Proposition~3.11 in \cite{JLM22}, which (adapted to this context) states that there is $C_{\eqref{eq1prev}}$ depending on $c_{\eqref{localstable}}$ alone, such that if $\lambda^2p(a,1)\geq C_{\eqref{eq1prev}}$ then	\begin{align}\label{eq1prev}\P\bigg(|\Co'_{s+2,i}|\geq \tfrac{c_{\eqref{localstable}}\lambda p(a,1)}{80}\;\bigg|\;U(s),\,\mathcal{A}_1,\,\mathcal{A}_2\bigg)&\geq\,1-\sfrac{C_{\eqref{eq1prev}}(\lambda\log(\lambda p(a,1))+ 1)}{\lambda^2 p(a,1)}\;\geq\,1-\sfrac{2C_{\eqref{eq1prev}}}{\lambda^2 p(a,1)},\end{align}
	where the second inequality follows from the hypothesis $p(a,1)\leq e^{1/\lambda}$, and by taking $c_\lambda<1$. Observe that we can use this result by simply imposing that $C_{\ref{survlong}}>C_{\eqref{eq1prev}}$. Once there are sufficiently many infected neighbours of $i$, local survival around this star follows directly from Proposition 3.10 in \cite{JLM22}, which (once again adapted to this context) states that 
	by choosing $c_{\lambda}$ small depending on $c_{\eqref{localstable}}$, there is a universal $C_{\eqref{ineq:localsurvival}}>0$ such that
	{\small\begin{equation}\label{ineq:localsurvival}
		\P\bigg(\bigcap_{k=2}^{\mathfrak{n}-2}\mathcal{W}_i(s+k)\,\bigg|\,U(s),\mathcal{A}_1,\mathcal{A}_2,|\Co'_{s+2,i}|\geq \tfrac{c_{\eqref{localstable}}\lambda p(a,1)}{80}\,\bigg)\;\geq\;1-(U(s)-s)C_{\eqref{ineq:localsurvival}}e^{-\tfrac{c_{\eqref{localstable}}\lambda^2p(a,1)}{200}}	
	\end{equation}}
$\!\!$on the event $\{U(s)-s\geq 4\}$, where the term $U(s)-s$ (not appearing in the original statement) follows from a union bound over the probability of failing to satisfy at least one of the $\mathcal{W}_i(s+k)$ events. Putting \eqref{auxA1}, \eqref{eq1prev} and \eqref{ineq:localsurvival} together we obtain that on the event $\{U(s)-s\geq 4\}$, we have
\[\P\bigg(\bigcap_{k=2}^{\mathfrak{n}-2}\mathcal{W}_i(s+k)\,\bigg|\,U(s),\mathcal{A}_1\bigg)\;\geq\;1-\sfrac{2C_{\eqref{eq1prev}}}{\lambda^2 p(a,1)}-(U(s)-s)\bigg[e^{-\frac{c_{\eqref{localstable}} p(a,1)}{3}}+C_{\eqref{ineq:localsurvival}}e^{-\tfrac{c_{\eqref{localstable}}\lambda^2p(a,1)}{200}}\bigg].\]
Assuming that $c_\lambda<1$ is enough to guarantee that $C_{\eqref{ineq:localsurvival}}e^{-c_{\eqref{localstable}}\lambda^2p(a,1)/200}\geq e^{-c_{\eqref{localstable}} p(a,1)/3}$ and hence we arrive at
	\begin{equation}\label{lbtT}
\P\bigg(\bigcap_{k=2}^{\mathfrak{n}-2}\mathcal{W}_i(s+k)\;\bigg|\;U(s),\,\mathcal{A}_1\bigg)\geq 1-\sfrac{2C_{\eqref{eq1prev}}}{\lambda^2 p(a,1)}-2(U(s)-s)C_{\eqref{ineq:localsurvival}}e^{-c_{\eqref{localstable}}\lambda^2p(a,1)/200},
	\end{equation}
which addresses condition $(2)$ in the definition of $\mathcal{S}_s^{long}$. In order to address condition $(3)$ within that definition we now prove that
	\begin{equation}\label{step4}\P\big(\bar{X}_{U(s)}^i(i)=1\,\big|\,\mathcal{W}_i(s+\mathfrak{n}-2),\,U(s),\,\mathcal{A}_1\big)\geq1-\tfrac{100}{c_{\eqref{localstable}}\lambda^2p(a,1)}.\end{equation}
	Observe that if $\mathcal{W}_i(s+\mathfrak{n}-2)$ holds then $\big|\Co'_{s+\mathfrak{n}-2,i}\big|>c_{\eqref{localstable}}\lambda p(a,1)/80$ and since stable connectors do not update nor recover in at least $4$ units of time, these represent sources of infection throughout the interval $[s+\mathfrak{n}-1,U(s)]$. Now, let
	\[\mathcal{Z}\,:=\,\bigcup_{j}\I_0^{ij}\,\cup\,\Rec^i,\]
	where $j$ ranges over all the connectors in $\Co'_{s+\mathfrak{n}-2,i}$, and which from our assumption on the number of stable connectors, is a Poisson point process with intensity bounded from below by
	$(c_{\eqref{localstable}}\lambda^2p(a,1)/80)+1$. It is sufficient for $i$ to be infected at time $U(s)$ that
	\begin{itemize}
		\item $\mathcal{Z}\cap[s+\mathfrak{n}-1,U(s)]\neq\emptyset$ (observe that $[s+\mathfrak{n}-1,U(s)]$ has length at least $1$), and
		\item  the last element of $\mathcal{Z}$ in this set does not belong to $\Rec^i$,
	\end{itemize}
	so computing the probability of these two events gives
	the lower bound
	\begin{align*}\big(1-e^{-\tfrac{c_{\eqref{localstable}}\lambda^2p(a,1)}{80}-1}\big)\big(1-\tfrac{80}{c_{\eqref{localstable}}\lambda^2p(a,1)+80}\big)&\geq\;1-e^{-\tfrac{c_{\eqref{localstable}}\lambda^2p(a,1)}{80}}-\tfrac{90}{c_{\eqref{localstable}}\lambda^2p(a,1)}\\[1pt]&\geq\;1-\tfrac{100}{c_{\eqref{localstable}}\lambda^2p(a,1)},\end{align*}
	where both inequalities follow from the hypothesis $\lambda^2p(a,1)\geq C_{\ref{survlong}}$ with $C_{\ref{survlong}}$ sufficiently large (depending on $c_{\eqref{localstable}}$). Combining \eqref{step4} with \eqref{lbtT}, and assuming that $C_{\eqref{eq1prev}}\geq 50$ gives
	\[\P\left(\mathcal{S}_s^{long}\,\big|\,U(s),\,\mathcal{A}_1\right)\geq1-\sfrac{4C_{\eqref{eq1prev}}}{\lambda^2 p(a,1)}-2(U(s)-s)C_{\eqref{ineq:localsurvival}}e^{-c_{\eqref{localstable}}\lambda^2p(a,1)/200}\]
	on the event $\{U(s)-s\geq4\}$. Taking expectations with respect to $U(s)-s$ (which is exponentially distributed with rate $\kappa(a)$) we obtain
	\[\P\left(\mathcal{S}_s^{long}\,\big|\,\mathcal{A}_1\right)\geq e^{-4\kappa(a)}-\sfrac{4C_{\eqref{eq1prev}}}{\lambda^2 p(a,1)}-2\kappa(a)^{-1}C_{\eqref{ineq:localsurvival}}e^{-c_{\eqref{localstable}}\lambda^2p(a,1)/200},\]
	so in order to conclude \eqref{proplong} it is enough to show that by properly choosing the constants $c_{\lambda}$ and $C_{\ref{survlong}}$ we have
	\[2\kappa(a)^{-1}C_{\eqref{ineq:localsurvival}}e^{-c_{\eqref{localstable}}\lambda^2p(a,1)/200}\leq\sfrac{4C_{\eqref{eq1prev}}}{\lambda^2 p(a,1)}.\]
	Recall from \eqref{condp} that $p(a,1) \le\cu a^{-\gamma}$, so using condition (ii) from the hypothesis the previous inequality follows by showing that
	\[2C_{\eqref{ineq:localsurvival}}(a\lambda)^{\frac{c_{\eqref{localstable}}C_{\ref{survlong}}}{200}}\leq\sfrac{4C_{\eqref{eq1prev}}\kappa_0a^{\gamma(1-\eta)}}{\lambda^2 \cu}\]
	Since $a,\lambda<1$ it is enough to choose $\frac{c_{\eqref{localstable}}C_{\ref{survlong}}}{200}\geq \gamma(1-\eta)$ and $c_\lambda$ so small that \smash{$c_\lambda^2\leq\frac{4C_{\eqref{eq1prev}}\kappa_0}{2C_{\eqref{ineq:localsurvival}}\cu}$}. The result follows by fixing $C_{\eqref{proplong}}=8C_{\eqref{eq1prev}}$.
\end{proof}

\medskip

\begin{remark}\label{re:smalla}
	Observe that the previous result holds even if $(i)$ is not satisfied. In this case the lower bound is slightly worse;
	\[	\P(\mathcal{S}^{long}_s\,\big|\,\mathfrak{F}^i_s,\mathfrak{F})\geq e^{-4\kappa(a)}-\frac{C_{\eqref{proplong}}\log( p(a,1))}{\lambda p(a,1)}\]
\end{remark}

\pagebreak[3]

Now define $\mathcal{S}_s:=\mathcal{S}_s^{long}\cup\mathcal{S}_s^{short}$ and observe that under the conditions of Proposition~\ref{survshort} and Proposition~\ref{survlong} we~have
\[\P(\mathcal{S}_s\,\big|\,\mathfrak{F}^i_s,\mathfrak{F})\geq 1-\frac{5C_{\eqref{propshort}}\kappa(a)}{c_{\eqref{localstable}}\lambda^2p(a,1)}\log\left(1+\frac{c_{\eqref{localstable}}\lambda^2p(a,1)}{5(1+\kappa(a))^2}\right)-\frac{C_{\eqref{proplong}}}{\lambda^2p(a,1)}\]
on the event
$\mathcal{A}_1$ if we choose $\lambda$ and $a$ satisfying the conditions of Proposition \ref{survlong}. Observe however that by definition of $\kappa(a)$ and \eqref{condp} we have
\[\kappa(a)\log\left(1+\frac{c_{\eqref{localstable}}\lambda^2p(a,1)}{5(1+\kappa(a))^2}\right)\leq \kappa_0a^{-\gamma\eta}\log\left(1+c_{\eqref{localstable}}\cu a^{-\gamma}\right)\]
which is bounded since $\eta<0$, and where the bound depends on the parameters appearing in the inequality above. It follows that we can further bound
\begin{equation}\label{totS}\P(\mathcal{S}_s\,\big|\,\mathfrak{F}^i_s,\mathfrak{F})\geq \;1-\frac{C_{\eqref{totS}}}{\lambda^2p(a,1)}\end{equation}
where $C_{\eqref{totS}}$ again only depends on $c_{\eqref{localstable}}$, $\gamma$, $\eta$, $\kappa_0$, $\cl$ and $\cu$.
\smallskip

The event $\mathcal{S}_s$ can be seen as surviving the next updating event after $s$, but in a specific way if the event takes a long time to occur. The main result of this section is Proposition~\ref{survivalvertex}, which states that once infected the process can survive in this manner until a given time horizon. Let $T>0$ represent this time horizon and $0\leq t\leq T$ be the time at which $i$ gets infected (the reader might think of taking the value $t=0$ for all practical purposes). Also, let $u_1< u_2<\cdots$ be the ordered elements of $\U^i\cap[t,T]$ to which we also add $u_0=t$. We use these times and events to give a technical definition of local survival. 

\pagebreak[3]

\begin{definition}\label{definfected}
	Let $0\leq t<T$ we say that $i\in\St$ is \emph{$[t,T]$-infected} if the following conditions are satisfied:\smallskip
	\begin{enumerate}
		\item $\frac{T-t}{3}\kappa(\tfrac{a}{2})\leq |\U^i\cap[t,T]|\leq 3(T-t)\kappa(a)$.\\[-3pt]
		\item At every $u_k$ with $k<|\U^i\cap[t,T]|$ both $\mathcal{S}_{u_k}$ and $\{|\Co_{u_k,i}|>\tfrac{c_{\eqref{localstable}}}{5} p(a,1)\}$ hold.\\[-3pt]
		\item $\bar{X}^i_{T}(i)=1$.
	\end{enumerate}
\end{definition}

The next result states that given the conditions of Theorem \ref{theoslow}, the probability of being $[t,T]$-infected is bounded away from zero for certain choices of $t$ and $T$.
\smallskip

\begin{proposition}[\textbf{Local survival}]\label{survivalvertex}
	Recall the definition of $T_{loc}$ in \eqref{deftloc} and the constants $c_{\lambda}$, $C_{\ref{survlong}}$ from Proposition \ref{survlong}. Fix $i\in\St$ and assume that $\eta<0$. For any $\varepsilon>0$ there are $C_{\ref{survivalvertex}}>C_{\ref{survlong}}$, $0<\tilde{c}_\lambda\leq c_\lambda$ and $0<\mathfrak{c}_0<1$ independent of $a$, $\lambda$ and $N$, such that if $\lambda\in (0,\tilde{c}_\lambda)$ and $a\in(0,1)$ are chosen satisfying
		\begin{itemize}
			\item[(i)] $p(a,1)\leq e^{1/\lambda}$
			\item[(ii)] $\lambda^2p(a,1)>-C_{\ref{survivalvertex}}\log(a\lambda)$
		\end{itemize}
		then for any $0\leq t\leq T\leq e^{c_{\eqref{localstable}}N}$ with $\frac{\mathfrak{c}_0\TL}{2}\leq T-t\leq\mathfrak{c}_0\TL$,
	\[\P(i\text{ is }[t,T]\text{-infected}\,\big|\,\mathfrak{F}^i_t,\mathfrak{F})\,\geq\, 1-\varepsilon,\]
	on the event 
	\[\{\bar{X}_t^i(i)=1\}\cap\mathcal{A}_0\]
	where $\mathfrak{F}^i_t$ and $\mathfrak{F}$ are as defined before  Proposition~\ref{survshort}. \end{proposition}

\begin{remark}\label{re:conditions}
	The condition $p(a,1)\leq e^{\frac{1}{\lambda}}$ is assumed here to avoid a logarithmic factor appearing as in Proposition 3.11 in \cite{JLM22}, and it will be satisfied for all the cases treated in this paper since $a$ is always a power of $\lambda$ (maybe with some logarithmic correction). However the result still holds redefining $\TL$ as
	\[\TL=\frac{\lambda p(a,1)}{\kappa(a)\log(p(a,1))},\] which follows from the proof of Proposition \ref{survivalvertex} and Remark \ref{re:smalla}. The condition $\lambda^2p(a,1)>-C_{\ref{survivalvertex}}\log(a\lambda)$ on the other hand is slightly stronger than the usual condition $\lambda^2k\gg1$ required for local survival of a vertex of degree $k$ in the literature. Even though we could derive local survival under this weaker condition, we do not prove this and instead we ensure $[t,T]$-infection which is a stronger event that will be sufficient for the spreading mechanisms described in Sections \ref{DDS} and \ref{DIS}.
\end{remark}

\begin{proof}
	As in the previous results we abuse notation and assume that $\P$ stands for the law of the process conditioned on $\mathfrak{F}^i_t$ and $\mathfrak{F}$. Since $|\U^i\cap[t,T]|$ is Poisson distributed with rate between $(T-t)\kappa(\frac{a}{2})$ and 
	 $(T-t)\kappa(a)$, both of them bounded from below by $2^{\gamma\eta-1}\mathfrak{c}_0\lambda^2p(a,1)\geq C$, with $C=2^{\gamma\eta-1}\mathfrak{c}_0C_{\ref{survivalvertex}}$.
	It follows from standard Chernoff bounds that
	\[\P\big(|\U^i\cap[t,T]|\geq 3(T-t)\kappa(a)\big)\,\leq\,e^{-C}\]and \[\P\big(|\U^i\cap[t,T]|\leq \tfrac{(T-t)}{3}\kappa(\tfrac{a}{2})\big)\,\leq\,e^{-0.3C},\]
	so the probability of not satisfying condition $(1)$ in Definition~\ref{definfected} is bounded by $2e^{-0.3C}$. 
	To address condition $(2)$ observe that, as in \eqref{ldevctx},
	\[\P\big(|\Co_{t,i}|> \tfrac{c_{\eqref{localstable}}}{5} p(a,1)\,\big|\,\mathcal{A}_0\big)\geq 1-e^{-c_{\eqref{localstable}} p(a,1)/3},\]
	and since $C_{\ref{survivalvertex}}>C_{\ref{survlong}}$, $\tilde{c}_\lambda<c_\lambda$, and we are assuming $\bar{X}_t^i(i)=1$, then \eqref{totS} gives that \smash{$\P(\mathcal{S}_s)\geq 1-\frac{C_{\eqref{totS}}}{\lambda^2p(a,1)}$}. The event is $\mathfrak{F}^i_{u_1}$-measurable so we can use the strong Markov property to restart the process at $u_1$, where the event $\mathcal{S}_s$ now gives $\bar{X}_{u_1}^i(i)=1$. Since on $u_1$ there is an updating time we can again use \eqref{ldevctx} to deduce 
	\smash{$\P\big(|\Co_{u_1,i}|> \tfrac{c_{\eqref{localstable}}}{5} p(a,1)\,\big|\,\mathcal{A}_0\big)\geq 1-e^{-c_{\eqref{localstable}} p(a,1)/3}$},
	which in turn allows us to bound the probability of $\mathcal{S}_{u_1}$ from below as before. Proceeding repeatedly we get,
	\begin{align*}\P(i\text{ does not satisfy }(2))&\leq\,\P(i\text{ does not satisfy }(3)\text{ and } |\U^i\cap[t,T]|\leq 3(T-t)t\kappa(a))\,+\,e^{-C}\\&\leq\;3\mathfrak{c}_0\TL\kappa(a)\left(\frac{C_{\eqref{totS}}}{\lambda^2p(a,1)}+e^{-c_{\eqref{localstable}} p(a,1)/3}\right)\,+\,e^{-C}\\&\leq\;3\mathfrak{c}_0\TL\kappa(a)\left(\frac{C_{\eqref{totS}}}{\lambda^2p(a,1)}+\lambda^{\frac{c_{\eqref{localstable}} }{3\lambda^2C_{\ref{survivalvertex}}}}\right)\,+\,e^{-C}\\&\leq\;3C_{\eqref{totS}}\mathfrak{c}_0\,+\,3\mathfrak{c}_0\cu\lambda^{\frac{c_{\eqref{localstable}} }{3\lambda^2C_{\ref{survivalvertex}}}-2-\gamma}\,+\,e^{-C},\end{align*}
	where in the third inequality we have used $(ii)$ from our assumptions, and in the last line we used the definition of $\TL$ and \eqref{condp}.
	 Finally, to address condition~$(3)$ we can repeat the proofs of Propositions \ref{survshort} and \ref{survlong} but replacing the event $\bar{X}^i_{U(s)}(i)=1$ with $\bar{X}^i_{T}(i)=1$, thus obtaining from \eqref{totS}
	\begin{align*}\P\big(i\text{ satisfies }(3)\,\big|\,u_l,\,\bar{X}_{u_l}^i(i)=1\big)&\geq 1-\frac{C_{\eqref{totS}}}{\lambda^2p(a,1)}\geq1-\frac{C_{\eqref{totS}}}{C_{\ref{survivalvertex}}}.\end{align*}
	Putting together all the bounds obtained thus far gives
	\[\P(i\text{ is not  }[t,T]\text{-infected}\,\big|\,\mathfrak{F}^i_t,\mathfrak{F})\,\leq\, 4e^{-0.3C}+3C_{\eqref{totS}}\mathfrak{c}_0+3\mathfrak{c}_0\cu\lambda ^{\frac{c_{\eqref{localstable}} }{3\lambda^2C_{\ref{survivalvertex}}}-2-\gamma}+\frac{C_{\eqref{totS}}}{C_{\ref{survivalvertex}}}\]
	Now fix the constants $C_{\ref{survivalvertex}}$, $\tilde{c}_\lambda$, and $\mathfrak{c}_0$ as follows; first choose $\mathfrak{c}_0$ small such that \smash{$3C_{\eqref{totS}}\mathfrak{c}_0\leq\frac{\varepsilon}{4}$}, then fix $C_{\ref{survivalvertex}}$ sufficiently large such that ${C_{\eqref{totS}}}/{C_{\ref{survivalvertex}}}\leq\frac{\varepsilon}{4}$ and $4e^{-0.3C}=$ \smash{$4e^{-0.32^{\gamma\eta-1}\mathfrak{c}_0C_{\ref{survivalvertex}}}\leq\frac{\varepsilon}{4}$}. Finally, fixing $\tilde{c}_\lambda$ sufficiently small so that $$3\mathfrak{c}_0\cu\lambda ^{\frac{c_{\eqref{localstable}} }{3\lambda^2C_{\ref{survivalvertex}}}-2-\gamma}\leq\frac{\varepsilon}{4}$$ we obtain the desired bound.
\end{proof}

\subsection{Delayed direct spreading}\label{DDS}

In this section we prove that under the main conditions of Theorem \ref{theoslow}  together with
\[(\kappa(a)\wedge\lambda) ap(a,a)\TL>M_{(iii)},\]
for sufficiently large $M_{(iii)}$, there is slow extinction and we get the corresponding lower bounds for the metastable density.\smallskip

To begin the proof recall the definitions of \smash{$\TL$, $J_k$, $\Co_{t,i}$, $\Co'_{t,i}$, $\mathcal{A}_0$}, the $\mathcal{S}$ events, and the $\mathcal{W}_i(t)$ events pertaining to a given star $i\in\St$, together with the definition of $i$ being $[t,T]$-infected. Also recall the constants $\tilde{c}_\lambda$, $C_{\ref{survivalvertex}}$ and $0<\mathfrak{c}_0<1$ appearing in Proposition \ref{survivalvertex}, we will assume throughout this section that $C_{\ref{survivalvertex}}>C_{\ref{theoslow}}$ and $c_{\ref{theoslow}}\leq\tilde{c}_\lambda$. Fix a realization of the $\Rec^{j}$ and $\U^j$ processes with $j\in\Co^0$ such that $\mathcal{A}_0$ holds, and define the sequences $\{\tilde{\St}_m\}_{m\in\N}$ and $\{\St_m'\}_{m\in\N}$ as
\begin{align*}
\tilde{\St}_m&\,:=\,\big\{i\in\St \colon \bar{X}_{\mathfrak{c}_0\TL m}^i(i)=1\big\} \\[1ex]\displaystyle\St_m'&\displaystyle\,:=\,\big\{i\in\tilde{\St}_m \colon i\text{ is }[\mathfrak{c}_0\TL m,\mathfrak{c}_0\TL(m+1)]\text{-infected}\big\}
\end{align*}
Since at time $0$ all vertices are infected we have $|\tilde{\St}_0|=aN/2$. Our goal is to show that there are small constants \smash{$c_{\eqref{recursivevertex}},c_{\eqref{recursivevertex}}'>0$} 
with \smash{$c_{\eqref{recursivevertex}}$} independent of $a$, $\lambda$ and $N$, and \smash{$c_{\eqref{recursivevertex}}'$} independent of $N$ such that
\begin{equation}\label{recursivevertex}
\P\big(|\tilde{\St}_{m+1}|>c_{\eqref{recursivevertex}}a N\,\big|\,|\tilde{\St}_{m}|>c_{\eqref{recursivevertex}}a N\big)\,\geq\,1-e^{-c_{\eqref{recursivevertex}}' N},
\end{equation}
which together with a union bound gives slow extinction. Define $\mathcal{F}_t$ as the $\sigma$-algebra generated by the graphical construction up to time $t$, our proof relies on the fact from Proposition \ref{survivalvertex} that, for any $i\in\St$ and $\mathfrak{c}_0\TL m\leq t\leq \mathfrak{c}_0\TL(m+1/2)$, taking $\varepsilon=\frac{1}{2}$,
\begin{equation}\label{keydd}\P\big(i\text{ is }[t,\mathfrak{c}_0\TL(m+1)]\text{-infected}\,\big|\,\mathcal{F}_{t}\big)\,\geq\,1/2\end{equation}
on the event $\{\bar{X}_{t}^i(i)=1\}$ as soon as $\lambda$ and $a$ are sufficiently small (satisfying the conditions of the proposition). Now, it follows from the definition of $[t,\mathfrak{c}_0\TL(m+1)]$-infection that any star satisfying the event above must be infected at time $\mathfrak{c}_0\TL(m+1)$ so in the particular case $t=\mathfrak{c}_0\TL m$ we conclude that $\St'_m\subseteq \tilde{\St}_{m+1}$.  On the other hand, from \eqref{keydd} we deduce that each $x\in\tilde{\St}_m$ belongs to $\St'_m$ independently with probability at least $1/6$ so from a large deviation bound we obtain
\begin{equation}\label{enoughvertex}\P\Big(|\tilde{\St}_{m+1}|>\sfrac{1}{4}|\tilde{\St}_m|\,\Big|\,\mathcal{F}_{\mathfrak{c}_0\TL m}\Big)\,\geq\,\big(1-e^{-c_{\eqref{recursivevertex}}a N/10}\big)\one_{\{|\tilde{\St}_m|>c_{\eqref{recursivevertex}}aN\}}.\end{equation}
From this bound it follows that \eqref{recursivevertex} is automatically satisfied in the case $|\tilde{\St}_m|>4c_{\eqref{recursivevertex}}aN$ therefore we assume onwards that $c_{\eqref{recursivevertex}}aN<|\tilde{\St}_m|\leq4c_{\eqref{recursivevertex}}aN$, and further that $c_{\eqref{recursivevertex}}<1/16$ so that $4c_{\eqref{recursivevertex}}aN\leq\frac{aN}{4}$. Divide $[\mathfrak{c}_0\TL m,\mathfrak{c}_0\TL(m+\sfrac{1}{2})]$ into $\lfloor\mathfrak{c}_0\TL\kappa(a)/2\rfloor$ intervals 
\[K_k=[\mathfrak{c}_0\TL m+k\kappa(a)^{-1},\,\mathfrak{c}_0\TL m+(k+1)\kappa(a)^{-1}]\]
of length $\kappa(a)^{-1}$ (recall from our assumptions that  $\TL\kappa(a)=\lambda^2p(a,1)$ is large) and divide these intervals  further into their first and second half, $K_k^{-}$ and $K_k^{+}$, respectively. Using the $K_k$ we define a sequence of i.i.d. Bernoulli random variables $U(i,k)$ with $i\in\St'_m$ and $0\leq k\leq\lfloor\mathfrak{c}_0\TL\kappa(a)/2\rfloor-1$ such that
\begin{equation}\label{defu}U(i,k)=1\;\Longleftrightarrow\;\U^i\cap K_k^{-}\neq\emptyset\text{ and }\U^i\cap K_k^+=\emptyset,\end{equation}
that is, $U(i,k)$ indicates the event in which $i$ updates on the first half of $K_k$ and on the second half it does not update. Since the $K_k$ have length $\kappa(a)^{-1}$ it follows that $\P(U(i,k)=1)\geq \varphi$ for some constant $\varphi$ depending on $\kappa_0$ and $\eta$ alone. Using these random variables we construct the set
\[\St_m''\,:=\,\Big\{i\in\St'_m \colon\sum_{k=0}^{\lfloor\mathfrak{c}_0\TL\kappa(a)/2\rfloor-1}U(i,k)\,\geq\,\tfrac{\varphi\mathfrak{c}_0\TL\kappa(a)}{4}\Big\},\]
which contains the stars that satisfy the condition appearing in \eqref{defu} for sufficiently many intervals $K_k$. Noticing that the random variables $\{U(i,k)\}_{i\in\St'_m,k\geq0}$ are independent and that $\mathfrak{c}_0\TL\kappa(a)\geq\mathfrak{c}_0C_{\ref{theoslow}}$ can be taken to be large, we deduce from a large deviation argument together with \eqref{keydd} that
\[\P\big(i\in\St''_m\,\big|\,\mathcal{F}_{\mathfrak{c}_0\TL m}\big)\,\geq\,\frac{1}{8}\]
on the event $\{\bar{X}_{\mathfrak{c}_0\TL m}^i(i)=1\}$, and hence for $N$ sufficiently large yet another large deviation argument gives
\[\P\Big(|\St''_{m}|>\sfrac{c_{\eqref{recursivevertex}}aN}{16}\,\Big|\,\mathcal{F}_{\TL m}\Big)\,\geq\,\left(1-e^{-c_{\eqref{recursivevertex}} N/30}\right)\one_{\{|\tilde{\St}_m|>c_{\eqref{recursivevertex}}aN\}}.\]
Let $\mathcal{A}_m=\{|\St''_{m}|>\sfrac{c_{\eqref{recursivevertex}}aN}{16}\}$ the event within the probability above, which we will assume holds for a given $m$ since the lower bound is already of the form $1-e^{-cN}$. Let $i\in\St\setminus\tilde{\St}_m$ be a star not infected at time $\mathfrak{c}_0\TL m$; our main goal is to show that the probability that at some point in $[\mathfrak{c}_0\TL m,\mathfrak{c}_0\TL(m+\sfrac{1}{2})]$ $i$ becomes infected by some $i'\in\St''_m$, is bounded from below by some constant independent of $a$, $\lambda$ and $N$. To do so fix $i'\in\St''_m$ and observe:
\begin{itemize}[leftmargin=*]
	\item By definition $U(i',k)=1$ for at least $\varphi\mathfrak{c}_0\TL\kappa(a)/4$ many intervals $K_k$.
	\item For any fixed such $k$ there are no updating events of $i'$ on $K_k^{+}$ which means that the time difference between the last updating event $u_l$ in $K_k^{-}$ to the next one is greater than $4$.
	\item Since $i'$ is $[\mathfrak{c}_0 m,\mathfrak{c}_0\TL(m+1)]$-infected we know that $\mathcal{S}_{u_l}^{long}$ holds and hence the event $\mathcal{W}_{i'}(u_l+\ell)$ holds 
	for every $\ell\in\N$ with $u_l+\ell\in K_k^+$. In particular, we have
	\[\int_{K_k^+}\bar{X}_{s}^{i'}(i')\, \mathrm ds\,\geq\,\kappa(a)^{-1}/8\]
	\item Since there is at least one updating event of $i'$ in $K_k^-$, the edge $\{i,i'\}$ is present throughout $K_k^+$ with probability \smash{$\frac{p(a,a)}{N}$}.
\end{itemize}
As a result of these observations the probability of $i'$ infecting $i$ at some $t\in K_k$ is at least \[\big(1-e^{-\lambda\kappa(a)^{-1}/8}\big)\frac{p(a,a)}{N}\;\geq\;\frac{(1\wedge\lambda\kappa(a)^{-1})p(a,a)}{20N}\]
and since $i'\in\St''_m$, there are at least $\varphi\mathfrak{c}_0\TL\kappa(a)/4$ intervals $K_k$ of length $\kappa(a)^{-1}$ such that $U(i',k)=1$. When considering all the pairs $(i',K_k)$ with $i'\in\St_m''$ and $K_k$ as before we conclude that
\begin{equation}\label{getsinfected}\P\big(\exists i'\in\St''_m\text{ that infects }i\text{ in }[\mathfrak{c}_0\TL m,\mathfrak{c}_0\TL(m+\sfrac{1}{2})]\,\big|\,|\St''_{m}|>\sfrac{c_{\eqref{recursivevertex}}aN}{16}\big)\end{equation}
is bounded from below by
\[1-\left(1-\tfrac{(1\wedge\lambda\kappa(a)^{-1})p(a,a)}{20N}\right)^{\frac{c_{\eqref{recursivevertex}}\varphi\mathfrak{c}_0\TL\kappa(a)aN}{32}}\,\geq\,1-\exp\left(-\tfrac{c_{\eqref{recursivevertex}}\mathfrak{c}_0\varphi(\kappa(a)\wedge\lambda)ap(a,a)\TL}{1500}\right)\,\geq\,\frac{2}{3},\]
where the last inequality follows from our assumptions on the parameters, given that $M_{(iii)}$ is taken large enough. Assume now that the event above holds and let $\tau_i$ be the first time in $[\mathfrak{c}_0\TL m,\mathfrak{c}_0\TL(m+\sfrac{1}{2})]$ that $\bar{X}_{\tau_i}^i(i)=1$. The idea at this point is to ask whether $i$ is $[\tau_i,\mathfrak{c}_0\TL(m+1)]$-infected which in particular implies that $i$ is infected at time $\TL(m+1)$ and hence $i\in\tilde{\St}_{m+1}$. Now, these events rely on the $\I^{i,j}$, $\mathcal{C}^{i,j}$ and $\U^i$ events on $[\tau_i,\mathfrak{c}_0\TL(m+1)]$ and hence are independent of the event in \eqref{getsinfected}. Also, since $\varphi$ is independent of $a$, $\lambda$ and~$N$ we can apply \eqref{keydd} with $t=\tau_i$ to conclude that 
\[\P\big(i\in\tilde{\St}_{m+1}\,\big|\,\exists i'\in\St''_m\text{ that infects }i\text{ in }[\mathfrak{c}_0\TL m,\mathfrak{c}_0\TL(m+\sfrac{1}{2})]\big)\,\geq\,\frac{1}{12}.\]
Finally, observe that from our assumption \smash{$c_{\eqref{recursivevertex}}<\sfrac{1}{48}$} we have \smash{$|\St\setminus\tilde{\St}_m|\geq \frac{aN}{4}$} in the event \smash{$\{c_{\eqref{recursivevertex}}aN<|\tilde{\St}_m|\leq4c_{\eqref{recursivevertex}}aN\}$} and since the events $\{i\in\tilde{\St}_{m+1}\}$ are independent we can use a large deviation bound to conclude
\[\P\big(|\tilde{\St}_{m+1}|\geq \sfrac{aN}{320} \,\big|\,c_{\eqref{recursivevertex}}'aN<|\tilde{\St}_m|\leq4c_{\eqref{recursivevertex}}aN\big)\,\geq\,1-e^{-c_{\eqref{recursivevertex}}'N},\]
for some $c_{\eqref{recursivevertex}}'>0$ independent of $N$. By taking $c_{\eqref{recursivevertex}}<\sfrac{1}{320}$ we finally conclude \eqref{recursivevertex}.\smallskip

To obtain the lower bound on the metastable density given by \eqref{lowdensityetaposdelayed} fix $\mathfrak{c}_0\TL<t<e^{c_{\eqref{recursivevertex}}'N}$ (we can take $t>\mathfrak{c}_0\TL$ w.l.o.g. since $\TL$ does not depend on $N$) and observe that from the last proof we have deduced in particular that
\[\P\big(|\tilde{\St}_{m-1}|>c_{\eqref{recursivevertex}}a N\big)\,\geq\,1-e^{-c_{\eqref{recursivevertex}}' N},\]
where $m$ is such that $t\in[\mathfrak{c}_0\TL m,\mathfrak{c}_0\TL(m+1)]$.
 We can use a slightly extended proof of 
Lemma~\ref{survivalvertex} including intervals of length up to $2\mathfrak{c}_0\TL$ to conclude that for every $i\in\tilde{\St}_{m-1}$ we have
\[\P\big(i\text{ is }[\mathfrak{c}_0\TL(m-1),t]\text{- infected}\,\big|\,|\tilde{\St}_{m-1}|>c_{\eqref{recursivevertex}}a N\big)\,\geq\,c'\]
for some $c'>0$ independent of $a$, $\lambda$ and $N$. Using a large deviation bound we conclude that the total number of stars satisfying the event above is at least \smash{$ac_{\eqref{recursivevertex}}c'N/2$} with high probability. As these stars are infected at time $t$ the result follows from Lemma~\ref{lemmalower}.

\subsection{Delayed indirect spreading}\label{DIS}

In this section we prove that under the main conditions of Theorem \ref{theoslow}  together with
\[\lambda^2a p^2(a,1)\TL>M_{(iv)},\]
for sufficiently large $M_{(iv)}$, there is slow extinction and we get the corresponding lower bounds for the metastable density. We follow a similar structure to the one used in the previous section but incorporating the ideas used for \textit{delayed indirect spreading} in \cite{JLM19}. As before, recall from Section~\ref{localsurvival} the definitions of $\TL$, $J_k$, $\Co_{t,i}$, $\Co'_{t,i}$, $\mathcal{A}_0$, the $\mathcal{S}$ events, and the $\mathcal{W}_i(t)$ events pertaining to a given star $i\in\St$, together with the definition of $[t,T]$- infection. As before recall the constants $\tilde{c}_\lambda$, $C_{\ref{survivalvertex}}$ and $0<\mathfrak{c}_0<1$ appearing in 
Proposition~\ref{survivalvertex}, and assume that $C_{\ref{survivalvertex}}>C_{\ref{theoslow}}$ and $c_{\ref{theoslow}}\leq\tilde{c}_\lambda$. Further, assume that $\mathcal{A}_0$ holds and defin e the sequences $\{\tilde{\St}_m\}_{m\in\N}$ and $\{\St_m'\}_{m\in\N}$ as in Section \ref{DDS}. Once again, the main goal is to show that there are small constants \smash{$c_{\eqref{recursivevertex2}},c_{\eqref{recursivevertex2}}'>0$} to be fixed later, with $c_{\eqref{recursivevertex2}}$ independent of $a$, $\lambda$ and $N$, and $c_{\eqref{recursivevertex2}}'$ independent of $N$ such that
\begin{equation}\label{recursivevertex2}
\P\big(|\tilde{\St}_{m+1}|>c_{\eqref{recursivevertex2}}a N\,\big|\,|\tilde{\St}_{m}|>c_{\eqref{recursivevertex2}}a N\big)\,\geq\,1-e^{-c_{\eqref{recursivevertex2}}' N},
\end{equation}
which together with a union bound gives slow extinction. Following the same reasoning as in the previous section we obtain
\[\P\Big(|\tilde{\St}_{m+1}|>\sfrac{1}{4}|\tilde{\St}_m|\,\Big|\,\mathcal{F}_{\mathfrak{c}_0\TL m}\Big)\,\geq\,\big(1-e^{-c_{\eqref{recursivevertex2}} N/10}\big)\one_{\{|\tilde{\St}_m|>c_{\eqref{recursivevertex2}}aN\}},\]
and \eqref{recursivevertex2} is automatically satisfied in the case $|\tilde{\St}_m|\ge 4c_{\eqref{recursivevertex2}}aN$. Henceforth we may assume that \smash{$c_{\eqref{recursivevertex2}}aN<|\tilde{\St}_m|\leq4c_{\eqref{recursivevertex2}}aN$} (once again choosing \smash{$c_{\eqref{recursivevertex2}}\leq\frac{1}{16}$}). It is at this point that the proof differs from the previous one. Indeed, while in the case of \emph{direct} spreading the proof relied on showing that most stars update in a `normal' fashion, in this case the main idea is to show that the local survival of stars is sufficiently `synchronous', meaning that during a fraction of the time we will always observe sufficiently many infected stars that can transmit their infections to connectors in $\Co^1$. We begin by introducing a set of indices 
\[\mathcal{J}:=5\N\cap[\mathfrak{c}_0\TL m,\mathfrak{c}_0\TL(m+\sfrac{1}{2})]\]
which are used to  define time intervals of the form $[k,k+5]$ with $k\in\mathcal{J}$. Using these indices we now define a new sequence of i.i.d. Bernoulli random variables $\bar{U}(i,k)$ with $i\in\St'_m$ and $k\in \mathcal{J}$ such that
\begin{equation}\label{defv}\bar{U}(i,k)=1\;\Longleftrightarrow\;\mathcal{R}^i\cap[k,k+5]=\emptyset,\end{equation}
and use these random variables to define the set of \textit{synchronous} stars $\St^s_m$ as
\[\St^s_m\,:=\,\Big\{i\in\St'_m \colon \sum_{k\in\mathcal{J}}\bar{U}(i,k)>\sfrac{\mathfrak{c}_0\TL}{60e^5}\Big\}.\]
The definition of $\St^s_m$ given here is similar to the one introduced in Definition 5 in \cite{JLM19}. The main idea is that intervals $[k,k+5]$ with sufficiently many stars satisfying $\bar{U}(i,k)=1$ serve as time windows in which connectors can become infected. These connectors, however, need to be able to hold the infection in order to further infect more vertices. We thus define for every $k\in \mathcal{J}$ the set
\[\Co^1_k\,:=\,\left\{j\in\Co^1 \colon \U^j\cap[k+3,k+4]\neq\emptyset,\,(\U^j\cup\Rec^j)\cap[k+4,k+5]=\emptyset\right\}\]
of \textit{stable} connectors, which neither update nor recover throughout $[k+4,k+5]$ but updated before. We claim that for any fixed $k\in\mathcal{J}$, a stable connector $j\in\Co_k^1$ is infected throughout $[k+4.5,k+5]$ as soon as there is some $i\in\St'_m$ for which:

\begin{itemize}
	\item[(i)] $\Rec^i\cap[k,k+5]=\emptyset$,
	\item[(ii)] $i$ is a neighbour of $j$ at time $k+4$, and
	\item[(iii)] $\I_0^{ij}\cap[k+4,k+4.5]\neq\emptyset$.
\end{itemize}
Indeed, suppose that $(i)$ holds; because $i\in\St'_m$, by definition it must be $[\mathfrak{c}_0\TL m,\mathfrak{c}_0\TL(m+1)]$-infected and hence it is either the case that there is an updating event $u\in[k,k+4]$ at which $\mathcal{S}_u$ holds so $i$ is infected, or there are no updating events in this interval, in which case at the last updating event $u$ we had $\mathcal{S}_u^{long}$ and hence $\mathcal{W}_i(u+l)$ holds for some $u+l\in[k,k+3]$ which again means that $i$ is infected at some time in $[k,k+4]$. Either way, because $\mathcal{R}^i\cap[k,k+5]=\emptyset$ the infection of $i$ is sustained throughout $[k+4,k+5]$ and hence it follows directly from $(ii)$, $(iii)$ that there is a valid infection event between $i$ and $j$ in $[k+4,k+4.5]$. By definition of $\Co^1_k$, $j$ does not recover throughout $[k+4,k+5]$, thus proving the claim.

\smallskip

Observe that a given $i\in\St'_m$ belongs to $\St^s_m$ with probability at least  $1-e^{-0.3\mathfrak{c}_0\TL/20e^5}\geq\frac{1}{8}$, where the bound follows by choosing $C_{\ref{survivalvertex}}$ large (depending on $\mathfrak{c}_0$). From this bound and Proposition~\ref{survivalvertex} (fixing $\varepsilon=\frac{1}{2}$) we thus obtain that
\[\P\big(i\in\St^s_m\,\big|\,\mathcal{F}_{\mathfrak{c}_0\TL m}\big)\,\geq\,\frac{1}{16}\]
on the event $\{\bar{X}_{\mathfrak{c}_0\TL m}^i(i)=1\}$ and hence a large deviation argument gives 
\begin{equation}\label{DISp1}\P\big(|\St^s_m|\geq \sfrac{c_{\eqref{recursivevertex2}}}{48}aN\,\big|\,\mathcal{F}_{\mathfrak{c}_0\TL m}\big)\,\geq\,\left(1-e^{-c_{\eqref{recursivevertex2}}N/80}\right)\one_{\{|\tilde{\St}_m|>c_{\eqref{recursivevertex2}}aN\}},
\end{equation}
which is already a bound of the right form for slow extinction so we may as well assume that 
\smash{$|\St^s_m|\geq \sfrac{c_{\eqref{recursivevertex2}}}{48}aN$} as soon as \smash{$|\tilde{\St}_m|>c_{\eqref{recursivevertex2}}aN$}. Similarly, observe that since the connectors have index at least $N/2$ we can safely say that each $j\in\Co^1$ belongs to $\Co^1_k$ independently with some probability $\varphi$ depending on $\kappa_0$ alone. Hence it follows from a large deviation argument that
\begin{equation}\label{C1}\P(|\Co_k^1|\geq\varphi N/3)\,\geq\,1-e^{-0.3\varphi N}\end{equation}
which is again of the right form for slow extinction so we again assume that $|\Co_k^1|\geq\varphi N/3$ for all $k\in\mathcal{J}$. Write $\mathcal{B}$ for the event $\{|\St^s_m|\geq \sfrac{c_{\eqref{recursivevertex2}}}{48}aN\}\cap\{(|\Co_k^1|\geq\varphi N/3),\;\forall k\in\mathcal{J}\}$ in which there are many synchronous stars to spread the infection, and stable connectors to receive it. 
Define now the index set
\[\mathcal{K}:= \left\{k\in\mathcal{J} \colon \big|\{i \in \St^s_m, \bar{U}(i,k)=1\}\big| \ge \sfrac{|\St^s_m|}{12e^5}\right\},\]
which corresponds to the set of intervals at which we can find enough stars in $\St^s_m$ that do not recover. Using the previous claim and the definitions of $\mathcal{B}$ and $\mathcal{K}$ we deduce that there is some constant $c$ independent of $\lambda$, $a$ and $N$ such that
\begin{equation}\label{desc1}
\P\big(|\{j\in\Co^1_k \colon X_{k+4.5}(j)=1\}| \geq \sfrac{c\varphi\lambda ap(a,1)N}{9}\,\big|\,\mathcal{F}_{\mathfrak{c}_0\TL m},\mathcal{B},k\in\mathcal{K}\big)\,\geq\,1-e^{-0.1\varphi c\lambda ap(a,1)N}
\end{equation}
on the event $\{|\tilde{\St}_m|>c_{\eqref{recursivevertex2}}aN\}$. Indeed, for any $j\in\Co^1_k$ and $i\in\St_m^s$ satisfying $\mathcal{R}^i\cap[k,k+5]=\emptyset$ the probability of satisfying $(ii)$ and $(iii)$ is equal to
\[\tfrac{p(a,1)}{N}(1-e^{-\lambda/2})\,\geq\,\tfrac{\lambda p(a,1)}{4N}\]
because by definition $j$ updated during $[k+3,k+4]$. Since $\mathcal{B}$ holds there are many synchronous stars, and hence using the definition of $\mathcal{K}$ we conclude from a large deviation argument and the previous claim that any given $j\in\Co^1_k$ is infected at time $k+4.5$ with probability at least
\[1-\left(1-\tfrac{\lambda p(a,1)}{4N}\right)^{\frac{c_{\eqref{recursivevertex2}}}{600e^5}aN}\,\geq\,1-\exp\left(-\tfrac{c_{\eqref{recursivevertex2}}\lambda ap(a,1)}{2400e^5}\right)\,\geq\,c\lambda ap(a,1),\]
where $c=1/4800e^5$. Hence \eqref{desc1} follows from the bound on $|\Co_k^1|$ and yet another large deviation argument. 
\pagebreak[3]\smallskip

So far we have seen that for any given $k\in\mathcal{K}$ with high probability there are of order $\lambda ap(a,1)N$ infected connectors at time $k+4.5$. However, in order to spread the infection we need to consider all these sources of infections at different times $k\in\mathcal{K}$ simultaneously. For that purpose we introduce the set $\mathscr{P}_m$ of pairs $(k,j)$ as follows
\[\mathscr{P}_m\,:=\,\left\{(k,j) \colon k\in \mathcal{K},\,j\in\Co^1_k,\,X_{k+4.5}(j)=1\right\}.\]
 To lower bound the number of synchronous times in $\mathcal{K}$ we use a double count to deduce
\[\sfrac {|\St^s_m| \mathfrak{c}_0\TL}{60e^5} \le \sum_{i\in \St^s_m} \sum_{k\in\mathcal{J}}\bar{U}(i,k) \le |\mathcal{K}| |\St^s_m| + (\sfrac{\mathfrak{c}_0\TL}{10}-|\mathcal{K}|) \sfrac {|\St^s_m|}{12e^5},\]
where the left inequality follows from the definition of 
$\St^s_m$, and the right one from the fact that there are at most $\sfrac{\mathfrak{c}_0\TL}{10}$ intervals in $[\mathfrak{c}_0\TL m,\mathfrak{c}_0\TL(m+\sfrac{1}{2})]$, and the definition of $\mathcal{K}$. We conclude that $|\mathcal{K}|\ge \mathfrak{c}_0\TL/240e^5$. Putting this together with \eqref{DISp1}, \eqref{C1}, and \eqref{desc1} gives
\begin{equation}\label{enoughpairs}
\P\big(|\mathscr{P}_m|\geq\sfrac{c\varphi\lambda ap(a,1)\mathfrak{c}_0\TL N}{720e^5}\,\big|\,\mathcal{F}_{\mathfrak{c}_0\TL m}\big)\,\geq\,1-e^{-c_{\eqref{enoughpairs}}N},
\end{equation}
for some $c_{\eqref{enoughpairs}}>0$ independent of $N$, on the event $\{|\tilde{\St}_m|>c_{\eqref{recursivevertex2}}aN\}$. In order to conclude \eqref{recursivevertex2} we need to show that stars in \smash{$\St\setminus\tilde{\St}_m$} have a sufficiently large chance of getting infected by the sources of infections represented by the pairs time-connector in $\mathscr{P}_m$. Fix $i\in \St\setminus\tilde{\St}_m$ and observe that in order for $i$ to get infected in $[\mathfrak{c}_0\TL m,\mathfrak{c}_0\TL(m+\sfrac{1}{2})]$ it is enough to find $(k,j)\in\mathscr{P}_m$ such that
\begin{itemize}
	\item[(i)] $j$ is a neighbour of $i$ at time $k+4$, and
	\item[(ii)] $\I_0^{i,j}\cap[k+4.5,k+5]\neq\emptyset$. 
\end{itemize}
Indeed, by definition of $\Co^1_k$ the connector $j$ does not recover not update throughout $[k+4.5,k+5]$ so any infection towards $i$ in this interval is valid. Once again the probability of satisfying these conditions for a given pair $(k,j)$ is at least 
\smash{$\frac{\lambda p(a,1)}{5N}$} and since the state of the edges $\{i,j\}$ at times $k+4.5$ are independent across the $j$ and $k$ (since by definition the connector updated in the interval $[k+3,k+4]$), a large deviation argument yields
\begin{align*}&\P\Big(\exists t\in[\mathfrak{c}_0\TL m,\mathfrak{c}_0\TL(m+\sfrac{1}{2})] \mbox{ with }X_t(i)=1\, \Big| \,\mathcal{F}_{\mathfrak{c}_0\TL m},\,|\mathscr{P}_m|\geq\sfrac{c\varphi\lambda ap(a,1)\mathfrak{c}_0\TL N}{720e^5}\,\Big)\\&\qquad \geq\,1-\left(1-\tfrac{\lambda p(a,1)}{4N}\right)^{\sfrac{c\varphi\lambda ap(a,1)\mathfrak{c}_0\TL N}{720e^5}}\,\geq1-\exp\left(\sfrac{c\varphi\lambda^2 ap^2(a,1)\mathfrak{c}_0\TL}{7000e^5}\right)\end{align*}
on the event $\{|\tilde{\St}_m|>c_{\eqref{recursivevertex2}}aN\ \mbox{ and }i\in\St\setminus\tilde{\St}_m\}$. From the hypothesis of the Theorem we know that $\lambda^2a p^2(a,1)\TL>M_{(iv)}$, and the constants $\mathfrak{c}_0$, $c$ and $\varphi$ do not depend on $a$, $\lambda$ 
or~$N$, so taking  $M_{(iv)}$ large enough we finally obtain
\[P\Big(\exists t\in[\mathfrak{c}_0\TL m,\mathfrak{c}_0\TL(m+\sfrac{1}{2})],\;X_t(i)=1\,\Big|\,\mathcal{F}_{\mathfrak{c}_0\TL m},\,|\mathscr{P}_m|\geq\sfrac{c\varphi\lambda ap(a,1)\mathfrak{c}_0\TL N}{720e^5}\Big)\,\geq\,\frac{2}{3}\]
on the event \smash{$\{|\tilde{\St}_m|>c_{\eqref{recursivevertex2}}aN \mbox{ and }i\in\St\setminus\tilde{\St}_m\}$}. From this point onwards the proof is the same as the one in the previous section;  defining $\tau_i$ as the first time in $[\mathfrak{c}_0\TL m,\mathfrak{c}_0\TL(m+\sfrac{1}{2})]$ that $X_{\tau_i}^i(i)=1$ we flip a coin to ask whether $i$ is $[\tau_i,\mathfrak{c}_0\TL(m+1)]$-infected which in particular implies that $i$ is infected at time $\mathfrak{c}_0\TL(m+1)$ and hence \smash{$i\in\tilde{\St}_{m+1}$}. Using \eqref{keydd} with $t=\tau_i$ we conclude that
\[\P\big(i\in\St''_{m+1}\,\Big|\,\mathcal{F}_{\mathfrak{c}_0\TL m},\,|\mathscr{P}_m|\geq\sfrac{c\varphi\lambda ap(a,1)\mathfrak{c}_0\TL N}{720e^5}\big)\,\geq\,\tfrac{1}{36}\cdot\one_{\{c_{\eqref{recursivevertex2}}aN<|\tilde{\St}_m|\,\mbox{ and } i\in\St\setminus\tilde{\St}_m\}.}\]
Finally, on the event $\{c_{\eqref{recursivevertex2}}aN<|\tilde{\St}_m|\leq4c_{\eqref{recursivevertex2}}aN\}$ we have $|\St\setminus\tilde{\St}_m|\geq \frac{aN}{4}$ and it follows from a large deviation bound together with \eqref{enoughpairs} that
\[\P\big(|\St''_{m+1}|\geq \sfrac{aN}{320} \,\big|\,\mathcal{F}_{\mathfrak{c}_0\TL m}\big)\,\geq\,\big(1-e^{-c'N}\big)\one_{\{c_{\eqref{recursivevertex2}}aN<|\tilde{\St}_m|\leq4c_{\eqref{recursivevertex2}}aN\}}\]
for some $c'>0$ independent of $N$. The result then follows by taking $c_{\eqref{recursivevertex2}}<\sfrac{1}{320}$. The corresponding lower bound for the metastable density follows the exact same proof as in the previous section so we omit it. 

\subsection{Slow extinction and optimal strategies}
\label{sec_3.4}

In this section we prove Theorem \ref{teofinal} 
in the case $\eta<0$, that is, we show how to use the results from the previous sections to deduce slow extinction and find an upper bound for the metastability exponent
for both the factor and preferential attachment kernels. More precisely,
for each kernel and survival strategy we describe the parameters $\eta, \tau$ for which the strategy is available, with the union of these parameter sets constituting the region of slow extinction. As for any strategy the metastable density is bounded from lower by $\lambda ap(a,1)$, which is increasing in~$a$, we only need to obtain the maximal value of $a$ for which the strategy succeeds; this gives the best possible lower bound for the strategy and comparing all strategies we find the optimal strategy and the corresponding upper bound on the metastability exponent. In all cases, the maximal $a$ is a power of $\lambda$ and hence the condition $a\leq e^{{1}/{\lambda}}$ appearing in Theorem \ref{theoslow} is automatically satisfied so we disregard it throughout this section.

\subsubsection{The factor kernel}\label{fkoptimal}

Recall that for the factor kernel $p(a,a)=a^{-2\gamma}$ and $p(a,1)=a^{-\gamma}$ which we use to analyse the survival mechanisms separately:
\begin{itemize}[leftmargin=*]
	\item \textbf{Quick direct spreading:} The condition ensuring the mechanism works is that $\lambda a^{1-2\gamma}>M_{(i)}$, which is possible if $\gamma>\frac{1}{2}$ with $a$ of order $a_{qds}=\lambda^{\frac{1}{2\gamma-1}+o(1)}$.\medskip
	 
	\item \textbf{Quick indirect spreading:}  The condition ensuring the mechanism works is that $\lambda^2 a^{1-2\gamma}>M_{(ii)}$, which is possible if $\gamma>\frac{1}{2}$ with $a$ of order $a_{qis}=\lambda^{\frac{2}{2\gamma-1}+o(1)}$.\medskip
	
	\item \textbf{Delayed direct spreading:} For this mechanism two conditions need to be met, 
		\begin{equation}\label{condDDS}		
\lambda^2a^{-\gamma}\geq C_{\ref{theoslow}}\log(\tfrac{1}{a\lambda})\quad\text{and}\quad
		(1\wedge\lambda a^{\gamma\eta}) a^{1-3\gamma}		\lambda^2>M_{(iii)}.\end{equation}
	The mechanism is available as soon as $\gamma>\frac{1}{3}$.
	The metastability exponent analysis needs to be handled by dividing it into different subcases.\medskip
	\begin{itemize}[leftmargin=*]
		\item Suppose the first condition in \eqref{condDDS} is the most restrictive one. Maximizing $a$ subject to this condition gives $a$ of order \smash{$a_{ls}=\lambda^{{2}/{\gamma}+o(1)}$}.
		Here the actual spreading mechanism does not affect the exponent. 
		To check that the former succeeds we still require the second condition in \eqref{condDDS} to hold. Replacing $a_{ls}$ in this condition gives
		\[(1\wedge\lambda^{1+2\eta-o(1)})\lambda^{\frac{2}{\gamma}-4-o(1)}=\lambda^{(0\vee(1+2\eta))+\frac{2}{\gamma}-4-o(1)}>M_{(iii)} \]
		We get the restrictions $\gamma\geq\frac{1}{2}$  and $\gamma\geq\frac{2}{3-2\eta}$.\pagebreak[3]\medskip
		
		\item Suppose now that the second condition in \eqref{condDDS} is the most restrictive one while at the same time $\lambda a^{\gamma\eta}\leq1$ holds. The maximal $a$ is of order		\[a_{dds}=\lambda^{\frac{3}{3\gamma-1-\gamma\eta}+o(1)}\]
		Since we now require that both $\lambda^2a^{-\gamma}\ge C_{\ref{theoslow}}$ and $\lambda a^{\gamma\eta}\leq 1$ hold, substituting the expression for $a_{dds}$ gives
		the restrictions $\gamma\leq\frac{2}{3-2\eta}$ and $\gamma>\frac{1}{3+2\eta}$.
		\medskip
		
		\item Suppose that the second condition in \eqref{condDDS} is the most restrictive one while $\lambda a^{\gamma\eta}\geq1$ holds, which is the case when depletion occurs. The maximal $a$ in the condition $\lambda^2 a^{1-3\gamma}>M_{(iii)}$ is of order
		\[a_{ddds}=\lambda^{\frac{2}{3\gamma-1}+o(1)}\]
		We now require $\lambda^2a^{-\gamma}\ge C_{\ref{theoslow}}$ and $\lambda a^{\gamma\eta}\geq 1$. Substituting the expression for $a_{ddds}$ gives
		the restrictions $\gamma\leq\frac{1}{2}$ and $\gamma\leq\frac{1}{3+2\eta}$.
		\bigskip
	\end{itemize}
	\item \textbf{Delayed indirect spreading:} As in the previous mechanism there are two conditions to be met, which after performing the same simplification as before are of the form
	\begin{equation}\label{condDIS}\lambda^2a^{-\gamma}\ge C_{\ref{theoslow}}\quad\text{ and }\quad\lambda^4 a^{1-3\gamma+\gamma\eta}\ge M_{(iv)}.\end{equation}
	In this case the condition under which slow extinction holds is \smash{$\frac{1}{3-\eta}<\gamma$} and for the metastability exponent we again treat separate scenarios:\medskip
	\begin{itemize} [leftmargin=*]
		\item Suppose that the first condition in \eqref{condDIS} is the most restrictive one, so that the maximal $a$ is of order $a_{ls}$.  Substituting this in the second condition of \eqref{condDIS} gives 
		the restriction $\gamma\geq\frac{1}{1-\eta}$.
		Recall that delayed concurrent spreading		
		is possible if both the direct and indirect spreading mechanisms succeed, which is the case if
		\smash{$\gamma\geq\frac12$} and \smash{$\gamma\ge\frac{1}{1-\eta}$}. \smallskip
		\item Suppose now that the second condition in \eqref{condDIS} is the most restrictive one, so that the maximal $a$ is of order		\[a_{dis}=\lambda^{\frac{4}{3\gamma-1-\gamma\eta}+o(1)}\]
		Substituting $a_{dis}$ in the first condition of \eqref{condDIS} gives 
		the restriction $\gamma\leq\frac{1}{1-\eta}$.\medskip
	\end{itemize}
\end{itemize}
 
 For each one of the densities above we compute the corresponding lower bound for the metastable density $\lambda ap(a,1)=\lambda a^{1-\gamma}$ and express both the result and the restrictions on the parameters in terms of $\eta$ and $\tau=\frac{1}{\gamma}+1$. The information is summarized in Table \ref{Table1}. The only thing left to do in order to conclude the first part of Theorem \ref{teofinal} is to compute which is the largest bound for the metastable density (or equivalent, which is the smallest exponent) in the region of slow extinction. This is an elementary exercise.
 
 \renewcommand{\arraystretch}{1.6}
 \begin{table}[h!]
 	\caption{Density and restrictions per mechanism for the factor kernel}
 	\label{Table1}
 	\begin{tabular}{|c|c|c|}
 		\hline
 		Mechanism & Density & Region \\ \hhline{|=|=|=|}
 		 Quick direct spreading & $\lambda^{\frac{1}{3-\tau}+o(1)}$ & $R_1:=\{\tau<3\}$ \\ \hline
 		Quick indirect spreading & $\lambda^{\frac{\tau-1}{3-\tau}+o(1)}$ & $R_1=\{\tau<3\}$ \\ \hline
 		Local survival & $\lambda^{2\tau-3+o(1)}$ & $R_2:=
 		\{\tau\leq(\tfrac{5}{2}-\eta)\wedge3\}\cup
 		\{\tau\leq2-\eta\}$ \\ \hline
 		Delayed direct spreading & $\lambda^{\frac{2\tau-2-\eta}{4-\tau-\eta}+o(1)}$ & $R_3:=\{\tfrac{5}{2}-\eta<\tau\leq4+2\eta\}$  \\ \hline
 		Delayed depleted direct spreading & $\lambda^{\frac{\tau}{4-\tau}+o(1)}$& $R_4:=\{(4+2\eta)\vee3\leq\tau<4\}$ \\ \hline
 		Delayed indirect spreading & $\lambda^{\frac{3\tau-4-\eta}{4-\tau-\eta}+o(1)}$  &  $R_5:=\{2-\eta\leq\tau<4-\eta\}$\\ \hline
 	\end{tabular}
 \end{table}
\renewcommand{\arraystretch}{1}


\subsubsection{The preferential attachment kernel}\label{paoptimal} Recall that for the preferential attachment kernel $p(a,a)=a^{-1}$ and $p(a,1)=a^{-\gamma}$. Again we analyse the survival mechanisms separately.\medskip

\begin{itemize}[leftmargin=*]
	\item \textbf{Quick direct spreading:} In this case the condition ensuring the mechanism is $\lambda >M_{(i)}$, which fails for $\lambda$ small for any choice of the parameters.\smallskip
	\item \textbf{Quick indirect spreading:}  In this case the condition ensuring the mechanism is \smash{$\lambda^2a^{1-2\gamma}>M_{(ii)}$} which was already analysed for the factor kernel. It follows that the mechanism holds as soon as $\gamma>\frac{1}{2}$ and the maximal $a$ is of order $a_{qis}$ as before.%
	\smallskip
	\item \textbf{Delayed direct spreading:} For this mechanism the simplified conditions are
	\begin{equation}\label{condDDSpa}\lambda^2a^{-\gamma}\ge C_{\ref{theoslow}}\quad\text{and}\quad(1\wedge\lambda a^{\gamma\eta})\lambda^2 a^{-\gamma}> M_{(iii)}.\end{equation}
	Surprisingly, for this kernel the mechanism is available for all values $\gamma$ and $\eta$, and we analyse the metastability exponent by considering two cases.
	\smallskip
	
	\begin{itemize}[leftmargin=*]
		\item Suppose that the first condition in \eqref{condDDSpa} is the most restrictive one, which shows that the maximal $a$ is of order $a_{ls}$. Substituting $a_{ls}$ into the second condition in \eqref{condDDSpa} 
		    gives the restriction $\eta\leq-\frac{1}{2}$.\medskip
		\item Suppose now that the second condition in \eqref{condDDSpa} is the most restrictive one and further assume that $\lambda a^{\gamma\eta}\leq1$ since otherwise we fall into the previous case. Maximizing $a$ in the condition $\lambda^3 a^{-\gamma+\gamma\eta}> M_{(iii)}$ gives the order
		\[a_{dds}=\lambda^{\frac{3}{\gamma-\gamma\eta}+o(1)}\]
		Observe that the condition $\lambda^2a^{-\gamma}\ge C_{\ref{theoslow}}$ is trivially satisfied, while replacing $a_{dis}$ into the assumption gives 
		the restriction $\eta\geq-\frac{1}{2}$. \medskip
	\end{itemize}
	\item \textbf{Delayed indirect spreading:} Since the kernel appears in the conditions only through $p(a,1)=a^{-\gamma}$ the analysis of this mechanism is the same as for the factor kernel and hence we have that the mechanism holds as soon as $\frac{1}{3-\eta}<\gamma$. 
	\begin{itemize}[leftmargin=*]
		\item If $\gamma\geq\frac{1}{1-\eta}$ then the maximal $a$ is of order $a_{ls}=\lambda^{\frac{2}{\gamma}+o(1)}$. Hence delayed concurrent spreading is possible if $\eta\leq-\frac12$ and $\gamma\geq\frac{1}{1-\eta}$.		
		\item If $\gamma\leq\frac{1}{1-\eta}$ then the maximal $a$ is of order $a_{dis}=\lambda^{\frac{4}{3\gamma-1-\gamma\eta}+o(1)}$.
	\end{itemize}
\end{itemize}
As in the previous section we compute the corresponding lower bound for the metastable density $\lambda ap(a,1)=\lambda a^{1-\gamma}$ for each mechanism and express both the result and the restrictions on the parameters in terms of $\eta$ and $\tau=\frac{1}{\gamma}+1$, which are summarized in Table~\ref{Table2}. 
\renewcommand{\arraystretch}{1.6}
\begin{table}[h!]
	\caption{Densities and restrictions per mechanism}
	\label{Table2}
	\begin{tabular}{|c|c|c|}
		\hline
		Mechanism & Density & Region \\ \hhline{|=|=|=|}
		Quick indirect spreading & $\lambda^{\frac{\tau-1}{3-\tau}+o(1)}$ & $R_1:=\{\tau\leq 3\}$ \\ \hline
		Local survival & $\lambda^{2\tau-3+o(1)}$ & 
		$R_2:= \{\tau\leq2-\eta\}\cup
		\{\eta\leq-\frac{1}{2}\}$ \\ \hline
		Delayed direct spreading & $\lambda^{\frac{3\tau-5-\eta}{1-\eta}+o(1)}$ & $R_3:=\{\eta\geq-\frac{1}{2}\}$  \\ \hline
		Delayed indirect spreading & $\lambda^{\frac{3\tau-4-\eta}{4-\tau-\eta}+o(1)}$  &  $R_4:=\{2-\eta\leq\tau<4-\eta\}$\\ \hline
	\end{tabular}
\end{table}
\renewcommand{\arraystretch}{1}

We observe that among the delayed mechanisms the density associated to local survival is the largest and hence dominates in the domain $R_2$ where it is feasible.
It is easy to calculate that delayed indirect spreading dominates 
delayed direct spreading if $\tau\leq \frac{8}{3}+\frac{\eta}{3}$ and quick indirect spreading only prevails if no other strategy is available. This completes the proof of Theorem \ref{teofinal} in the case $\eta<0$.

%

\subsection{Fast extinction phase}

Recall the distinction between fast extinction, where the expected extinction time grows as a power of $N$, and ultra-fast extinction, where it grows subpolynomially. 
Both for the factor kernel and the preferential attachment kernel ultra-fast extinction occurs if $\eta\geq\frac{1}{2}$ and $\tau>3$, which follows directly from Theorem 2 in \cite{JLM19}, which also gives a polynomial upper bound for the extinction time expectation in the fast extinction regimes for $\eta\geq0$. For $\eta<0$ an appropriate upper bound will follow from Theorem \ref{teoupper_vertex_improved} in the next section. In order to conclude that no ultra-fast extinction occurs unless 
$\eta\geq\frac{1}{2}$ and $\tau>3$, we need a lower bound for the expected extinction time, which we provide here in the case $\eta<0$.  The case $\eta\geq0$ follows from analogous ideas together with the results found in Sections 4.3 and 4.4 in \cite{JLM19}. 

\begin{lemma}\label{lastlemma}
	Let $p$ be any kernel satisfying \eqref{condp} and 
	$\eta<0$. Then for any $\varepsilon>0$ there are constants $\tilde{c}_\lambda$ and $c_{\ref{lastlemma}}$ such that for any $\lambda<\tilde{c}_\lambda$ and all $N$ large,
	\[\P\big(T_{\rm{ext}}\,\geq\,\sfrac{\lambda c_{\ref{lastlemma}} N^{\gamma(1-\eta)}}{\log N}\,\big|\,X_0\equiv1\big)=1-o(1)\]
\end{lemma}
\begin{remark}
	Observe that the exponent given does not necessarily match the one from Theorem~\ref{teoupper_vertex_improved} below. Finding the exact exponent is an interesting question which could shed light onto the behaviour in the fast extinction regime. We do not pursue this here.
\end{remark}\pagebreak[3]
\begin{proof}
Fix $i=\frac{1}{N}$ which is the most powerful vertex in the network, and observe that it is initially infected since $X_0\equiv1$. For this choice of $a$ we have from \eqref{condp} that $p(a,1)\geq \cl a^{-\gamma}=\cl N^{\gamma}$ and $\kappa(a)=\kappa_0N^{\gamma\eta}$ it then follows from Proposition \ref{survivalvertex} together with Remark \ref{re:conditions} that for any $\varepsilon>0$ there is a choice of $\tilde{c}_\lambda$ and $\mathfrak{c}_0$ such that
\[\P(i\text{ is }[0,\mathfrak{c}_0\TL]\text{-infected}\,\big|\,X_0\equiv 1,\mathfrak{F})\,\geq\, 1-\varepsilon,\]
on the event $\mathcal{A}_0$. Now, in this particular case
\[\TL=\frac{\lambda p(a,1)}{\kappa(a)\log(p(a,1))}=\frac{c\lambda N^{\gamma(1-\eta)}}{\log(N)},\]
for some constant $c$. Taking $c_{\ref{lastlemma}}=c\mathfrak{c}_0$ and observing that from \eqref{localstable} we already have that $\P(\mathcal{A}_0)>1-e^{-c_{\eqref{localstable}}N}$, the result then follows from the observation that by time $\mathfrak{c}_0\TL$ the vertex $i$ is infected and hence $T_{ext}>\mathfrak{c}_0\TL$. 
\end{proof}

\section{Upper bounds}\label{sec_upper_overall}

\subsection{Upper bound by local approximation}

Proposition~8 of~\cite{JLM22} can be easily adapted to our situation and yields the following result. 

\begin{proposition}\label{static_upperbound}
	Suppose $\eta\le 0$. There is a constant $C_{\eqref{static_upperdens}}>0$ 
	such that, for all $\lambda>0$,  the upper metastable density $\rho^+(\lambda)$ satisfies
	\begin{equation}\label{static_upperdens}
	\rho^+(\lambda)\le \left\{ \!\!
	\begin{array}{ll}
	C_{\eqref{static_upperdens}} \lambda^{\frac 1 {3-\tau}}& \!\!\mbox{in the case of the factor kernel with }\tau\le \frac 5 2,\\[2mm]
	C_{\eqref{static_upperdens}}\lambda^{2\tau-3}& \!\!\mbox{in the case of either}  \left\{\begin{array}{l}
	\mbox{the preferential attachment kernel,}\\
	\mbox{the factor kernel with }\tau> \frac 5 2.\end{array}\right.
	\end{array}
	\right.
	\end{equation}
\end{proposition}

This gives the right upper bound appearing in Theorem~\ref{teofinal} only in the regions $R_2$ introduced in Section~\ref{sec_3.4} for both kernels. The remaining regions  
	are covered in next section.

\medskip

\subsection{Upper bound by the supermartingale technique}
\label{sec: supermart}
We complete the proof of Theorem~\ref{teofinal} by showing the following two propositions.

\begin{proposition} 
	\label{supermartingale_upperbound}
	Suppose $\eta< 0$. There is a constant $C_{\ref{supermartingale_upperbound}}>0$ 
	such that the following holds:
	\begin{itemize}
		\item[(a)]
		For the factor kernel,
		\begin{itemize}
			\item[(i)] if $\tau>4-\eta$, there is fast extinction;
			\item[(ii)] if $\frac 5 2 - \eta\le \tau<4+2\eta$, the upper metastable density satisfies
			\[
			\rho^+(\lambda)\le C_{\ref{supermartingale_upperbound}} \lambda^{\frac {2\tau-2-\eta}{4-\eta-\tau}}
			\]
			\item[(iii)] If $ 3 \wedge (4+2\eta) \le \tau \le 4+\eta$, the upper metastable density satisfies
			\[
			\rho^+(\lambda)\le C_{\ref{supermartingale_upperbound}}  \lambda^{\frac \tau {4-\tau}}
			\]
			\item[(iv)] If $(2-\eta)\wedge (4+2\eta)\le \tau <4-\eta$, the upper metastable density satisfies
			\[
			\rho^+(\lambda)\le C_{\ref{supermartingale_upperbound}}  \lambda^{\frac {3\tau - 2 -\eta}{4-\tau-\eta}}
			\] 
		\end{itemize}
		
		\item[(b)]
		For the preferential attachment kernel and $-\frac 12< \eta<0$ the upper metastable density satisfies:
		\[
		\rho^+(\lambda)\le 
		\left\{
		\begin{array}{ll}
		C_{\ref{supermartingale_upperbound}}  \lambda^{\frac {3\tau-4-\eta}{4-\eta-\tau}} &\mbox{if }2-\eta\le \tau \le \frac 8 3+\frac \eta 3,\\
		C_{\ref{supermartingale_upperbound}}  \lambda^{\frac {3\tau-5-\eta}{1-\eta}} &\mbox{if }\tau\ge \frac 8 3+\frac \eta 3.
		\end{array}
		\right.!
		\]
	\end{itemize}
\end{proposition}

\begin{proposition} 
	\label{supermartingale_upperbound2}
	Suppose $\eta\ge 1/2$ and $\tau>3$. Then for the factor or preferential attachment kernel, there is ultra-fast extinction.
\end{proposition}

\smallskip

As in~\cite[Theorem~2]{JLM19} or \cite[Theorem~5]{JLM22} the basic approach for the proof of Propositions~\ref{supermartingale_upperbound} and~\ref{supermartingale_upperbound2} is to use a function $s\colon(0,1]\to[1,\infty)$ associated with the model and a functional inequality satisfied by~$s$ to define a supermartingale, which in first approximation associates the score $s(i/N)$ to each infected vertex $i$. The optional stopping theorem gives an upper bound on the time when the  total score hits zero. The main difficulties in this approach are:
\begin{itemize}[leftmargin=*]
	\item \emph{Choosing an appropriate functional inequality.} There is no ready-made recipe, although the inequalities we get,  namely~\eqref{IMI_quick}, \eqref{IMI_slow}, and~\eqref{OMIweak}--\eqref{OMIstrongslow}, are directly inspired by understanding of the local survival and spreading mechanisms studied in Section~\ref{sec:lower_bounds}.\smallskip
	\item \emph{Defining the supermartingale.} This is hard as the underlying Markov process of network evolution and infection has a large and complex state space with a sophisticated transition structure depending on the network geometry. To deal with this we introduce an exploration process which only partially reveals the network structure depending on the infection paths. Again, we are guided by our understanding of the survival mechanisms. 
\end{itemize} \medskip

Theorem~\ref{teoupper_vertex_improved} below 
allows us to prove Proposition~\ref{supermartingale_upperbound2} and to treat the case of the factor kernel in Proposition~\ref{supermartingale_upperbound} (a),(i) and (ii), and of the preferential attachment kernel in Proposition~\ref{supermartingale_upperbound} (b).
Theorem~\ref{teoupper_optimal} is considerably more difficult to prove as it introduces a new `hybrid' approach where we treat low-degree vertices locally as in a static network and include this treatment in our global supermartingale approach. The combination of local and global ideas is the main technical innovation of this paper. Theorem~\ref{teoupper_optimal} allows us to treat the case of the factor kernel in Proposition~\ref{supermartingale_upperbound} (a)-(iii) and (iv). \pagebreak[3]


\begin{definition}\label{def: piandTloc}
	Define $\pi\colon (0,1)\to (0,\infty)$ and 
	 the time-scale function $T^{\rm{loc}}\colon (0,1)\rightarrow [1,\infty)$~as 
	\begin{eqnarray}
	\pi(x)&=& \int_0^1 \frac {p(x,y)} {\kappa(x)+\kappa(y)} \, \mathrm d y,\\
	T^{\rm{loc}}(x)&=&\max\left\{8,\frac 8 {3\kappa(x)},16 \lambda^2 \frac {\pi(x)}{\kappa(x)}\right\}.
	\end{eqnarray}
\end{definition}

For given $j\in \{1,\ldots, N\}$, the time $T^{\rm{loc}}_j:=T^{\rm{loc}}(j/N)$ should be interpreted as an upper bound for the time during which the infection can survive locally around $j$ using only infections of its neighbours and direct reinfection of $j$ from these neighbours, 
but the precise definition of $T^{\rm{loc}}$ comes from the statements and proofs of Theorems~\ref{teoupper_vertex_improved} and~\ref{teoupper_optimal} below.
The definition of $T^{\rm{loc}}$ is of course to be compared with the definition of $T_{\rm{loc}}$ by~\eqref{deftloc}. In many cases we have $\pi(x)\asymp p(x,1)\asymp x^{-\gamma}$ all of the same order, so for strong vertices we will typically have \smash{$T^{\rm{loc}}_j\asymp \lambda^2 (j/N)^{-\gamma} \kappa(j/N)^{-1}\asymp T_{\rm{loc}}(j/N, \lambda)$} all of the same order, which is an indication that  in these cases $T^{\rm{loc}}$ is an accurate upper bound. 
\medskip

In Theorem~\ref{teoupper_vertex_improved} below, we will say a vertex $j$ is \emph{quick} if it updates at rate $\kappa(j/N)\ge \lambda$ and \emph{slow} if it updates at rate $\kappa(j/N)< \lambda$. When $\eta\ge 0$, we assume for simplicity $\lambda\le \kappa_0$ so all vertices are quick. When $\eta<0$, a vertex is slow if $j/N< \asl$, where $\asl$ is defined as
\begin{equation}\label{def_asl}
	\asl=\left(\frac \lambda {\kappa_0}\right)^{-\frac 1 {\gamma \eta}}.
\end{equation}

\begin{theorem}
	\label{teoupper_vertex_improved}
	Let $\eta\in\R$ and suppose there exists some $a=a(\lambda)\ge 0$ and some non-increasing function $s:[a,1]\rightarrow [1,\infty)$, or $s:(0,1]\rightarrow [1,\infty)$ if $a=0$, such that for all $x>a$, we have
	\begin{equation}\label{IMI_quick}
		 7 T^{\rm{loc}}(x) \int_0^1 \lambda p(x,y) s(y\vee a) \, \mathrm d y\le s(x) 
	\end{equation}
 	if $\eta\ge 0$ or $x\ge \asl$, or
	\begin{equation}
		\label{IMI_slow}
		7\ T^{\rm{loc}}(x) \left(\int_0^{\asl} (\kappa(x) \vee \kappa(y)) p(x,y) s(y\vee a) \, \mathrm d y+ \int_{\asl}^1 \lambda p(x,y) s(y) \, \mathrm d y\right)\le s(x)
	\end{equation}
	if $\eta<0$ and $x< \asl$.

	\begin{enumerate}
		\item If $a=0$, then the expected extinction time is at most linear in $N$ and in particular there is fast extinction. More precisely,  writing $(H)_\delta$ for the hypothesis that $\Tloc s^{-\delta}$ is a bounded function, we have
		\begin{enumerate}
			\item $(H)_1$ is a consequence of \eqref{IMI_quick} and \eqref{IMI_slow}.\smallskip
			\item If $(H)_\delta$ is satisfied for some $\delta \in (0,1]$, then there is $\omega=\omega(\lambda, \delta)$ such that for all $N$,
			\begin{equation}\label{resextime}
			\E\big[T_{\rm ext}\big]\;\leq\; \omega N^{\delta}.
			\end{equation}
			Combining with (a) we have that $\E[T_{\rm ext}]$ is always at most linear in $N$.\smallskip
			\item $(H)_0$ is satisfied exactly when $\eta\ge 1/2$, and in that case there is ultra-fast extinction. More precisely, there is $\omega=\omega(\lambda, \delta)$ such that for all $N$,
			\begin{equation}\label{logboundext}
			\E\big[T_{\rm ext}\big]\;\leq\; \omega \log N. \notag
			\end{equation}
		\end{enumerate}
		\item If $a> 0$ then there exists $\omega=\omega(a,\lambda)>0$ 
		such that, for all  large $N$  and all $t\ge0$,
		\begin{equation}\label{ineqmetastablethm3}
		I_N(t)\;\leq\;a+\frac{7}{6s(a)}\int_a^1 s(y)\,  \mathrm dy + \frac \omega t+\frac 1 N.
		\end{equation}
		In particular, if there is metastability, then the upper metastable density satisfies
		\begin{equation}
		\label{upperdens3}
		\rho^+(\lambda) \le a(\lambda)+\frac{7}{6s(a(\lambda))}\int_{a(\lambda)}^1 s(y)\, \mathrm  dy.
		\end{equation}
	\end{enumerate}
\end{theorem}

Note that Theorem~2 in~\cite{JLM19} is similar to Theorem~\ref{teoupper_vertex_improved} here, but it holds only for $\eta\ge 0$, and always requires a condition similar to~\eqref{IMI_quick} (with only a different multiplicative constant in the left-hand side of the inequality), irrespectively of whether the vertex is quick or slow. Inequality~\eqref{IMI_quick} would actually be the natural inequality to look at if we were considering a model of infection without any underlying graph, where an infected vertex $i$ could infect any other vertex $j$ with the `temporal mean rate' $\lambda p_{ij}\le \lambda p(i/N,j/N)/N$, and could recover at rate $1/(7 T^{\rm{loc}}_i)$ instead of one. It is remarkable that our model can in this sense be upper bounded by a model without any underlying graph, or more precisely where the only manifestation of the underlying geometry lies in the introduction of this `local survival time' $T^{\rm{loc}}_i$ at each vertex. The fact that Theorem~2 in~\cite{JLM19} only requires a condition like~\eqref{IMI_quick} is also an indication that this theorem takes into account adequately the degradation and regeneration effects discussed in the introduction. In a sense, our Theorem~\ref{teoupper_vertex_improved} generalizes Theorem~2 of~\cite{JLM19} with the following additional ingredients:
	\begin{enumerate}
		\item[(i)] The local time $T^{\rm{loc}}_i$ is appropriately extended to negative values of $\eta$ and thus takes into account the weaker degradation effect for slow evolutions.
		\item[(ii)] The depletion effect is now taken into account by requiring only Inequality~\eqref{IMI_slow} to be satisfied for slow vertices, instead of Inequality~\eqref{IMI_quick}.
	\end{enumerate}
Theorem~\ref{teoupper_vertex_improved} is quite strong but still has a weakness. For weak vertices with typical degree smaller than $\lambda^{-2}$, we might still have $T^{\rm{loc}}$ large, which does not reflect well the fact that reinfections should be rare and the local survival effect should be nonexistent.\smallskip

Theorem~\ref{teoupper_optimal} below provides a special treatment for these weak vertices. It requires to first distinguish weak and strong vertices, 
and then distinguish strong quick and strong slow vertices. More precisely, we define
\begin{align}\label{def_astr}
	\astr&:=\sup\Big\{x\in (0,1],\colon
	\int_0^1 p(x,y) \, \mathrm dy > \tfrac1{10\lambda^2}\Big\},\\
	\label{def_assl}
	\assl&:= \astr \wedge \asl,
\end{align}
and say a vertex $j$ is \emph{strong slow} if $j/N<\assl$, \emph{strong quick} if $\assl\le j/N<\astr$ and \emph{weak} if $\astr\le j/N$. Note there is no strong quick vertex if $\assl=\astr$.

\begin{theorem}
	\label{teoupper_optimal}
	There exists a large constant $\CMI$ such that the following holds: {Let $\eta\le 0$ and} suppose 
	there exists some $a=a(\lambda)> 0$ and some non-increasing function $s:[a,1]\rightarrow [1,\infty)$ such that the following master inequality is satisfied: for all $x>a$,
	 \begin{equation}
	\label{OMIweak}
	\CMI\ \int_0^1 \lambda p(x,y) s(y\vee a) \, \mathrm d y\le s(x) 
	\end{equation}
	if $x\ge \astr$, or
	\begin{equation}
			\label{OMIstrongquick}
			\CMI\ T^{\rm{loc}}(x) \int_0^1 \lambda\ p(x,y) s(y\vee a) \, \mathrm d y\le s(x)
	\end{equation}
	if $\assl\le x<\astr$, or
	\begin{equation}
		\label{OMIstrongslow}
		\CMI\ T^{\rm{loc}}(x) \left(\int_0^{\assl} (\kappa(x) \vee \kappa(y)) p(x,y) s(y\vee a) \, \mathrm d y+ \int_{\assl}^1 \lambda p(x,y) s(y) \, \mathrm d y\right)\le s(x)
	\end{equation}
	if $x<\assl$. Suppose also the following technical assumptions are satisfied:
	\begin{enumerate}
		\item[(H1)] 
		\null \hspace{10mm} $\displaystyle
		\label{condTechnical_optimal}
		s(\astr) \int_{y<\astr} p(x,y)  \mathrm dy \le s(x) \quad\text{ for all $x>\astr$ and }
		$
		\item[(H2)] there exists some $\alpha>0$ and $c_{\eqref{cond:H2}}>0$ not depending on $\lambda$, such that
		\begin{align} \label{cond:H2}
		s(\astr)&\le c_{\eqref{cond:H2}} \lambda^{-3+\alpha}, \\
		s(\astr)&\le c_{\eqref{cond:H2}} \lambda^{-1+\alpha} \inf \big\{\tfrac{s(x)}{T^{\rm{loc}}(x)} \colon x<\astr\big\}.
		\end{align}
	\end{enumerate}
	Suppose finally the functions $a^{-1}$ and $s(a)$ have polynomial growth 
	at 0, in the sense that they are bounded by \smash{$\lambda^{-\CMI'}$} when $\lambda\to 0$ for some $\CMI'>0$. 
	Then, there exists $\omega=\omega(\lambda)>0$ and $\omega'= \omega'(\lambda)>0$ such that, for all $N$ and all $t\ge0$, we have, for small $\lambda$,
	\begin{equation}
	\label{boundIn}
	I_N(t)\;\leq\;2a+ \frac{4}{3s(a)}\int_a^1 s(y)\, \mathrm dy +\frac \omega t+\frac {\omega'} N.
	\end{equation}
	If there is metastability, we get that the upper metastable density satisfies, for small $\lambda>0$,
	\begin{equation}
	\label{upperdens4}
	\rho^+(\lambda) \le 2 a(\lambda) +\frac{4}{3s(a(\lambda))}\int_{a(\lambda)}^1 s(y)\, \mathrm dy.
	\end{equation}
\end{theorem}

\pagebreak[3]

We postpone the proof of these theorems to Section~\ref{sec_upper_new}, and check here how they can be used to deduce Propositions~\ref{supermartingale_upperbound} and~\ref{supermartingale_upperbound2}. In order to use Theorems~\ref{teoupper_vertex_improved}~and~\ref{teoupper_optimal}, we have to determine a level $a$ and a function $s$ satisfying the hypotheses of the theorems and providing the required upper bounds. The way we do this is led by two complementary principles. First, a purely analytic approach, when the conditions required by Theorems~\ref{teoupper_vertex_improved}~and~\ref{teoupper_optimal} lead to optimal or natural choices. Second, the comparison of the approach with the lower bounds. In the lower bounds, we were seeking for the largest possible threshold $a$ such that the stars $\St:=\{1,\ldots,\lfloor aN\rfloor\}$ are typically infected in the metastable state. By contrast, in this upper bound section we seek for the smallest possible threshold $a$ such that the contact process shows a subcritical behaviour while infecting only vertices in $\{\lfloor aN\rfloor+1,\ldots, N\}$.  We thus expect that the $a$ we use in this section to be larger than the $a$ used in the lower bounds section, 
and our aim is to make them of the same order.\smallskip

To avoid cluttered notation, we henceforth assume $\beta=1$ in the definition of the kernels. Recall that $\tau\in (2,\infty)$ and $\gamma\in (0,1)$ are linked by~$\tau=1+1/\gamma$. 
To shorten computations, we write $f\lesssim g$ if the positive functions $f$ and $g$ of $\lambda$ satisfy that $f/g$ is bounded when $\lambda\to 0$, and similarly $f\gtrsim g$ or $f\asymp g$. 
We will use repeatedly that 
\smash{$\int_a^b y^{r-1} \mathrm dy \asymp a^{r}+b^{r}$} when the parameter $r$ is different from 0. Of course, when the sign of $r$ is known, only one of the two terms $a^{r}$ or $b^{r}$ has to be kept. For simplicity we disregard in our computations the cases involving a parameter $r=0$. These cases would induce a logarithmic factor coming from the integration of the harmonic function, but this logarithmic factor never concerns the leading term of the quantities we are estimating and can therefore be omitted.
\pagebreak[3]


\subsection*{Application of Theorem~\ref{teoupper_vertex_improved} to the preferential attachment kernel.}

When applying Theorem~\ref{teoupper_vertex_improved}, 
we actually check~\eqref{IMI_quick} for all $x$ irrespectively of whether $x\ge \asl$ or $x<\asl$, so as to ease computations\footnote{Checking instead the less demanding Condition~\eqref{IMI_slow} for slow vertices would actually provide slightly better upper bounds, but only in regimes where our optimal upper bound requires Theorem~\ref{teoupper_optimal} anyway.}. 
Recall that Proposition \ref{supermartingale_upperbound} assumes $\eta \in (-\frac12, 0)$, and under this hypothesis, we have for \smash{$c_3=\sfrac 1 {1- \gamma}+ \sfrac 1 {\gamma+\gamma \eta}$} and $y \in (0,1]$,
$$\pi(y)=\int_0^1 \frac  {p(y,z)}{\kappa(y)+\kappa(z)} \mathrm d z \le c_3 y^{-\gamma}.$$
Introducing $f(x)=s(x)/\Tloc(x)$ Inequality~\eqref{IMI_quick} is equivalent to  
$\Delta_a(f)(x)\le f(x)$ for all $ x\ge a,$
where we have defined the functional $\Delta_a$ by
$$ \Delta_a(f)(x):= 8 \lambda \int_0^1 p(x,y) f(y\vee a) \Tloc(y \vee a) \, \mathrm d y.$$
The assumption $\eta<0$ yields
$\Tloc(y)\asymp y^{\gamma \eta}+\lambda^2 y^{-\gamma + \gamma \eta}.$
In our pursuit to match the dominant strategies in Section \ref{paoptimal} we now suppose $a(\lambda) \lesssim \lambda^{2/\gamma}\approx a_{ls}$ (where $a_{ls}$ was introduced in Section \ref{fkoptimal}). With this assumption, when $y=a$ the last estimate simplifies to
$
\Tloc(a)\asymp \lambda^2 a^{-\gamma + \gamma \eta}.
$
In case of the preferential attachment kernel we get, for~$x\ge a$,
\begin{align*}
\Delta_a(f)(x) \asymp &\quad \lambda x^{\gamma-1} \int_0^a a^{-\gamma} f(a)\; \lambda^2 a^{-\gamma + \gamma \eta} \mathrm d y \\
& +\ \lambda x^{\gamma-1} \int_a^x y^{-\gamma} f(y) (y^{\gamma \eta}+\lambda^2 y^{-\gamma + \gamma \eta}) \mathrm d y \\
& +\ \lambda x^{-\gamma} \int_x^1 y^{\gamma-1} f(y) (y^{\gamma \eta}+\lambda^2 y^{-\gamma + \gamma \eta}) \mathrm d y.
\end{align*}
Searching for a monomial $f(x)=x^{\alpha}$ with $\alpha<0$, and assuming for simplicity that none of the four exponents $1-\gamma+\alpha+\gamma \eta$, $1-2\gamma+\alpha+\gamma \eta$, $\gamma+\alpha+\gamma \eta$, or $\alpha+\gamma \eta$, equals 0, we get
\[ 
\Delta_a(f)(x) \sim  \lambda x^{\alpha+\gamma \eta}+ \lambda^3 x^{\alpha-\gamma+\gamma\eta}+ \lambda^3 x^{\gamma -1} a^{\alpha+1-2\gamma +\gamma \eta}+ \lambda x^{-\gamma}.
\]
For $\Delta_a(f)(x)\le f(x)$ to hold for all $ x\ge a,$ we need\smallskip

\begin{tabular}{lll}
	(i) $\lambda \lesssim a^{-\gamma \eta},$ &
	(ii) $\lambda^3 \lesssim a^{\gamma (1-\eta)},$ &
	(iii) $ \lambda^3 \lesssim a^{2 \gamma - 1 - \alpha-\gamma \eta},$ \\[1mm]
	(iv) $\lambda \lesssim a^{\alpha+\gamma}$, &
	(v) $\lambda^4 \lesssim a^{3\gamma-1 -\gamma \eta}.$ & \\
\end{tabular}
\smallskip

\noindent
Here (v) is just a consequence of  (iii) and (iv). (i), (ii) and (v) do not depend on $\alpha$, and prevent us from choosing $a$ too small. We now consider the different cases.
\medskip

\pagebreak[3]

\noindent
{\bf (1)\ The case $-\frac 12<\eta\le0$ and $\tau>\frac 8 3+ \frac \eta 3$.}

\smallskip
\noindent
This is the  delayed direct spreading phase, where  
(i) and (v) are less restrictive than (ii), leading naturally to the choice (with $r$ some constant to be defined later)
$$a(\lambda) = r \lambda^{\sfrac 3 {\gamma(1-\eta)}}.$$
(iii) requires $\alpha\ge \gamma -1$, while (iv) requires $\alpha\ge \frac{\gamma}{3}(\eta-2)$ which is less restrictive under our assumptions, so we choose $\alpha= \gamma -1$, as~\eqref{upperdens3} will then give the best possible upper bound for the upper metastable density. 
 It is now a simple verification that choosing $r$ large enough indeed ensures~\eqref{IMI_quick} is satisfied for small $\lambda>0$. Then~\eqref{upperdens3} yields
\[
\rho^+(\lambda)\lesssim \frac 1 {s(a(\lambda))}\lesssim \frac 1 {a^{\gamma- 1} \lambda^2 a^{-\gamma + \gamma \eta}} \lesssim \lambda^{\frac {3-2\gamma -\gamma \eta} {\gamma -\gamma \eta}} =\lambda^{\frac{3\tau-5-\eta}{1-\eta}}.
\]

\smallskip
\noindent
{\bf (2)\ The case $-\frac 12<\eta\le0$ and $2-\eta \le \tau\le \frac 8 3+ \frac \eta 3$.}

\smallskip
\noindent
This is the delayed indirect spreading phase, 
where (i) and (ii) are less restrictive than (v), leading naturally to the choice (with $r$ some constant to be defined later)
$$a(\lambda) = r \lambda^{\sfrac 4 {\gamma(3-\eta)-1}},$$
(iii) and (iv) is more restrictive than (iii), thus leading to the choice 
$\alpha= -\frac {\gamma + \gamma \eta + 1} 4.$
Choosing $r$ large we see that~\eqref{IMI_quick} is satisfied for small $\lambda>0$. Furthermore 
\smash{$a\lesssim \lambda^{2/\gamma}$}  holds as $\gamma<\frac 1 {1- \eta}.$
Finally, \eqref{upperdens3} yields
\[
\rho^+(\lambda)\lesssim \lambda^{\frac {3-\gamma -\gamma \eta} {3\gamma -\gamma \eta -1}}=\lambda^{\frac{3\tau-4-\eta}{4-\eta-\tau}}.
\]

\subsection*{Application of Theorem~\ref{teoupper_vertex_improved} to the factor kernel.}

{ As discussed previously, for the factor kernel we use Theorem~\ref{teoupper_vertex_improved} to prove the upper bounds in Proposition~\ref{supermartingale_upperbound} only in the cases (a),(i) and (ii), so we we either assume $\frac{5}{2}-\eta\leq\tau<4+2\eta$ or $4-\eta<\tau$}.  As in the previous analysis we check for Condition~\eqref{IMI_quick}.
{Replacing $p(x,y)=x^{-\gamma}y^{-\gamma}$, }the left side of this inequality  
factorizes as $D_a \Tloc(x) x^{-\gamma}$, where $D_a$ does not depend on $x$ and is given by
$$D_a:= 8\lambda \int_0^1 y^{-\gamma} s(a \vee y) \mathrm d y.$$
It is thus natural to choose the scoring function $s(x)= \Tloc(x) x^{-\gamma},$ so that~\eqref{IMI_quick} becomes equivalent to $D_a\le 1$.
{From our assumptions we have $2-\eta<\tau$}, so that
$\pi(x)=O(x^{-\gamma})$ and we get
\begin{equation}\label{ineqfkT3} D_a\lesssim \lambda+ \lambda a^{1- 2 \gamma+\gamma \eta}+\lambda^3 a^{1- 3 \gamma+\gamma \eta},\end{equation}
assuming for simplicity that the exponents $1-2\gamma+\gamma \eta$ and $1-3\gamma+\gamma \eta$ are nonzero. 
To discuss which values of $a$ can make $D_a$ smaller than 1, we consider different cases.\pagebreak[3]\smallskip

\noindent
{\bf (1)\ The case $\tau>4-\eta$.}

\smallskip
\noindent
{In this case the two exponents on the right-hand side of \eqref{ineqfkT3} are positive so }we have $D_0 \lesssim c \lambda$ with finite $c$, and in particular~\eqref{IMI_quick} is satisfied for $a=0$ and $\lambda\le 1/c$, yielding fast extinction. Moreover, the reader might check that the hypothesis $(H)_\delta$ is satisfied for $\delta\ge \frac {1-\eta} {2-\eta}$. It follows that, for $\lambda\le 1/c$, the expected extinction time is at most \smash{$\omega N^{\frac {1-\eta} {2-\eta}},$}
for some $\omega=\omega(\lambda)<\infty$.%
\smallskip%
\pagebreak[3]

\noindent
{\bf (2)\ The case $\tau=4-\eta$.}

\smallskip
\noindent
We get $D_a\sim \lambda + \lambda^3 \log(1/a)$ and so
$D_a\le 1$ for 
$
a(\lambda)= e^{-r \lambda^{- 3}},
$
for well-chosen~$r$ and $\lambda$ small. Thus we get an exponentially small upper bound for the upper metastable density.

\smallskip
\noindent
{\bf (3)\ The case $\frac 5 2-\eta \le \tau\le 4+2\eta$.}

\smallskip
\noindent
In this region the largest term of \eqref{ineqfkT3} is $\lambda^3a^{1-3\gamma+\gamma\eta}$ so} we get $D_a\le 1$ for small $\lambda$ by taking \smash{$a(\lambda)= r \lambda^{\frac 3 {3\gamma-1-\gamma \eta}},$}
for a well-chosen $r$.  \eqref{upperdens3} then gives us the upper bound for the upper metastable density as
\begin{eqnarray*}
	\rho^+(\lambda)&\lesssim& a(\lambda)+\frac 1 {s(a(\lambda))}\int_{a(\lambda)}^1 s(y) \mathrm d y 
	\lesssim \frac 1 {s(a(\lambda))} \lesssim \lambda^{\frac {2-\gamma\eta}{3\gamma-1-\gamma \eta}}=\lambda^{\frac{2\tau-2-\eta}{4-\tau-\eta}}.
\end{eqnarray*}
Notice that this bound is available in the full region $\frac 5 2-\eta \le \tau< 4-\eta$, yet it does not match the lower bounds obtained in Section \ref{fkoptimal} when $4+2\eta\leq\tau<4-\eta$. In the next section we use Theorem~\ref{teoupper_optimal} to obtain the matching upper bounds we are missing here.

\subsection*{Application of Theorem~\ref{teoupper_optimal} to the factor kernel}

We consider the vertex updating model for the factor kernel $p(x,y)=x^{-\gamma} y^{-\gamma}$ in the region delimited by the inequalities $\tau>3$, $\tau>4+2\eta$, $\tau>2-\eta$ and $\tau<4-\eta$. 
The inequality $\tau>2-\eta$ implies $\pi(x)\asymp x^{-\gamma}$. We also have $\astr\asymp \lambda^{2/\gamma}$. For $\eta>-1/2$, we have $\assl\asymp \lambda^{1/-\gamma \eta} <\astr$ while for $\eta<-1/2$, we have $\assl= \astr\asymp \lambda^{2/\gamma}$. We now search for some $a<\assl$ and function $s$ satisfying Conditions~\eqref{OMIweak},~\eqref{OMIstrongquick} and~\eqref{OMIstrongslow} as well as the Hypotheses $(H1)$ and $(H2)$.
Let us introduce $A_1$, $A_2$ and $A_3$ the following integrals over the function $s$:
$$
A_1=\int_0^{\assl} \lambda y^{-\gamma} s(y\vee a) dy, \, \,
A_2= \int_0^{\assl} \kappa(y) y^{-\gamma} s(y\vee a) dy, \, \,
A_3= \int_{\assl}^1 \lambda y^{-\gamma} s(y) dy.
$$
{Recalling the form of the factor kernel $p(x,y)=x^{-\gamma}y^{-\gamma}$ and bounding the max $\kappa(x)\wedge \kappa(y)$ by the sum, we see that the inequalities \eqref{OMIweak}, \eqref{OMIstrongquick} and \eqref{OMIstrongslow} are guaranteed by
	\begin{align*}
		\CMI  (A_1+A_3)x^{-\gamma} \le s(x), &\qquad \mbox{ for } x\ge \astr,\\
		\CMI  (A_1+A_3)x^{-\gamma}T^{\rm{loc}}(x) \le s(x), &\qquad \mbox{ for } \assl\le x<\astr,\\
		\CMI  \left(\frac {\kappa(x)}\lambda A_1+ A_2+A_3\right)x^{-\gamma}T^{\rm{loc}}(x)\le s(x),&\qquad \mbox{ for } x<\assl.
	\end{align*}
A natural choice for the scoring function $s$ is then given by
\[
s(x)= \left\{
\begin{array}{ll}
	x^{-\gamma} &\quad \mbox{ for } x\ge \astr \\
	x^{-\gamma} T^{\rm{loc}}(x) &\quad \mbox{ for } \assl\le x< \astr, \\
	\frac {\kappa(x)}\lambda x^{-\gamma} T^{\rm{loc}}(x) + \rho x^{-\gamma}T^{\rm{loc}}(x) &\quad \mbox{ for } x<\assl,
\end{array}
\right.
\]
leading naturally to the requests
\[
\left\{
\begin{array}{l}
	\CMI (A_1+A_3)\le 1, \\
	\CMI (A_2+A_3)\le \rho,
\end{array}
\right.
\]
}
which in turn are implied by the three requests
\[
A_1\le \frac 1 {2\CMI},\quad A_2\le \frac \rho {2\CMI},\quad A_3\le \frac {1\wedge \rho}{2 \CMI}.
\]
Before computing these terms, we check that the hypotheses $(H1)$ and $(H2)$ are satisfied with this choice of scoring function. The hypothesis $(H1)$ is equivalent to
\[
{\astr}^{-\gamma}\int_0^{\astr} x^{-\gamma} y^{-\gamma} \mathrm dy \le x^{-\gamma},  \qquad \mbox{ for } x>\astr,
\]
which is in turn equivalent to $\astr^{1-2\gamma}/(1-\gamma)\le 1$, and is satisfied for small $\lambda$ as we lie in the phase $\tau>3$ or equivalently $\gamma<1/2$. Furthermore, we have
\[s(\astr)=\astr^{-\gamma}\asymp \lambda^{-2},\]
so the first part of the hypothesis $(H2)$ is satisfied, with $\alpha=1$. We will check the second part of the hypothesis $(H2)$ later on, while we now bound the terms $A_1$ to $A_3$, up to multiplicative constants depending only on $\gamma$ and $\eta$. First, observe that since $\assl\le \astr$ we have that $T^{\rm{loc}}(x)\asymp\lambda^2x^{-\gamma+\gamma\eta}$,  for $x<\assl$, and hence by our choice of $s$ and the definition of $\assl$,
\begin{align*}
A_1&\asymp \lambda a^{1-\gamma} s(a)+ \rhoMI \int_a^{\assl} \lambda^3 y^{-3\gamma+\gamma \eta} dy +  \int_a^{\assl} \lambda^2 y^{-3\gamma} dy\\
&\asymp \rhoMI  \lambda^3 a^{1-3 \gamma +\gamma \eta} + \lambda^2 a^{1-3 \gamma} + \lambda^2 \assl^{1-3\gamma},
\end{align*}
using that $1-3\gamma+\gamma \eta= (\tau-4+\eta)\gamma<0$ for $\tau<4-\eta$, and assuming for simplicity~$1-3\gamma\ne 0$.%
\smallskip%

When $\eta\ge -1/2$, we have $\assl\asymp \lambda^{-1/\gamma \eta}$ and therefore we get that 
\smash{$\lambda^2 \assl^{1-3\gamma}\asymp \lambda^{\frac {\tau-4-2\eta}{-\eta}}$}, and this term tends to 0 as $\lambda$ tends to 0 as $\tau>4+2\eta$. When $\eta<-1/2$, we have $\assl\asymp \lambda^{2/\gamma}$ and thus the term
$\lambda^2 \assl^{1-3\gamma}\asymp \lambda^{2\tau-6}$ also goes to 0 as $\tau>3$. Thus, to guarantee the inequality $A_1\le 1/2\CMI$ for small $\lambda$, it suffices to satisfy the following requests:
\begin{eqnarray}
\label{request1}\lambda^2 a^{1-3 \gamma} &\lesssim& 1,\\
\label{requestA4}\lambda^3 a^{1-3 \gamma +\gamma \eta} &\lesssim&1/\rhoMI,
\end{eqnarray}
with an abuse of notation, as for these requests to make sense, the involved multiplicative constants should be specified. In other words, the request $\lesssim 1$ should be understood as $\le \mathfrak{c}(\gamma, \eta)$, where $\mathfrak{c}(\gamma,\eta)>0$ is some constant which may change from one inequality to the other, and which could be made explicit if necessary.\smallskip

Similarly, we obtain
\[
A_2\asymp \rhoMI(\lambda^2 a^{1-3\gamma}+\lambda^2 \assl^{1-3\gamma}) + (\lambda a^{1-3 \gamma - \gamma \eta} +\lambda \assl^{1-3 \gamma - \gamma \eta}), 
\]
leading to the requests~\eqref{request1} and~\eqref{requestA4} again, as well as the additional requests
\begin{equation}
\label{requestA2}\lambda a^{1-3 \gamma - \gamma \eta} +\lambda \assl^{1-3 \gamma - \gamma \eta} \lesssim \rhoMI.
\end{equation}
To bound $A_3$, we first treat the case $\eta<-1/2$ for which \smash{$A_3= \lambda \int_{\assl}^1 y^{-2\gamma} \mathrm dy\asymp \lambda$} as we have $\tau>3$ or equivalently $\gamma<1/2$. In the case $\eta>-1/2$, we obtain
\begin{align*}
A_3&\asymp \lambda +(1+\rho) \lambda^3 \assl^{1-3\gamma+ \gamma \eta}\\
& \asymp \lambda + (1+\rho) \lambda \assl^{1-3\gamma-\gamma \eta},
\end{align*}
using the value of $\assl$ in that case. Assuming further $\rho\le 1$, the original request for $A_3$ now reads 
\[\lambda + \lambda\assl^{1-3\gamma-\gamma \eta}\lesssim \rho\]
where the bound on $\lambda\assl^{1-3\gamma-\gamma \eta}$ already follows from~\eqref{requestA2} and thus we only require
\begin{equation}
	\label{requestadd} \lambda \lesssim \rho.
\end{equation}
Thus far we have several requests, namely, \eqref{request1}, \eqref{requestA4}, \eqref{requestA2}, and \eqref{requestadd}. In order to make these inequalities independent of $\rho$ we  combine \eqref{requestadd} with inequality~\eqref{requestA4}, obtaining
\begin{equation}
\label{request2} \lambda^4 a^{1-3\gamma+ \gamma \eta}\lesssim 1,
\end{equation}
and also combine inequalities \eqref{requestA4} and \eqref{requestA2}, which gives
\begin{align*}
\lambda^4 a^{2-6\gamma} &\lesssim 1,\\
\lambda^4 a^{1-3\gamma +\gamma \eta} \assl^{1-3\gamma-\gamma \eta}&\lesssim 1.
\end{align*}
These last two inequalities are redundant with~\eqref{request1} and~\eqref{request2}; the first of these requests amounts to~\eqref{request1}, while the second is less restrictive than either~\eqref{request1} or~\eqref{request2} depending on the sign of $1-3\gamma-\gamma \eta$ (using the fact that $a\leq\assl$).\smallskip

\pagebreak[3]

We are thus left with the two main requests~\eqref{request1} and~\eqref{request2}. The other requests only yield restictions on the choice of $\rho$, which we should check are compatible with $\rho\le 1$. Which of~\eqref{request1} or~\eqref{request2} is the most demanding depends on whether $\tau$ is larger or smaller than $4+\eta$, and we now treat these two cases.
\begin{itemize}[leftmargin=*]
	\item $\tau< 4+\eta$. In that case the most restrictive request is~\eqref{request1}, leading to the choice
	\[
	a\asymp \lambda^{\frac 2 {3\gamma-1}}.
	\]
	Moreover, $\rho$ should satisfy
	\[
	\lambda\lesssim \rho \lesssim \lambda^{-3} a^{-1+3\gamma+\gamma \eta} \asymp \lambda^{\frac {\tau-4-2\eta}{4-\tau}}\asymp \lambda^{1-2 \frac {4+\eta-\tau}{4-\tau}}.
	\]
	In particular, the choice $\rho\asymp \lambda^{1-2 \frac {4+\eta-\tau}{4-\tau}}$ allows to satisfy the second part of $(H2)$ with $\alpha=2\frac {4-\tau+\eta}{4-\tau} \in (0,2).$
	Pushing further the computations we get
	\[\frac 1 {s(a)}\int_{a}^1 s(y)\mathrm dy \asymp \frac 1 {s(a(\lambda))}
	\asymp \lambda^{\frac \tau {4-\tau}}.
	\]
		We thus obtain an upper bound for the metastable density as
	\[
	\rho^+(\lambda)\lesssim \lambda^{\frac \tau {4-\tau}}.
	\]
	\item $\tau\ge 4+\eta$. In that case the most restrictive request is~\eqref{request2}, leading to the choice
	\[
	a\asymp \lambda^{\frac 4 {3\gamma-1-\gamma \eta}}
	\]
	In that case we have $\rho\asymp \lambda$, which does not allow to satisfy the hypothesis $(H2)$ for $\alpha>0$. The first term in the definition of $s(x)$ gives another bound, but this also does not always allows to satisfy $(H2)$. However, we can modify slightly the definition of $s(x)$ so that $(H2)$ is satisfied, by adding the term 	\smash{$\lambda^{1-\alpha} \astr^{-\gamma} T^{\rm{loc}}(x)$} when $x<\astr$,
	for some $\alpha>0$, which makes $(H2)$ automatically satisfied. This new term changes the values of the integrals by adding new terms $A'_1$, $A'_2$ and $A'_3$ to the previous ones, however the new inequalities that we request remain almost unchanged
		\[
\left\{
\begin{array}{l}
	\CMI ((A_1+A'_1)+(A_3+A_3'))\le 1 \\
	\CMI ((A_2+A_2')+(A_3+A'_3))\le \lambda,
\end{array}
\right.
\]
since it is enough to satisfy the inequalities for the previous score function (because the new one is bigger). 
We thus obtain the additional requests $A'_1\lesssim 1$ and $A'_2,A'_3\lesssim \lambda$. Computing the integrals yields
\begin{align*}
A'_1 & \asymp \lambda^{4-\alpha} \astr^{-\gamma}  \left(a^{1-2\gamma+\gamma\eta}+ 
\assl^{1-2\gamma+\gamma \eta}\right)\\
&\asymp \lambda^{2-\alpha} \left(a^{1-2\gamma+\gamma\eta}+ 
\assl^{1-2\gamma+\gamma \eta}\right).
\end{align*}
By the value of $a$, the first term is $\lambda$ to the power \smash{$\sfrac {2(\tau-2+\eta)}{4-\tau-\eta}-\alpha$}, leading to the condition \smash{$\alpha<\sfrac {2(\tau-2+\eta)}{4-\tau-\eta}$}. For the second term, we split the analysis between the case $\eta\ge -1/2$, with $\assl\asymp \lambda^{-1/{\gamma \eta}}$, leading to \smash{$\alpha< \sfrac {\tau-3-\eta}{-\eta}$}, and the case $\eta<-1/2$, with $\astr \asymp \lambda^{2/\gamma}$, leading to the request $\alpha<2(\tau-2+\eta)$.
Similarly, we compute
\[
A'_2\asymp \lambda^{3-\alpha} \astr^{-\gamma} \assl^{1-2\gamma}\asymp \lambda^{1-\alpha}\assl^{1-2\gamma}.
\]
Splitting again the analysis in the two cases $\eta\ge -1/2$ and $\eta<-1/2$, we see that the request $A'_2\lesssim \lambda$ is guaranteed by asking $\alpha<\sfrac {\tau-3}{-\eta}$ in the first case, and $\alpha<2(\tau-3)$ in the second case.\smallskip

Finally, the term $A'_3$ is present only in the case $\eta>-1/2$, where $\astr>\assl\asymp \lambda^{1/-\gamma \eta}$, and we then have 
\[
A'_3\asymp \lambda^{2-\alpha}\assl^{1-2\gamma+\gamma \eta}\asymp \lambda^{1-\alpha}\assl^{1-2\gamma}\asymp A'_2,
\]
thus not leading to any further request. There is now no difficulty choosing $\alpha>0$ small satisfying all the requests, and pushing further the computations we then get
\[\frac 1 {s(a)}\int_{a}^1 s(y)\mathrm dy \asymp \frac 1 {s(a(\lambda))}\asymp \lambda^{\frac {3\tau-4-\eta} {4-\tau-\eta}}.
\]
We thus obtain an upper bound for the metastable density as
$
\rho^+(\lambda)\lesssim \lambda^{\frac {3\tau-4-\eta} {4-\tau-\eta}}.
$
\end{itemize}

\section{Proof of Theorems~\ref{teoupper_vertex_improved}~and~\ref{teoupper_optimal}}
\label{sec_upper_new}

\subsection{Proof of Theorem~\ref{teoupper_vertex_improved}}\label{sec:Thm3}

We now provide the proof of Theorem~\ref{teoupper_vertex_improved}, using settings and notations that will be reusable in the proof of Theorem~\ref{teoupper_optimal}. In particular, we rewrite~\eqref{IMI_quick} as
\begin{align*}
		\CMI\ T^{\rm{loc}}(x) \int_0^1 \lambda p(x,y) s(y\vee a) \, \mathrm d y&\le s(x)
\end{align*}
with of course $\CMI=7$, and similarly for ~\eqref{IMI_slow}.
We suppose $a(\lambda)$ and $s(x)$ are given satisfying the hypotheses of Theorem~\ref{teoupper_vertex_improved}. When $a>0$, we extend $s$ to $(0,1]$ by $s(x)=s(x\vee a)$. The theorem also involves other functions defined on $(0,1]$, namely $\kappa$, $\pi$ and $\Tloc$.
We use the subscript notation when considering the corresponding functions defined on $\{1,\ldots, N\}$
, so
$\kappa_j=\kappa(j/N)$, $\pi_j=\pi(j/N)$, $\Tloc_j=\Tloc(j/N)$, and $s_j=s(j/N)$. Recall the notation \smash{$p_{i,j}=\sfrac1N {p(i/N,j/N)} \wedge 1$}.%
\medskip%

The fact that $(H)_1$ follows from \eqref{IMI_quick} and \eqref{IMI_slow} 
(when $a=0$) is immediate as then, using the monotonicity of $p$,
\[
\frac {s(x)}{\Tloc(x)}\ge \CMI \int_0^1 (\lambda \wedge \kappa(y)) p(1,y)s(y\vee a) \, \mathrm d y >0.
\]
Moreover, $(H)_0$ is  satisfied when $\eta\ge 1/2$ since
\[
\frac {\pi(x)}{\kappa(x)}\le \frac 1 {\kappa(x)^2}\int_0^1 p(x,y) \mathrm d y\le \frac {\cu}{\kappa_0^2}x^{-\gamma+2 \gamma \eta},
\]
where the second inequality follows from \eqref{condp}. Using again \eqref{condp} we deduce that for $\eta<1/2$ it cannot hold since
\[
\frac {\pi(x)}{\kappa(x)}\ge \frac {\cl x^{-\gamma}}{2\kappa(x)\big(\kappa(x)+ \min(\kappa(y), y\in[1/2,1]) \big)}.
\]
If $a>0$ we call the vertices $j<aN$ stars. The proof attaches a score to a configuration that behaves as a supermartingale as long as no star gets infected. We stop the process when infecting a star, thus before any edge between stars can be used to transmit the infection. It is convenient and clearer to consider in the following that there are simply no edges between stars, for example by redefining $p_{ij}=0$ when $i$ and $j$ are stars.\medskip
	
Recall the definition of slow vertices just before Theorem~\ref{teoupper_vertex_improved}, and in particular the definition of $\asl$ in~\eqref{def_asl}. In Theorem~\ref{teoupper_vertex_improved}, the request~\eqref{IMI_slow} seems to indicate that the infection propagates between slow vertices at rate 
$\kappa(x)\vee \kappa(y)$ instead of $\lambda$, which amounts to take into account the depletion effect, and is a significant difficulty in this theorem. We say an edge $\{i,j\}$ (not necessarily in $G$) is slow if $i$ and $j$ are both slow vertices, otherwise we refer to the edge as quick. So a slow edge updates at a rate bounded by $2\lambda$, and a quick edge at rate at least $\lambda$. In the following, we suppose $a<\asl$. When this does not hold, the proof still works and is actually easier as there are no slow edges.\medskip

The proof relies on the following ideas, introducing a modification of the entire process (which alters the underlying network, its dynamics and the evolution of the infection) which constitutes a stochastic upper bound:
	
	\begin{itemize}[leftmargin=*]\setlength\itemsep{1ex}
		\item For the new underlying network keep the original set of vertices $\{1,\ldots,N\}$, maintaining their previous classification as either slow or quick, as well as the definition of stars. We also maintain all quick edges, consisting of edges with at least one quick vertex as an endpoint.
		\item  As for the remaining edges (between slow vertices), we replace them with their directed versions so for each pair $i,j$ of slow vertices we consider $(i,j)$ and $(j,i)$ as possible edges of the network (even at the same time) instead of $\{i,j\}$.
		\item Each vertex $i$ still updates at rate $\kappa_i$, however updatings affect the edges differently according to their type. {At this point it} is enough to keep in mind that updatings maintain the role of losing the local neighbourhood of vertices.
		\item The new process is still a contact process at its core, where each vertex is classified as either healthy or infected, and can be infected, recover, and be reinfected. Vertices recover at rate $1$ independently from one another as in the original process, however the rate at which the infection is transmitted along each edge is more {intricate}.
		\item A new addition to the process is that we also subclassify infected vertices as \emph{saturated} and \emph{unsaturated}. The reader might think of a vertex as saturated whenever it is surrounded by many neighbours (or maybe just powerful ones), likely to reinfect it after its recovery.
		For this reason for saturated vertices we just negate recoveries altogether, and hence these become beacons of infection that \emph{cannot recover}. In order for saturated vertices to disappear they have to first update and observe a more favourable environment to become unsaturated, and thus be available for recovery. The different conditions that result in saturation will be given gradually as we explain below the evolution of the process.
		\item We begin by describing first the rules affecting slow vertices and slow (directed) edges, {see Figure~\ref{figu2} for an example}. By definition, for the original process the subnetwork of slow vertices evolves at a slower pace than the contact process and hence is seen as almost static by the infection, which explores it as it spreads. We adopt the point of view of the infection and thus classify slow edges as either \emph{unrevealed}, \emph{present} or, \emph{absent} (all mutually exclusive). A fourth label, \emph{saturated}, will be helpful when dealing with saturated vertices, but we leave the intuition behind this label for later. Slow edges can change their labels through updates or infections of slow vertices:
		
		\begin{itemize}[leftmargin=*]\setlength\itemsep{1ex}
			\item Every slow vertex $i$ updates at rate $\kappa_i$, and upon updating a two-step mechanism occurs:
			\begin{itemize}[leftmargin=*]
				\item First, all slow edges $(i,j)$ and $(j,i)$ become unrevealed regardless of their previous status.
				\item Next, if $i$ is infected at the time of the update, we \emph{resample} each edge $(i,j)$, declaring it \emph{present} or \emph{absent} by tossing a coin with probability $p_{i,j}$, independently from all other edges. We do the same for each edge $(j,i)$ for which $j$ is infected at the time of the update.
			\end{itemize}
			\item If a slow vertex $i$ becomes infected, we \emph{resample} each unrevealed edge $(i,j)$ leaving $i$ {in the same manner} as before. Observe that this procedure does not affect already revealed edges .
			
			\begin{figure}[h!]
				\scalebox{0.7}{
			\begin{tikzpicture}[x=0.75pt,y=0.75pt,yscale=-1,xscale=1]
				
				\draw   (69.17,216.92) .. controls (69.17,208.68) and (75.85,202) .. (84.08,202) .. controls (92.32,202) and (99,208.68) .. (99,216.92) .. controls (99,225.15) and (92.32,231.83) .. (84.08,231.83) .. controls (75.85,231.83) and (69.17,225.15) .. (69.17,216.92) -- cycle ;
				\draw   (50.17,116.92) .. controls (50.17,108.68) and (56.85,102) .. (65.08,102) .. controls (73.32,102) and (80,108.68) .. (80,116.92) .. controls (80,125.15) and (73.32,131.83) .. (65.08,131.83) .. controls (56.85,131.83) and (50.17,125.15) .. (50.17,116.92) -- cycle ;
				\draw  [fill={rgb, 255:red, 155; green, 155; blue, 155 }  ,fill opacity=1 ] (129.17,61.92) .. controls (129.17,53.68) and (135.85,47) .. (144.08,47) .. controls (152.32,47) and (159,53.68) .. (159,61.92) .. controls (159,70.15) and (152.32,76.83) .. (144.08,76.83) .. controls (135.85,76.83) and (129.17,70.15) .. (129.17,61.92) -- cycle ;
				\draw  [fill={rgb, 255:red, 155; green, 155; blue, 155 }  ,fill opacity=1 ] (170.17,154.92) .. controls (170.17,146.68) and (176.85,140) .. (185.08,140) .. controls (193.32,140) and (200,146.68) .. (200,154.92) .. controls (200,163.15) and (193.32,169.83) .. (185.08,169.83) .. controls (176.85,169.83) and (170.17,163.15) .. (170.17,154.92) -- cycle ;
				\draw  [dash pattern={on 0.84pt off 2.51pt}]  (65.08,131.83) .. controls (61.08,161.23) and (57.16,178.15) .. (72.06,206.11) ;
				\draw [shift={(73,207.83)}, rotate = 241.11] [color={rgb, 255:red, 0; green, 0; blue, 0 }  ][line width=0.75]    (10.93,-3.29) .. controls (6.95,-1.4) and (3.31,-0.3) .. (0,0) .. controls (3.31,0.3) and (6.95,1.4) .. (10.93,3.29)   ;
				\draw  [dash pattern={on 0.84pt off 2.51pt}]  (77,202.83) .. controls (71.07,228.02) and (85.26,151.81) .. (66.59,133.15) ;
				\draw [shift={(65.08,131.83)}, rotate = 37.41] [color={rgb, 255:red, 0; green, 0; blue, 0 }  ][line width=0.75]    (10.93,-3.29) .. controls (6.95,-1.4) and (3.31,-0.3) .. (0,0) .. controls (3.31,0.3) and (6.95,1.4) .. (10.93,3.29)   ;
				\draw    (144.08,76.83) .. controls (145.93,110.61) and (159.97,129.48) .. (175.32,139.75) ;
				\draw [shift={(177,140.83)}, rotate = 212.01] [color={rgb, 255:red, 0; green, 0; blue, 0 }  ][line width=0.75]    (10.93,-3.29) .. controls (6.95,-1.4) and (3.31,-0.3) .. (0,0) .. controls (3.31,0.3) and (6.95,1.4) .. (10.93,3.29)   ;
				\draw  [dash pattern={on 0.84pt off 2.51pt}]  (75,103.83) .. controls (98.64,115.65) and (109.67,90.6) .. (128.31,63.17) ;
				\draw [shift={(129.17,61.92)}, rotate = 124.47] [color={rgb, 255:red, 0; green, 0; blue, 0 }  ][line width=0.75]    (10.93,-3.29) .. controls (6.95,-1.4) and (3.31,-0.3) .. (0,0) .. controls (3.31,0.3) and (6.95,1.4) .. (10.93,3.29)   ;
				\draw  [dash pattern={on 0.84pt off 2.51pt}]  (99,216.92) .. controls (130.52,214.86) and (152.34,193.49) .. (177.83,170.87) ;
				\draw [shift={(179,169.83)}, rotate = 138.5] [color={rgb, 255:red, 0; green, 0; blue, 0 }  ][line width=0.75]    (10.93,-3.29) .. controls (6.95,-1.4) and (3.31,-0.3) .. (0,0) .. controls (3.31,0.3) and (6.95,1.4) .. (10.93,3.29)   ;
				\draw    (170.17,154.92) .. controls (138.16,135.54) and (97,120.89) .. (79.78,126.18) ;
				\draw [shift={(78,126.83)}, rotate = 336.37] [color={rgb, 255:red, 0; green, 0; blue, 0 }  ][line width=0.75]    (10.93,-3.29) .. controls (6.95,-1.4) and (3.31,-0.3) .. (0,0) .. controls (3.31,0.3) and (6.95,1.4) .. (10.93,3.29)   ;
				\draw  [dash pattern={on 0.84pt off 2.51pt}]  (78,126.83) .. controls (96.72,141.61) and (133.86,155.41) .. (168.58,154.95) ;
				\draw [shift={(170.17,154.92)}, rotate = 178.51] [color={rgb, 255:red, 0; green, 0; blue, 0 }  ][line width=0.75]    (10.93,-3.29) .. controls (6.95,-1.4) and (3.31,-0.3) .. (0,0) .. controls (3.31,0.3) and (6.95,1.4) .. (10.93,3.29)   ;
				\draw    (138.68,73.74) .. controls (115.96,103.71) and (89.35,181.35) .. (92.61,201.16) ;
				\draw [shift={(93,202.83)}, rotate = 252.33] [color={rgb, 255:red, 0; green, 0; blue, 0 }  ][line width=0.75]    (10.93,-3.29) .. controls (6.95,-1.4) and (3.31,-0.3) .. (0,0) .. controls (3.31,0.3) and (6.95,1.4) .. (10.93,3.29)   ;
				\draw  [dash pattern={on 0.84pt off 2.51pt}]  (93,202.83) .. controls (109.64,185.75) and (135.16,111.09) .. (138.54,75.34) ;
				\draw [shift={(138.68,73.74)}, rotate = 94.47] [color={rgb, 255:red, 0; green, 0; blue, 0 }  ][line width=0.75]    (10.93,-3.29) .. controls (6.95,-1.4) and (3.31,-0.3) .. (0,0) .. controls (3.31,0.3) and (6.95,1.4) .. (10.93,3.29)   ;
				\draw   (256.17,213.92) .. controls (256.17,205.68) and (262.85,199) .. (271.08,199) .. controls (279.32,199) and (286,205.68) .. (286,213.92) .. controls (286,222.15) and (279.32,228.83) .. (271.08,228.83) .. controls (262.85,228.83) and (256.17,222.15) .. (256.17,213.92) -- cycle ;
				\draw   (237.17,113.92) .. controls (237.17,105.68) and (243.85,99) .. (252.08,99) .. controls (260.32,99) and (267,105.68) .. (267,113.92) .. controls (267,122.15) and (260.32,128.83) .. (252.08,128.83) .. controls (243.85,128.83) and (237.17,122.15) .. (237.17,113.92) -- cycle ;
				\draw  [fill={rgb, 255:red, 155; green, 155; blue, 155 }  ,fill opacity=1 ] (316.17,58.92) .. controls (316.17,50.68) and (322.85,44) .. (331.08,44) .. controls (339.32,44) and (346,50.68) .. (346,58.92) .. controls (346,67.15) and (339.32,73.83) .. (331.08,73.83) .. controls (322.85,73.83) and (316.17,67.15) .. (316.17,58.92) -- cycle ;
				\draw  [fill={rgb, 255:red, 255; green, 255; blue, 255 }  ,fill opacity=1 ] (357.17,151.92) .. controls (357.17,143.68) and (363.85,137) .. (372.08,137) .. controls (380.32,137) and (387,143.68) .. (387,151.92) .. controls (387,160.15) and (380.32,166.83) .. (372.08,166.83) .. controls (363.85,166.83) and (357.17,160.15) .. (357.17,151.92) -- cycle ;
				\draw  [dash pattern={on 0.84pt off 2.51pt}]  (252.08,128.83) .. controls (248.08,158.23) and (244.16,175.15) .. (259.06,203.11) ;
				\draw [shift={(260,204.83)}, rotate = 241.11] [color={rgb, 255:red, 0; green, 0; blue, 0 }  ][line width=0.75]    (10.93,-3.29) .. controls (6.95,-1.4) and (3.31,-0.3) .. (0,0) .. controls (3.31,0.3) and (6.95,1.4) .. (10.93,3.29)   ;
				\draw  [dash pattern={on 0.84pt off 2.51pt}]  (264,199.83) .. controls (258.07,225.02) and (272.26,148.81) .. (253.59,130.15) ;
				\draw [shift={(252.08,128.83)}, rotate = 37.41] [color={rgb, 255:red, 0; green, 0; blue, 0 }  ][line width=0.75]    (10.93,-3.29) .. controls (6.95,-1.4) and (3.31,-0.3) .. (0,0) .. controls (3.31,0.3) and (6.95,1.4) .. (10.93,3.29)   ;
				\draw    (331.08,73.83) .. controls (332.93,107.61) and (346.97,126.48) .. (362.32,136.75) ;
				\draw [shift={(364,137.83)}, rotate = 212.01] [color={rgb, 255:red, 0; green, 0; blue, 0 }  ][line width=0.75]    (10.93,-3.29) .. controls (6.95,-1.4) and (3.31,-0.3) .. (0,0) .. controls (3.31,0.3) and (6.95,1.4) .. (10.93,3.29)   ;
				\draw  [dash pattern={on 0.84pt off 2.51pt}]  (262,100.83) .. controls (285.64,112.65) and (296.67,87.6) .. (315.31,60.17) ;
				\draw [shift={(316.17,58.92)}, rotate = 124.47] [color={rgb, 255:red, 0; green, 0; blue, 0 }  ][line width=0.75]    (10.93,-3.29) .. controls (6.95,-1.4) and (3.31,-0.3) .. (0,0) .. controls (3.31,0.3) and (6.95,1.4) .. (10.93,3.29)   ;
				\draw  [dash pattern={on 0.84pt off 2.51pt}]  (286,213.92) .. controls (317.52,211.86) and (339.34,190.49) .. (364.83,167.87) ;
				\draw [shift={(366,166.83)}, rotate = 138.5] [color={rgb, 255:red, 0; green, 0; blue, 0 }  ][line width=0.75]    (10.93,-3.29) .. controls (6.95,-1.4) and (3.31,-0.3) .. (0,0) .. controls (3.31,0.3) and (6.95,1.4) .. (10.93,3.29)   ;
				\draw    (357.17,151.92) .. controls (325.16,132.54) and (284,117.89) .. (266.78,123.18) ;
				\draw [shift={(265,123.83)}, rotate = 336.37] [color={rgb, 255:red, 0; green, 0; blue, 0 }  ][line width=0.75]    (10.93,-3.29) .. controls (6.95,-1.4) and (3.31,-0.3) .. (0,0) .. controls (3.31,0.3) and (6.95,1.4) .. (10.93,3.29)   ;
				\draw  [dash pattern={on 0.84pt off 2.51pt}]  (265,123.83) .. controls (283.72,138.61) and (320.86,152.41) .. (355.58,151.95) ;
				\draw [shift={(357.17,151.92)}, rotate = 178.51] [color={rgb, 255:red, 0; green, 0; blue, 0 }  ][line width=0.75]    (10.93,-3.29) .. controls (6.95,-1.4) and (3.31,-0.3) .. (0,0) .. controls (3.31,0.3) and (6.95,1.4) .. (10.93,3.29)   ;
				\draw    (325.68,70.74) .. controls (302.96,100.71) and (276.35,178.35) .. (279.61,198.16) ;
				\draw [shift={(280,199.83)}, rotate = 252.33] [color={rgb, 255:red, 0; green, 0; blue, 0 }  ][line width=0.75]    (10.93,-3.29) .. controls (6.95,-1.4) and (3.31,-0.3) .. (0,0) .. controls (3.31,0.3) and (6.95,1.4) .. (10.93,3.29)   ;
				\draw  [dash pattern={on 0.84pt off 2.51pt}]  (280,199.83) .. controls (296.64,182.75) and (322.16,108.09) .. (325.54,72.34) ;
				\draw [shift={(325.68,70.74)}, rotate = 94.47] [color={rgb, 255:red, 0; green, 0; blue, 0 }  ][line width=0.75]    (10.93,-3.29) .. controls (6.95,-1.4) and (3.31,-0.3) .. (0,0) .. controls (3.31,0.3) and (6.95,1.4) .. (10.93,3.29)   ;
				\draw   (449.17,214.92) .. controls (449.17,206.68) and (455.85,200) .. (464.08,200) .. controls (472.32,200) and (479,206.68) .. (479,214.92) .. controls (479,223.15) and (472.32,229.83) .. (464.08,229.83) .. controls (455.85,229.83) and (449.17,223.15) .. (449.17,214.92) -- cycle ;
				\draw   (430.17,114.92) .. controls (430.17,106.68) and (436.85,100) .. (445.08,100) .. controls (453.32,100) and (460,106.68) .. (460,114.92) .. controls (460,123.15) and (453.32,129.83) .. (445.08,129.83) .. controls (436.85,129.83) and (430.17,123.15) .. (430.17,114.92) -- cycle ;
				\draw  [fill={rgb, 255:red, 155; green, 155; blue, 155 }  ,fill opacity=1 ] (509.17,59.92) .. controls (509.17,51.68) and (515.85,45) .. (524.08,45) .. controls (532.32,45) and (539,51.68) .. (539,59.92) .. controls (539,68.15) and (532.32,74.83) .. (524.08,74.83) .. controls (515.85,74.83) and (509.17,68.15) .. (509.17,59.92) -- cycle ;
				\draw  [fill={rgb, 255:red, 255; green, 255; blue, 255 }  ,fill opacity=1 ][line width=2.25]  (550.17,152.92) .. controls (550.17,144.68) and (556.85,138) .. (565.08,138) .. controls (573.32,138) and (580,144.68) .. (580,152.92) .. controls (580,161.15) and (573.32,167.83) .. (565.08,167.83) .. controls (556.85,167.83) and (550.17,161.15) .. (550.17,152.92) -- cycle ;
				\draw  [dash pattern={on 0.84pt off 2.51pt}]  (445.08,129.83) .. controls (441.08,159.23) and (437.16,176.15) .. (452.06,204.11) ;
				\draw [shift={(453,205.83)}, rotate = 241.11] [color={rgb, 255:red, 0; green, 0; blue, 0 }  ][line width=0.75]    (10.93,-3.29) .. controls (6.95,-1.4) and (3.31,-0.3) .. (0,0) .. controls (3.31,0.3) and (6.95,1.4) .. (10.93,3.29)   ;
				\draw  [dash pattern={on 0.84pt off 2.51pt}]  (457,200.83) .. controls (451.07,226.02) and (465.26,149.81) .. (446.59,131.15) ;
				\draw [shift={(445.08,129.83)}, rotate = 37.41] [color={rgb, 255:red, 0; green, 0; blue, 0 }  ][line width=0.75]    (10.93,-3.29) .. controls (6.95,-1.4) and (3.31,-0.3) .. (0,0) .. controls (3.31,0.3) and (6.95,1.4) .. (10.93,3.29)   ;
				\draw  [dash pattern={on 0.84pt off 2.51pt}]  (455,101.83) .. controls (478.64,113.65) and (489.67,88.6) .. (508.31,61.17) ;
				\draw [shift={(509.17,59.92)}, rotate = 124.47] [color={rgb, 255:red, 0; green, 0; blue, 0 }  ][line width=0.75]    (10.93,-3.29) .. controls (6.95,-1.4) and (3.31,-0.3) .. (0,0) .. controls (3.31,0.3) and (6.95,1.4) .. (10.93,3.29)   ;
				\draw  [dash pattern={on 0.84pt off 2.51pt}]  (479,214.92) .. controls (510.52,212.86) and (532.34,191.49) .. (557.83,168.87) ;
				\draw [shift={(559,167.83)}, rotate = 138.5] [color={rgb, 255:red, 0; green, 0; blue, 0 }  ][line width=0.75]    (10.93,-3.29) .. controls (6.95,-1.4) and (3.31,-0.3) .. (0,0) .. controls (3.31,0.3) and (6.95,1.4) .. (10.93,3.29)   ;
				\draw  [dash pattern={on 0.84pt off 2.51pt}]  (550.17,152.92) .. controls (518.16,133.54) and (477,118.89) .. (459.78,124.18) ;
				\draw [shift={(458,124.83)}, rotate = 336.37] [color={rgb, 255:red, 0; green, 0; blue, 0 }  ][line width=0.75]    (10.93,-3.29) .. controls (6.95,-1.4) and (3.31,-0.3) .. (0,0) .. controls (3.31,0.3) and (6.95,1.4) .. (10.93,3.29)   ;
				\draw  [dash pattern={on 0.84pt off 2.51pt}]  (458,124.83) .. controls (476.72,139.61) and (513.86,153.41) .. (548.58,152.95) ;
				\draw [shift={(550.17,152.92)}, rotate = 178.51] [color={rgb, 255:red, 0; green, 0; blue, 0 }  ][line width=0.75]    (10.93,-3.29) .. controls (6.95,-1.4) and (3.31,-0.3) .. (0,0) .. controls (3.31,0.3) and (6.95,1.4) .. (10.93,3.29)   ;
				\draw    (518.68,71.74) .. controls (495.96,101.71) and (469.35,179.35) .. (472.61,199.16) ;
				\draw [shift={(473,200.83)}, rotate = 252.33] [color={rgb, 255:red, 0; green, 0; blue, 0 }  ][line width=0.75]    (10.93,-3.29) .. controls (6.95,-1.4) and (3.31,-0.3) .. (0,0) .. controls (3.31,0.3) and (6.95,1.4) .. (10.93,3.29)   ;
				\draw  [dash pattern={on 0.84pt off 2.51pt}]  (473,200.83) .. controls (489.64,183.75) and (515.16,109.09) .. (518.54,73.34) ;
				\draw [shift={(518.68,71.74)}, rotate = 94.47] [color={rgb, 255:red, 0; green, 0; blue, 0 }  ][line width=0.75]    (10.93,-3.29) .. controls (6.95,-1.4) and (3.31,-0.3) .. (0,0) .. controls (3.31,0.3) and (6.95,1.4) .. (10.93,3.29)   ;
				\draw  [dash pattern={on 0.84pt off 2.51pt}]  (557,138.83) .. controls (548.27,98.25) and (538.6,90.45) .. (525.33,76.18) ;
				\draw [shift={(524.08,74.83)}, rotate = 47.46] [color={rgb, 255:red, 0; green, 0; blue, 0 }  ][line width=0.75]    (10.93,-3.29) .. controls (6.95,-1.4) and (3.31,-0.3) .. (0,0) .. controls (3.31,0.3) and (6.95,1.4) .. (10.93,3.29)   ;
				\draw  [dash pattern={on 0.84pt off 2.51pt}]  (553,167) .. controls (514,169.93) and (496.86,197.57) .. (480.28,213.7) ;
				\draw [shift={(479,214.92)}, rotate = 316.89] [color={rgb, 255:red, 0; green, 0; blue, 0 }  ][line width=0.75]    (10.93,-3.29) .. controls (6.95,-1.4) and (3.31,-0.3) .. (0,0) .. controls (3.31,0.3) and (6.95,1.4) .. (10.93,3.29)   ;

			\end{tikzpicture}}
		\caption{Left: Four slow vertices, {of which the top and right ones are infected}. Unrevealed edges are represented as dotted, present as solid, and absent are not drawn. In the picture the edges leaving the infected vertices, {marked as shaded}, have been revealed. Centre: The vertex on the right has recovered but no updatings have occurred so the information about edges is kept. Right: The vertex on the right has updated making all the edges leaving it unrevealed. Observe that the edge going from the top vertex to the right one also updates since the top vertex is infected, {but} it is immediately revealed again, flipping a coin to decide its status which in this case turns to be absent.}
		\label{figu2}
			\end{figure}
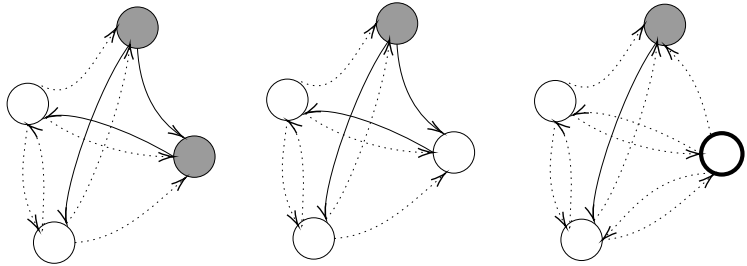
		\end{itemize}
		
		\item Slow vertices recover at rate $1$ (unless saturated) and transmit the infection through present edges at rate $\lambda$ (regardless of the state of the receiving vertex). If an infection is sent through the present edge $(i,j)$, then\smallskip
		\begin{itemize}
			\item both vertices $i$ and $j$ become instantly \emph{saturated};
			\item the edges $(i,j)$ and $(j,i)$ become \emph{saturated}, regardless of their previous state.\smallskip
		\end{itemize} 
		{See Figure~\ref{figu3} for an example.}
		The reason why this infection leads to saturation is that for a slow vertex $i$ we have $T^{loc}(i/N)\ge 8/3\kappa_i\ge 8/3\lambda$, which fits the intuition that both $i$ and $j$ are likely to have enough time to infect each other if the edge $(i,j)$ is present. Contrary to the intuition given for vertices, the saturation of an edge means that it becomes useless for the infection, since we know that both vertices are saturated and hence constantly infected. It will be important later that when referring to present edges we are not including saturated ones for this same reason.
		
		\begin{figure}[h!]
			
			\begin{tikzpicture}[x=0.65pt,y=0.65pt,yscale=-1,xscale=1]
				
				\draw  [fill={rgb, 255:red, 155; green, 155; blue, 155 }  ,fill opacity=1 ] (13,50.77) .. controls (13,39.45) and (22.22,30.28) .. (33.59,30.28) .. controls (44.96,30.28) and (54.18,39.45) .. (54.18,50.77) .. controls (54.18,62.08) and (44.96,71.26) .. (33.59,71.26) .. controls (22.22,71.26) and (13,62.08) .. (13,50.77) -- cycle ;
				\draw   (156.32,50.77) .. controls (156.32,39.45) and (165.54,30.28) .. (176.91,30.28) .. controls (188.28,30.28) and (197.5,39.45) .. (197.5,50.77) .. controls (197.5,62.08) and (188.28,71.26) .. (176.91,71.26) .. controls (165.54,71.26) and (156.32,62.08) .. (156.32,50.77) -- cycle ;
				\draw    (45.53,66.95) .. controls (72.31,94.4) and (139.84,94.81) .. (160.75,67.4) ;
				\draw [shift={(161.67,66.13)}, rotate = 134.31] [color={rgb, 255:red, 0; green, 0; blue, 0 }  ][line width=0.75]    (21.86,-6.58) .. controls (13.9,-2.79) and (6.61,-0.6) .. (0,0) .. controls (6.61,0.6) and (13.9,2.79) .. (21.86,6.58)   ;
				\draw [shift={(103.97,87.73)}, rotate = 45] [color={rgb, 255:red, 0; green, 0; blue, 0 }  ][line width=0.75]    (-11.18,0) -- (11.18,0)(0,11.18) -- (0,-11.18)   ;
				\draw  [dash pattern={on 0.84pt off 2.51pt}]  (164.14,34.17) .. controls (137.37,6.72) and (69.84,6.31) .. (48.93,33.72) ;
				\draw [shift={(48.01,34.99)}, rotate = 314.31] [color={rgb, 255:red, 0; green, 0; blue, 0 }  ][line width=0.75]    (21.86,-6.58) .. controls (13.9,-2.79) and (6.61,-0.6) .. (0,0) .. controls (6.61,0.6) and (13.9,2.79) .. (21.86,6.58)   ;
				\draw  [fill={rgb, 255:red, 155; green, 155; blue, 155 }  ,fill opacity=1 ] (231,50.77) .. controls (231,39.45) and (240.22,30.28) .. (251.59,30.28) .. controls (262.96,30.28) and (272.18,39.45) .. (272.18,50.77) .. controls (272.18,62.08) and (262.96,71.26) .. (251.59,71.26) .. controls (240.22,71.26) and (231,62.08) .. (231,50.77) -- cycle ;
				\draw  [fill={rgb, 255:red, 155; green, 155; blue, 155 }  ,fill opacity=1 ] (374.32,50.77) .. controls (374.32,39.45) and (383.54,30.28) .. (394.91,30.28) .. controls (406.28,30.28) and (415.5,39.45) .. (415.5,50.77) .. controls (415.5,62.08) and (406.28,71.26) .. (394.91,71.26) .. controls (383.54,71.26) and (374.32,62.08) .. (374.32,50.77) -- cycle ;
				\draw    (263.53,66.95) .. controls (290.31,94.4) and (357.84,94.81) .. (378.75,67.4) ;
				\draw [shift={(379.67,66.13)}, rotate = 134.31] [color={rgb, 255:red, 0; green, 0; blue, 0 }  ][line width=0.75]    (21.86,-6.58) .. controls (13.9,-2.79) and (6.61,-0.6) .. (0,0) .. controls (6.61,0.6) and (13.9,2.79) .. (21.86,6.58)   ;
				\draw [shift={(321.97,87.73)}, rotate = 45] [color={rgb, 255:red, 0; green, 0; blue, 0 }  ][line width=0.75]    (-11.18,0) -- (11.18,0)(0,11.18) -- (0,-11.18)   ;
				\draw    (382.14,34.17) .. controls (355.37,6.72) and (287.84,6.31) .. (266.93,33.72) ;
				\draw [shift={(266.01,34.99)}, rotate = 314.31] [color={rgb, 255:red, 0; green, 0; blue, 0 }  ][line width=0.75]    (21.86,-6.58) .. controls (13.9,-2.79) and (6.61,-0.6) .. (0,0) .. controls (6.61,0.6) and (13.9,2.79) .. (21.86,6.58)   ;
				\draw  [fill={rgb, 255:red, 0; green, 0; blue, 0 }  ,fill opacity=1 ] (448,52.77) .. controls (448,41.45) and (457.22,32.28) .. (468.59,32.28) .. controls (479.96,32.28) and (489.18,41.45) .. (489.18,52.77) .. controls (489.18,64.08) and (479.96,73.26) .. (468.59,73.26) .. controls (457.22,73.26) and (448,64.08) .. (448,52.77) -- cycle ;
				\draw  [fill={rgb, 255:red, 0; green, 0; blue, 0 }  ,fill opacity=1 ] (591.32,52.77) .. controls (591.32,41.45) and (600.54,32.28) .. (611.91,32.28) .. controls (623.28,32.28) and (632.5,41.45) .. (632.5,52.77) .. controls (632.5,64.08) and (623.28,73.26) .. (611.91,73.26) .. controls (600.54,73.26) and (591.32,64.08) .. (591.32,52.77) -- cycle ;
				\draw    (489.18,52.77) -- (591.32,52.77) ;
				
				\draw (26.45,71.4) node [anchor=north east] [inner sep=0.75pt]  [font=\large]  {$i$};
				\draw (242.45,69.4) node [anchor=north east] [inner sep=0.75pt]  [font=\large]  {$i$};
				\draw (456.45,69.4) node [anchor=north east] [inner sep=0.75pt]  [font=\large]  {$i$};
				\draw (197.45,69.4) node [anchor=north east] [inner sep=0.75pt]  [font=\large]  {$j$};
				\draw (415.45,69.4) node [anchor=north east] [inner sep=0.75pt]  [font=\large]  {$j$};
				\draw (635.45,69.4) node [anchor=north east] [inner sep=0.75pt]  [font=\large]  {$j$};
				\draw (329,111) node [anchor=north] [inner sep=0.75pt]  [font=\footnotesize] [align=left] {{If $\displaystyle j$ gets infected by a third}\\ {vertex, we reveal the edge $\displaystyle ( j,i)$, }\\{which in this case is present.}};
				\draw (547,111) node [anchor=north] [inner sep=0.75pt]  [font=\footnotesize] [align=left] {{An infection occured along $\displaystyle ( j,i)$.}\\{Both vertices and both edges are}\\{now saturated. Only an update of}\\{$\displaystyle i$ or $\displaystyle j$ will change this situation. }};
				\draw (506,34) node [anchor=north west][inner sep=0.75pt]   [align=left] {saturation};
				\draw (109,111) node [anchor=north] [inner sep=0.75pt]  [font=\footnotesize] [align=left] {{If $\displaystyle i$ becomes infected and $\displaystyle j$ is not, }\\{we reveal the edge $\displaystyle ( i,j)$, which in }\\{this case is absent. We keep the}\\{edge $\displaystyle ( j,i)$ unrevealed.}};
			\end{tikzpicture}	
			\caption{Dynamics of the slow edges $(i,j)$ and $(j,i)$.}
			\label{figu3}
		\end{figure}
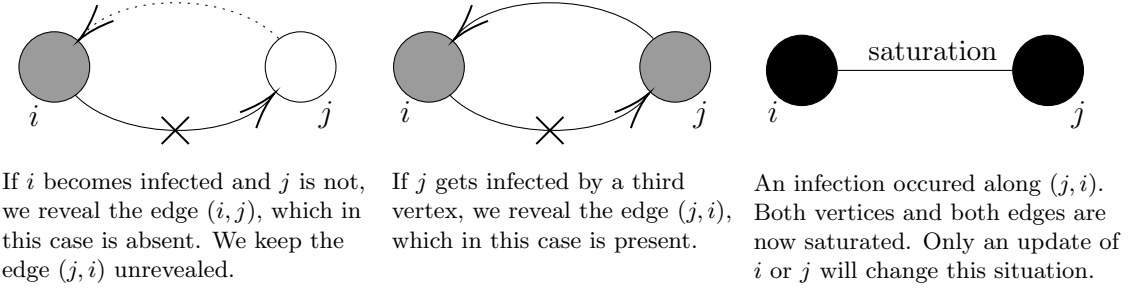
		
		\item For any given slow vertex $i$ define $\Npr_i(t)$ as the number of present edges leaving $i$ at time $t$. If at any given time it satisfies \[
		\Npr_i(t)\ge \frac {3\kappa_i \Tloc_i} {8\lambda},
		\]
		then we declare vertex $i$ to be \emph{saturated}.
		Observe that when a vertex is healthy, the number of present edges leaving it can only decrease so at any time in which the condition is met, $i$ must be necessarily infected. We remind the reader that a saturated vertex can only stop being saturated by updating, so dropping $\Npr_i(t)$ below \smash{$\frac {3\kappa_i \Tloc_i} {8\lambda}$} by the updating of its neighbours is irrelevant. {Intuitively, this condition entails that $i$ is anyway very likely to quickly transmit an infection to another slow vertex, thus leading to saturation by the previous rule.} 
		\item When a slow vertex $i$ updates while infected, we sample freshly the slow edges leaving it and thus observe a totally new value for $\Npr_i(t)$. 
		 Irrespectively of whether $i$ was saturated or unsaturated before the update, the vertex $i$ will be saturated after the update {only} if the new value of 
		 $\Npr_i(t)$ is too large, however this happens with probability bounded, using the Markov inequality, by	
		\begin{equation}\label{ineq:resaturationbyN}
			\frac {8\lambda}{3\kappa_i T^{\rm{loc}}_i} \sum_{j\le \asl N} p_{ij} \le \frac 1 {6\lambda \pi_i} \int_0^{\asl} p(i/N,z) \, \mathrm dz \le \frac 13,
		\end{equation}
		where the first inequality comes from the definition of $T^{\rm{loc}}_i$ and the second from that of $\pi_i$, together with the fact that the updating rate of a slow edge is bounded by~$2\lambda$.

		\item We now describe the rules affecting quick edges, for which we use the wait-and-see coupling introduced in~\cite[Section 5.1]{JLM19} where each such edge is treated as a possible {channel} for the infection, irrespectively of whether the edge belongs to the original network or not. Instead of thinking of quick edges as absent or present   we classify them as {\emph{clogged} or \emph{clear}}. A clogged edge $\{i,j\}$ transmits the infection at the reduced rate $\lambda p_{i,j}$ (as long as either $i$ or $j$ is infected), {and if this happens the edge is cleared.} A \emph{clear} edge then transmits the infection at the regular rate $\lambda$. Note that if the vertices $i$ and $j$ are infected and the edge is clogged, we still count an infection along the edge at rate $\lambda p_{i,j}$, with the only effect of clearing the edge. The reduced rate associated to clogged edges reflects the fact that the edge is likely to be absent from the original network. We refer to~\cite[Section 5.1]{JLM19} to justify that we get an upper bound for the original process. Observe that only quick edges can be clogged or clear whereas only slow edges can be unrevealed, present, absent or saturated.
		
		\begin{table}[h]
			\centering
			\begin{tabular}{|c||c|c|c|c|c|c|}
				\hline
				Type of edge &  \multicolumn{2}{c|}{Quick}& \multicolumn{4}{c|}{Slow}\\
				\hline
				 State of  edge  & \emph{clogged}  & \emph{clear} & \emph{unrevealed}  & \emph{absent}& \emph{present} & \emph{saturated}
				\\ \hline
				
				Infection rate      & $\lambda p_{i,j}$  & $\lambda$ &   \slashbox{}{} & 0&$\lambda$ & \slashbox{}{}                          \\  \hline               
			\end{tabular}
			\caption{Infection rate along an edge, depending on its state. }
			\label{tab:edgesinfectionrates}
		\end{table}
		\item Recall that quick edges can be adjacent to both slow and quick vertices $i$, the latter also updating at rate $\kappa_i$. When a vertex $i$ updates, all quick edges $\{i,j\}$ become clogged irrespectively of their previous status.
		
		\item  For every quick edge $\{i,j\}$ define $R_{ij}= \lambda/(\kappa_i+\kappa_j)$ and for any vertex $i$ (quick or slow) define
		\[R_i(t):=  \sum_{\{i,j\}\text{ clear}} R_{ij},\]
		which is an upper bound for the probability that one of the clear edges incident to $i$ transmits an infection to $i$ before being updated. If at any given time, vertex $i$ satisfies $R_i(t)\geq 1/4$, then we declare it to be \emph{saturated}. Observe that quick edges can only become clear through the spread of an infection, thus at the time in which $R_i(t)\geq1/4$ vertex $i$ is necessarily infected. Commonly, the number of infected neighbours of $i$ is a fraction of $R_i$ and hence a vertex, which is saturated due to this condition, would be very likely reinfected after a recovery.
	
		\item At time $t=0$ the initial configuration of the process is constructed as follows:
		\begin{itemize}[leftmargin=*]
			\item First, declare the set of vertices that are infected. Also, set each quick edge as clogged and each slow edge as unrevealed.
			\item Next, update each slow vertex through the two-step mechanism described above. This reveals all slow edges leaving slow infected vertices, while maintaining all other slow edges unrevealed.
		\end{itemize}
		\end{itemize}
		
		\smallskip
		
		The fact that this new process {is a stochastic upper bound of} the original contact process can be proved similarly to what was done in~\cite[Section 5.1]{JLM19}, together with Lemma~\ref{lem:splitting_edges} below, and {observing that saturated vertices cannot recover, and whenever a slow edge $(i,j)$ is used to spread the infection, both $(i,j)$ and $(j,i)$ become saturated, and thus can be thought as belonging to the network.}
		 
	\begin{lemma} \label{lem:splitting_edges}
		Suppose $p\in [0,1]$,  $G=(V,E)$ is a finite graph, and $e\in\mathcal P_2(V)\backslash E$ an edge not present in the graph. 
		Let $G^{e,p}$ and $\overrightarrow{G}^{e,p}$ be two random graphs such that:
		\begin{itemize}
			\item $G^{e,p}$ is obtained from $G$ by adding edge $e$ with probability $p$.\smallskip
			\item $\overrightarrow{G}^{e,p}$ is obtained from $G$ by adding independently\footnote{The lemma stays true without this independence property, and this assumption is not used in the proof.} the directed edge $\overrightarrow{e}$ as well as its reverse edge $\overleftarrow{e}$ with probability $p$. \smallskip
		\end{itemize}
		Writing $G^e=(V,E\cup\{e\})$, we thus have $\P(G^{e,p}=G^e)=p$. The contact process on $G^{e,p}$ is then stochastically dominated by the contact process on 
		\smash{$\overrightarrow{G}^{e,p}$} where, the first time an infection is transmitted along edge $\overrightarrow{e}$ or $\overleftarrow{e}$, the graph is instantly replaced by $G^e$.
	\end{lemma}
	\begin{proof}
		Consider the contact process on $G^e$, but at the first time the edge $e$ is used to transmit the infection, with probability $1-p$ the infection and the edge $e$ are deleted.~We get the contact process on $G^{e,p}$ (with the presence of $e$ in the graph not being revealed at once).
		If the infection (in $G^e$) was transmitted in the direction of $\overrightarrow{e}$ (resp. $\overleftarrow{e}$), we couple the choice of keeping $e$ in $G^{e,p}$ with that of keeping $\overrightarrow{e}$ (resp. $\overleftarrow{e}$) in \smash{$\overrightarrow{G}^{e,p}$}.	Then, the first infection using $e$ in $G^{e,p}$ (if $G^{e,p}=G^e$) is an infection in \smash{$\overrightarrow{G}^{e,p}$}. The lemma~follows.
	\end{proof}
	
	\smallskip
	
	We are now ready to define the score associated to a given configuration of the new process: Let $s$ be the function given in the hypothesis of the theorem and define \begin{equation}\label{def: t(x)}
		t(x)= \frac {s(x)}{\ratiost \kappa(x)\Tloc(x)},
	\end{equation}
	as well as $t_j=t(j/N)$, with the constant $\ratiost$ being defined\footnote{The proof of Theorem~\ref{teoupper_optimal} will use again the same notation but with a different value for $\ratiost$, which is why we do not use that $\ratiost=3/2$ up until the end of the proof.} by $\ratiost=3/2$. Since $\Tloc_i\leq 8$ and $\Tloc_i\leq\frac{8}{3\kappa_i}$ we obtain the following bounds which will be used extensively throughout this section: 
	\begin{equation}
		\label{ineq:t_ivss_i}
		t_j\le \frac {3 s_j}{8\ratiost}= \frac {s_j} 4,\quad\text{ and }\quad\kappa_i t_i\le \frac {\Tloc_i \kappa_i t_i}8= \frac {s_i}{8\ratiost}.
	\end{equation}	
	Define
	\[
	m_j(t):=\left\{ \begin{array} {ll}
		s_j+t_j  &\mbox{if } j \mbox{ is \color{magenta} infected saturated},\\
		s_j+ 4 R_j(t) t_j &\mbox{if } j \mbox{ is \color{red} infected unsaturated},\\
		2R_j(t) (s_j+t_j) &\mbox{if }j\mbox{ is \color{mygreen} healthy}.
	\end{array}\right.
	\]
	The score associated to the configuration is then
	\begin{equation}\label{eq:M_t}
	M_t:= \sum_{i}m_i(t)+ {\color{orange}\sum_{(i,j)\text{ present edge}} s_j+t_j+u_j},
	\end{equation}
	where $u_j=t_j/3$. We refer to the first term on the right-hand side of \eqref{eq:M_t} as the \emph{vertex score}, whereas we refer to the second term as the \emph{edge score}. We now discuss the idea behind this score and some computations involved:
	\begin{itemize}[leftmargin=*]
		\item Each vertex $j$ is associated to a score $m_j(t)$ that depends on its current state (healthy, infected unsaturated, infected saturated), on the quick edges $\{j,k\}$ and on the slow edges $(j,k)$ leaving $j$. Each present slow edge $(i,j)$ {entering $j$} is also associated to a score, that reflects the fact that $i$ is likely to later infect $j$. No score is associated to other slow or quick edges and in particular no score is associated to saturated edges.
		\item The score $m_j(t)$ associated to an infected vertex $j$ consists of a base score $s_j$ (reflecting the infective potential of the vertex) plus a term $4R_j(t)t_j$ which scales linearly with $R_j$ and thus with the number of clear edges connected to it. We cap this score at saturation.
		\item The score $m_j(t)$ associated to a healthy vertex $j$ 
		also scales with $R_j$, and thus on the estimated probability that $j$ could be later (re)infected by one of these edges.
		\item When revealing the neighbourhood of a slow vertex (whether it is a result of an infection or an update), the \emph{average} increase of the edge score is bounded by
		\begin{equation}\label{eq:boundnewneigh}
			\sum_{j<\asl N} p_{ij} (s_j+t_j+u_j)\le 
			\frac 4 3 \frac 1 {\kappa_i} 
			\sum_{j<\asl N}  (\kappa_i\vee \kappa_j)  p_{ij}s_j 
			\le  \frac 4 3 \frac {s_i}{\CMI \kappa_i \Tloc_i}=\frac {4 \ratiost}{\CMI} u_i \leq u_i, \hspace{-.8cm}
		\end{equation}
		where we have used that $s_j+t_j+u_j\le s_j+\sfrac 4 3 t_j\le \sfrac 4 3 s_j$, the assumptions on the function $s$, the definition of $t$, and our choice $\CMI > 4\ratiost=6$.
		\item 
		Note that when the present directed slow edge $(i,j)$ transmits an infection to $j$, the score $s_j+t_j+u_j$ associated to it, is itself an upper bound for the new score $m_j(t)=s_j+t_j$ assigned to $j$ (that is now infected and saturated) plus the average increase in score due to revealing  the directed slow edges leaving $j$ (since it is bounded by $u_j$ from the previous argument). 
		Simultaneously, the edge $(i,j)$ becomes saturated, with no score associated to it anymore. Hence, this infection does not increase the total score in average, except possibly for the score associated to $i$, which will then also be saturated. It is especially for this feature that we have introduced $u_j$ and associated the score $s_j+t_j+u_j$ to the present edge $(i,j)$.
		In a sense, associating the score $s_j+t_j+u_j$ to a present slow edge $(i,j)$ also amounts to considering that revealing the edge $(i,j)$ is ``as bad'' as using it to transmit the infection.
		\item 
		We now summarize the previously described conditions entailing saturation:\smallskip
		\begin{itemize}
			\item[(i)] $i$ is a slow vertex, and since its last update it has infected another slow vertex or has been infected by another slow vertex. 
			\item[(ii)] $R_i(t)\ge 1/4$,
			\item[(iii)] $i$ is a slow vertex and $
			\Npr_i(t)\ge \frac {3\kappa_i \Tloc_i} {8\lambda}$.\smallskip
		\end{itemize}
		Recall also that a vertex can only stop being saturated by updating, so saturation is maintained even if $R_i(t)$ or $\Npr_i$ drop below their respective values due to the updating of other vertices.

	\end{itemize}
	
	\bigskip

Let $T_{\rm{hit}}$ be the first infection time of a star in $[1,\ldots,\lfloor aN\rfloor]$. Our aim is to show that $M_{t\wedge T_{\rm{hit}}}$ is a supermartingale. To do so we fix $t>0$, and bound the expected infinitesimal change of the score
\begin{equation}\label{defM_t}
\frac 1 {\de t}  \E[M_{t+\de t} -M_t\vert \F_t]
\end{equation}
uniformly over all configurations in $\F_t$ where stars have not yet been infected. {Loosely speaking, if} $\de t$ is sufficiently small, then aside from some term tending to zero with $\de t$, the difference $M_{t+\de t} -M_t$ comes from a single configuration-changing event. The possible events are infections, recoveries and updates, whose effect on the configuration vary depending on the type of vertex/edge and its state. Since the score is a sum of individual scores, the expression \eqref{defM_t} becomes a large sum of terms associated to different vertices/edges affected by these events. In order to facilitate the analysis of this sum, we group these terms into \emph{contributions} which we control below. The reader {can easily check} that these contributions indeed exhaust the changes of $M_t$ under all possible events without double counting.

\medskip
$\mathbf{ (i)\ Negative\ Contributions}.$\medskip

$\bullet$ \emph{Recoveries of infected unsaturated vertices.} The recovery of an infected unsaturated vertex $i$ (whether fast or slow) holds at rate one and induces a change of score 
\[m_i(t+\de t)-m_i(t)=2R_i(t)(s_i+t_i)-(s_i+4R_i(t)t_i)\leq (2R_i(t)-1)s_i\leq-\frac{s_i}{2}\]
since the fact that $i$ is unsaturated implies that $R_i\leq\frac{1}{4}$. The total contribution to the expected infinitesimal change is thus upper bounded by
\begin{equation}\label{neg1} {\color{red}
-\sum_{\heap{i \text{ infected}}{\text{unsaturated}}}
\frac {s_i} 2.}\end{equation}

$\bullet$ \emph{Unrevealing of present slow edges.} In this contribution we group the terms which appear due to the effect of \emph{unrevealing} previously present edges as a result of the updating of slow vertices. For any given present directed slow edge $(i,j)$, 
the update of one of its endvertices occurs at rate $\kappa_i +\kappa_j$, and has probability at least $2/3$ (actually tending to $1$ with $N$) to stay unrevealed after the update. This yields a contribution
\begin{equation}\label{neg2}
{\color{orange}-\sum_{(i,j)\text{ present}} \frac 23 (\kappa_{i}+\kappa_j) (s_j+t_j+u_j).}
\end{equation}

$\bullet$ \emph{Updates of infected saturated vertices.} In this contribution we group the change in score of infected saturated vertices as a result of their own updating, together with the increase in score that occurs as a result of revealing their new neighbourhood. Each saturated vertex $i$ updates at rate $\kappa_i$, and after its update we have $R_i(t+\de t)=0$ so if $i$ is quick, then
\[m_i(t+\de t)-m_i(t)=s_i-(t_i+s_i)=-t_i.\] 
If $i$ is slow, on the other hand, we also need to take into account the effect of the update on the slow edges:
\begin{itemize}[leftmargin=*]
	\item[--] During the first step of the two-step mechanism all slow edges leaving $i$ become unrevealed, thus dropping $\Npr_i(t+\de t)$ to zero.
	\item[--] During the second step it could happen that $i$ reveals a large value of $\Npr_i(t+\de t)$, thus becoming saturated again, so $m_i(t+\de t)-m_i(t)=0$. However, from \eqref{ineq:resaturationbyN} the probability of this event is bounded by $1/3$.
	\item[--] Also during the second step, a new neighbourhood of $i$ is revealed, which induces a change in the edge score that is in average upper bounded by $u_i=t_i/3$ due to \eqref{eq:boundnewneigh}.
\end{itemize}
Adding both effects we obtain that for a slow vertex $i$ the average change in score is upper bounded by $\frac{2}{3}(-t_i)+\frac{t_i}{3}=-\frac{t_i}{3}$, and since this bound also works for quick vertices, we deduce that the total contribution to \eqref{defM_t} is at most
\begin{equation}\label{neg3}
{\color{magenta}- \sum_{\heap{i \text{ infected}}{\text{saturated}}} \frac 1{3} {\kappa_i t_i}}.
\end{equation}

$\bullet$ \emph{Updates of clear edges adjacent to at least one healthy vertex.} In this contribution we count the change in score $m_i(t+\de t)-m_i(t)$ for every healthy vertex $i$ due to the decrease of $R_i$ after either its own update or the update of one of its neighbours. Observe that the previous contribution addressed the change in score for saturated vertices, which are infected, and hence we are not double counting. For each healthy vertex $i$, its vertex score~is 
\[m_i(t)= 2 R_i(t)(s_i+t_i)= 2 \sum_{\heap{j\sim i}{\{i,j\} \text{ clear edge}}}\frac \lambda {\kappa_i+\kappa_j} (s_i+t_i)\]
This score can decrease by $2R_i(t)(s_i+t_i)$ after its own update (which occurs at rate $\kappa_i$), or can decrease by \smash{$2\frac{\lambda}{\kappa_i+\kappa_j}(s_i+t_i)$} after the update of its neighbour $j$ (which occurs at rate $\kappa_j$). This yields a contribution
\begin{equation}\label{neg4}\sum_{i \text{ healthy}}\bigg[-2R_i(t)\kappa_i(s_i+t_i)-2\sum_{\heap{j\sim i}{\{i,j\} \text{ clear}}}\frac{\lambda\kappa_j}{\kappa_i+\kappa_j}(s_i+t_i)\bigg]={\color{mygreen}
- \sum_{i \text{ healthy}} 2 \lambda \Ncl_i(t) (s_i+t_i),}\end{equation}
where $\Ncl_i(t)$ is the number of clear edges incident to $i$.\\

$\bullet$ \emph{Remaining negative contributions}. There are a few negative contributions we are neglecting in our count, such as the the decrease of $R_i$ after an infected unsaturated 
vertex~$i$ updates, or after one of its neighbours does. The reason behind neglecting such terms is that these are smaller in magnitude than the ones already considered, and are not necessary to counteract the positive contributions we count below.

\medskip

We now come to all the positive contributions, coming from the infections and from the updating of the infected vertices (which may reveal a less favourable environment). We will systematically express these positive terms so as to have them easily compared with and compensated by the negative ones.

\medskip
$\mathbf{ (ii)\ Positive\ contributions}.$\medskip

$\bullet$ \emph{Revealing the neighbourhood of infected unsaturated vertices.} In this contribution we group the terms coming from revealing the neighbourhood of infected unsaturated (slow) vertices after their own update. Observe that this leaves out the effect of revealing new neighbours as a result of the updating of said neighbours; this will be addressed in another contribution. Similarly as for saturated vertices, we count both the increase of the edge scores of the newly revealed present edges, and the one coming from the possibility of having $\Npr_i$ too large after the update. This gives a positive contribution bounded by
\begin{equation}\label{pos1}
\sum_{\heap{i \text{ infected}}{\text{unsaturated}}} \frac 2 {3} \kappa_i t_i \le {\color{red}\sum_{\heap{i \text{ infected}}{\text{unsaturated}}} \frac 2 {3}\ \frac {s_i}{8\ratiost}}.
\end{equation}


\medskip
$\bullet$ \emph{Infections along  clear edges}. The infection of a healthy vertex $i$ through clear edges holds at rate $\lambda \Ncl_i(t)$ and implies an increase of its vertex score of at most
\[m_i(t+\de t)-m_i(t)=(s_i+4R_i(t)t_i)-2R_i(t)(s_i+t_i)=s_i+2R_i(t)(t_i-s_i)\leq s_i\]
 Observe that if $i$ is slow, the infection of $i$ has the added effect of revealing its neighbourhood, thus increasing the edge score. Again, due to \eqref{eq:boundnewneigh} this increase is bounded in average by $u_i\leq t_i$, so grouping together this term and the previous one we obtain a contribution bounded by
\begin{equation}\label{pos2}
	{\color{mygreen}	\sum_{i \text{ healthy}}
		\lambda  \Ncl_i(t) (s_i+t_i)}.
	\end{equation}
	
	$\bullet$ \emph{Infections along present edges}. Each present slow edge $(i,j)$ with $i$ infected
	transmits the infection at rate $\lambda$. This infection has the effect of making $i$, $j$ and $(i,j)$ saturated, while also possibly revealing the neighbours of $j$. The saturation of $j$ implies a maximal increase in edge score of $s_j+t_j$ while the new edges revealed imply an average increase in score of $u_j$. Both contributions cancel out with the saturation of $(i,j)$ which imply a decrease in score of $s_j+t_j+u_j$. It follows that we only have to consider the increase in score given by the saturation of $i$, which is null if $i$ was saturated, and is bounded by $t_i$ otherwise (since $i$ was already infected). It follows that this contribution is bounded by
	\begin{equation}\label{pos3}
	\sum_{\heap{i \text{ infected}}{\text{unsaturated}}} \lambda \Npr_i(t)t_i
	\le \sum_{\heap{i \text{ infected}}{\text{unsaturated}}} \frac {3}{8} \Tloc_i \kappa_i t_i
	= {\color{red}\sum_{\heap{i \text{ infected}}{\text{unsaturated}}} 3\ \frac {s_i}{8\ratiost}},
	\end{equation}
	where in the first inequality we have used the bound $\Npr_i\leq\frac{3\Tloc_i\kappa_i}{8\lambda}$ which holds for unsaturated vertices.
	\medskip

	$\bullet$ \emph{Infections along clogged edges and revealing of slow edges $(i,j)$ by the update of $j$.}
	In this contribution we cluster the terms coming from two different effects, since the bounds obtained from each effect can be naturally merged into one. Consider first any infected vertex $i$. For any clogged edge $\{i,j\}$, we have transmission of an infection from $i$ to $j$ with intensity $\lambda p_{i,j}$. This increases:
	\begin{itemize}
		\item[$(I.1)$] The vertex score of $j$, which is bounded by $s_j+t_j$.
		\item[$(I.2)$] The edge score associated to the possible revealing of the neighbourhood of $j$ (if $j$ slow). By \eqref{eq:boundnewneigh} the average increase of this score is bounded by $u_j$.
		\item[$(I.3)$] The vertex score of $i$ by the increase of $R_i(t)$. This increase in score is equal to
		$\frac{4\lambda}{\kappa_i + \kappa_j}t_i$ if $i$ is unsaturated.
	\end{itemize}

	Suppose now that $i$ is an infected slow vertex (so its neighbourhood has been revealed) and fix any absent slow edge $(i,j)$. This edge becomes present by the update of $j$ with intensity $\kappa_j p_{i,j}$ (it can also become present by the update of $i$, but it was already taken into account). This has the effect of:
	\begin{itemize}
		\item[$(II.1)$] Increasing the edge score by $s_j+t_j+u_j$ due to the new present edge. 
		\item[$(II.2)$] Possibly making $i$ saturated if $\Npr_i(t)$ becomes too large by this additional present edge. Since $i$ was already infected, the score increase is at most $t_i$. 
	\end{itemize} 
	
	In order to deal with all these contributions suppose first that $i$ is infected and quick. Grouping $(I.1)$ and $(I.2)$ together gives a contribution bounded by
	\[\sum_{j}\lambda p_{i,j}(s_j+t_j+u_j)\leq\frac{4}{3}\sum_{j}\lambda p_{i,j}s_j\leq \frac{4}{3}\frac{s_i}{\CMI\Tloc_i}\]
	where in the last inequality we have used~\eqref{IMI_quick}. Suppose now that $i$ is infected and slow. In that case the contributions $(I.1)$ and $(I.2)$ only arise for neighbours $j$ that are quick, while for slow neighbours we can count the contribution coming from $(II.1)$:
	\begin{align*}\sum_{j\text{ quick}}\lambda p_{i,j}(s_j+t_j+u_j)+\sum_{j\text{ slow}}\kappa_j p_{i,j}(s_j+t_j+u_j)&=\sum_{j}(\lambda\wedge\kappa_j)p_{i,j}(s_j+t_j+u_j) \\&\leq\frac{4}{3}\sum_{j}(\lambda \wedge\kappa_j)p_{i,j}s_j\leq \frac{4}{3}\frac{s_i}{\CMI\Tloc_i},\end{align*}
	where this time we have used~\eqref{IMI_slow}. We conclude then that from $(I.1)$, $(I.2)$ and $(II.1)$ we obtain the bound
	\[\sum_{i \text{ infected}}\frac 43 \frac {s_i}{\CMI \Tloc_i}\]
	It will be convenient to split this sum, counting separately over saturated and unsaturated vertices. Using the bounds $\Tloc_i\geq 8$ and $\Tloc_i\geq\frac{8}{3\kappa_i}$ we obtain
	\begin{equation}\label{pos4}
	\sum_{\heap{i \text{ infected}}{\text{unsaturated}}} \frac 43 \frac {s_i}{\CMI \Tloc_i}\le 
	{\color{red}\sum_{\heap{i \text{ infected}}{\text{unsaturated}}} \frac {4  }{3\CMI} \frac {s_i} 8}
	\end{equation}
	and
	\begin{equation}\label{pos5}
		\sum_{\heap{i \text{ infected}}{\text{saturated}}} \frac 43 \frac {s_i}{\CMI \Tloc_i}\le 
		{\color{magenta}\sum_{\heap{i \text{ infected}}{\text{saturated}}} \frac {4r  }{3\CMI} \kappa_it_i}
	\end{equation}
	Regarding $(I.3)$ we first assume that $i$ is quick and unsaturated so the total contribution coming from clogged edges incident to $i$ is at most
	\[\sum_{j}\lambda p_{i,j}\frac{4\lambda}{\kappa_i+\kappa_j}t_i=4\lambda^2t_i\pi_i\leq \frac{\Tloc_i}{4}\kappa_it_i=\frac{s_i}{4r}\]
	where in the inequality we have used that $\Tloc_i\geq16\frac{\lambda^2\pi_i}{\kappa_i}$. Suppose now that $i$ is slow. In this case we can use the fact that for any slow vertex $j$ we have $\kappa_j\leq\lambda$ to bound the contribution coming from $(II.2)$ by
	\[\sum_{j\text{ slow}}\kappa_j p_{i,j}t_i\leq \sum_{j\text{ slow}}\lambda p_{i,j}t_i\leq \sum_{j\text{ slow}}\lambda p_{i,j}t_i\frac{2\lambda}{\kappa_i+\kappa_j}\]
	Since vertices connected to $i$ through clogged edges are necessarily quick we can thus add the terms coming from $(I.3)$ and $(II.2)$ (associated to $i$) together to obtain
	\[\sum_{j\text{ quick}}\lambda p_{i,j}\frac{4\lambda}{\kappa_i+\kappa_j}t_i+\sum_{j\text{ slow}}\kappa_j p_{i,j}t_i\leq \sum_{j}\frac{4\lambda^2}{\kappa_i+\kappa_j}t_ip_{i,j}\leq\frac{s_i}{4r}\]
	which is the same bound obtained for quick vertices, and hence we conclude the upper bound for the total contribution of terms coming from $(I.3)$ and $(II.2)$
	\begin{equation}\label{pos6}
			{\color{red}\sum_{\heap{i \text{ infected}}{\text{unsaturated}}}  \frac {s_i} {4\ratiost}}
	\end{equation}

	\smallskip
	We can now group all the different contributions to the expected infinitesimal change. We begin by grouping \eqref{neg1}, \eqref{pos1}, \eqref{pos3}, \eqref{pos4} and \eqref{pos6} together, so we arrive at
	\[
	{\color{red}
		-\sum_{\heap{i \text{ infected}}{\text{unsaturated}}}\left(4\ratiost- \frac {17} 3 - \frac{4 \ratiost}{3\CMI} \right) \frac {s_i}{8\ratiost}
	}=- \sum_{\heap{i \text{ infected}}{\text{unsaturated}}} \frac {1} {21}\frac {s_i}{8\ratiost}\le 
	- \sum_{\heap{i \text{ infected}}{\text{unsaturated}}} \frac {\kappa_i t_i}{21},
	\]
	where in the equality we have injected the values $\ratiost=3/2$ and $\CMI=7$, and in the inequality we have used \eqref{ineq:t_ivss_i}. Grouping together \eqref{neg3} and \eqref{pos5} yields the upper bound
	\[{\color{magenta}- \sum_{\heap{i \text{ infected}}{\text{saturated}}} \left(\frac 1 3-\frac {4 \ratiost}{3\CMI}\right) \kappa_i t_i= - \sum_{\heap{i \text{ infected}}{\text{saturated}}} \frac {\kappa_i t_i}{21}.
	}
	\]
	Next, we group together \eqref{neg4} and \eqref{pos2}, which gives
	\[{\color{mygreen}-\sum_{i \text{ healthy}} \lambda 
		\Ncl_i(t)(s_i+t_i).
	}\]
	Using the basic observation $R_i(t)\le \Ncl_i(t)$, we finally obtain that the average infinitesimal increase  
	$\frac 1 {\de t}  \E[M_{t+\de t} -M_t\vert \F_t]$ is upper bounded by
	\begin{equation}\label{finalboundM}
	-\sum_{i \text{ infected}} \frac {\kappa_i t_i}{21}- \sum_{i \text{ healthy}} \frac \lambda 2 m_i(t)-\sum_{(i,j) \text{ {present}}} \frac 2 3 (\kappa_{i}+\kappa_j)(s_j+t_j+u_j),
	\end{equation}
	and since this quantity is negative we conclude that $M_{t\wedge T_{\rm{hit}}}$ is a supermartingale.
	\medskip
	
	To finish the proof of the theorem, we first consider the case $a>0$, and go back to the original contact process $X$, for which we can use self-duality to get
	\begin{equation}
		\label{ineq:I_N_byduality}
		I_N(t)\le \frac{ \lceil aN\rceil}N + \frac 1 N \sum_{i=\lceil aN\rceil + 1}^N \P_i(t<T_{\rm{ext}}^X),
	\end{equation}
	where $\P_i$ stands for the law of the infection process with initially only vertex $i$ infected, and $T_{\rm{ext}}^X$ stands for its extinction time. 
	Since $X$ is stochastically bounded from above by the new process up until the hitting time $T_{\rm{hit}}$, we know that on the event $T_{\rm{ext}}<T_{\rm{hit}}$, where $T_{\rm{ext}}$ is the extinction time of the new process, we have $T_{\rm{ext}}^X\leq T_{\rm{ext}}$. Letting $T= T_{\rm{hit}}\wedge T_{\rm{ext}}$, we thus have
	\begin{equation}\label{ineqhitting}
	\P_i(t<T_{\rm{ext}}^X)\le \P_i(T_{\rm{hit}}<T_{\rm{ext}})+ \P_i(t<T).
	\end{equation}
	To control the term $\P_i(T_{\rm{hit}}<T_{\rm{ext}})$ we use that $(M_{t\wedge T})$ is a supermartingale to obtain
	\[\E_i[M_0]\geq\E_i[M_T]=\E_i[M_{T_{\rm{hit}}}\one_{T_{\rm{hit}}<T_{\rm{ext}}}]+\E_i[M_{T_{\rm{ext}}}\one_{T_{\rm{hit}}>T_{\rm{ext}}}]=\E_i[M_{T_{\rm{hit}}}\one_{T_{\rm{hit}}<T_{\rm{ext}}}],\]
	since $M_{T_{\rm{ext}}}\ge 0$. Now, at $T_{\rm{hit}}$ we know that there is at least one star infected, so $M_{T_{\rm{hit}}}\geq s(a)$. On the other hand, under the assumption that at time $t=0$ only $i$ is infected, the total score of the configuration comes from its vertex score, and the initially revealed slow edges $(i,j)$ (if $i$ slow). From \eqref{ineq:resaturationbyN} and \eqref{eq:boundnewneigh} we can thus bound
	\[
	\P_i(T_{\rm{hit}}<T_{\rm{ext}})\leq \frac{\E_i[M_0]}{s(a)}\le \frac{s_i+ \frac {t_i} 3 + u_i}{s(a)}\leq \frac 7 6  \frac {s_i}{s(a)}.
	\]
	To control the term $\P_i(t<T)$ we observe that on $t<T$ there is at least one vertex $j>aN$ infected and hence from \eqref{finalboundM} we deduce that 
	$
	\frac 1 {\de t}  \E[M_{t+\de t} -M_t\vert \F_t] \le - r
	$
	for some $r >0$ independent of $i$, $t$ and $N$ (but which might depend on $a$ and $\lambda$). Hence $M_{t\wedge T}+ r (t\wedge T)$ is a supermartingale, and in particular $r\E_i[T]\leq\E_i[M_{T}+ r T]\leq \E_i[M_0]$. We infer that
	\[
	\P_i(t<T)\le \frac 1 t \E_i[T]\le \frac 1 {r t} \E_i[M_0]\le \frac {7s_i} {6 r t}.
	\]
	Finally, replacing the bounds found in \eqref{ineqhitting} and using said bounds in \eqref{ineq:I_N_byduality} together with a simple sum-integral comparison yields
	\[
	I_N(t)\le a+\frac 1 N + \frac 7 {6 s(a)} \int s(y)dy+ \frac 7{6 r t} \int s(y) dy,
	\]
	and hence we conclude \eqref{ineqmetastablethm3}.\smallskip

	In the case $a=0$ with $(H)_\delta$ satisfied for some $\delta\in[0,1]$, there is $C=C(a,\lambda)$ independent of $N$ such that $\Tloc s^{-\delta}\leq C$. We can use this to further improve \eqref{finalboundM} by observing that
	\[\kappa_i t_i\geq\frac{1}{C}\kappa_i t_i\Tloc_i s_i^{-\delta}=\frac{s_i^{1-\delta}}{Cr}\geq\frac{(s_i+t_i+u_i)^{1-\delta}}{C'}, \]
	for some $C'$. It follows from \eqref{finalboundM} that
	\begin{align*}
		& \frac 1 {\de t}  \E[M_{t+\de t} -M_t\vert \F_t]\leq
			-\sum_{i \text{ infected}} \frac {\kappa_i t_i}{21}- \sum_{i \text{ healthy}} \frac \lambda 2 m_i(t)-\sum_{(i,j) \text{ present}} \frac 2 3 \kappa_jt_j\\&\leq	-\sum_{i \text{ infected}} \frac {(s_i+t_i)^{1-\delta}}{21C'}- \sum_{i \text{ healthy}} \frac \lambda 2 m_i(t)^{1-\delta}-\sum_{(i,j) \text{ present}} \frac {2}{3C'} (s_j+t_j+u_j)^{1-\delta}\\ &\leq -\frac{1}{C''}M_t^{1-\delta},
	\end{align*}
	where in the last inequality we used the concavity of the function $f(x)=x^{1-\delta}$. 
	\pagebreak[3]
	
	\noindent In particular, if $\delta>0$ we conclude that \smash{$M_{t\wedge T_{\rm{ext}}}^\delta + \frac{\delta}{C''} (t\wedge T_{\rm{ext}})$} is a supermartingale and so 
	\[
	\E_i[T_{\rm{ext}}]=\E_i\left[T_{\rm{ext}}+\frac{C''}{\delta}M^{\delta}{T_{\rm{ext}}}\right]\le \frac {C''} {r\delta} \E[M_0^\delta]\le \frac {C''}{r\delta} \E[M_0]^\delta \le \frac {C''\left(\frac 7 6 \int s(y)dy\right)^\delta} { \delta} N^\delta.
	\]
	We get that \eqref{resextime} holds. When $\delta=0$ we use that $\log(1+M_{t\wedge T_{\rm{ext}}})+r (t\wedge  T_{\rm{ext}})/2$ is a supermartingale to deduce the upper bound $\E[\Tex]=O(\log N)$, so \eqref{logboundext} holds. For more details, we refer here to the similar computations made in the appendix of \cite{JLM22}.

\smallskip	

\subsection{Proof of Theorem~\ref{teoupper_optimal}}

The general strategy of the proof is the same as in the proof of Theorem~\ref{teoupper_vertex_improved}: We introduce a stochastic upper bound for the infection process and then associate a score with each configuration. We then prove that the scoring function, recording the evolution in time of the score, defines a supermartingale. We maintain the name $M(t)$ for this scoring function, to emphasise the fact that several computations from the previous proof will be reused here. 
However, our upper bound is much closer to the actual process and the analysis of the supermartingale therefore more complex. In particular, the scores depend on a finer classification of vertices allowing us to separate local and global features of the infection process in the analysis.%
\smallskip 

As before, the threshold $a$ provided in the statement of the theorem defines stars $i<aN$ and non-star vertices $i\ge aN$. We also recall the thresholds
\[\asl=(\lambda /\kappa_0)^{-\frac 1 {\gamma \eta}}\quad \text{and} \quad\assl=\astr \wedge \asl,\]
where $\astr$ was defined by~\eqref{def_astr} and which, by \eqref{condp}, is of order $\lambda^{2/\gamma}$. 
 Using these thresholds we introduce the new classification of the vertices, dividing them into \emph{strong slow}, \emph{strong quick}, and \emph{weak}, as shown in Table~\ref{tab:verticestypes} below. Without loss of generality we 
suppose $a<\assl$, as this is the 
case featuring the largest variety in the classification for nonstar vertices. This also implies that star vertices are strong slow vertices.



\begin{table}[h!]
	\centering
	\begin{tabular}{|l||c|c|c|}
		\hline
		Vertex $i$      & $i \le \assl N$                    &      $\assl N< i\le \astr N$                           &   $\astr N <i \le N$                              \\ \hline\hline
		Type of vertex     & Strong slow                      & Strong quick                    & Weak                            \\ \hline
		Average degree             & \multicolumn{2}{c|}{$\ge 1/\left(10\lambda^2\right)$ }& $<1/(10\lambda^2)$
		\\ \hline
		Updating rate  &$<\lambda$  & $\ge \lambda$ &  
		\\ \hline
		$T^{\rm{loc}}_i$        & \multicolumn{2}{c|}{$\max\left(8,\frac 8{3\kappa_i},16\lambda^2 \frac {\pi_i}{\kappa_i}\right)$} &     not applicable    \\ \hline                       
	\end{tabular}
	\caption{The three types of vertices and their main features.}
	\label{tab:verticestypes}
\end{table}

To 
compare the current classification of vertices to the one used in the proof of Theorem~\ref{teoupper_vertex_improved} we consider separately the cases $\eta<-1/2$ and $\eta\geq -1/2$. In the first case we have $\assl\ll \astr$, and hence the \emph{strong slow vertices} coincide with slow vertices, while the quick vertices are subdivided into \emph{strong quick} and \emph{weak}. On the other hand, if $\eta\geq-1/2$, we have $\assl= \astr$, so \emph{strong quick} vertices do not exist. In that case, \emph{strong slow vertices} are still slow, but now weak vertices can be slow, too. Although weak vertices may exhibit qualitatively distinct behaviours, our methods allow us to work with these vertices without the need to differentiate between the quick and the slow ones. 
\pagebreak[3]

\begin{table}[h!]
	\begin{tabular}{|ll|llllllllllll|}
		\hline
		\multicolumn{2}{|c|}{Theorem~\ref{teoupper_vertex_improved}} & \multicolumn{6}{c|}{$\qquad\qquad$Slow$\qquad\qquad$}         & \multicolumn{6}{c|}{Quick}                         \\ \hline
		\multicolumn{2}{|c|}{Theorem~\ref{teoupper_optimal}, case $\eta<-1/2$} & \multicolumn{6}{l|}{Strong slow}         & \multicolumn{3}{l|}{Strong quick} & \multicolumn{3}{c|}{Weak} \\ \hline
		\multicolumn{2}{|c|}{Theorem~\ref{teoupper_optimal}, case $\eta\geq-1/2$} & \multicolumn{4}{l|}{Strong slow} & \multicolumn{8}{c|}{Weak}                                 \\ \hline
	\end{tabular}
	\captionof{table}{Comparison of the classification of vertices in both theorems\label{table:comparison}}
\end{table}

Compared to Theorem~\ref{teoupper_vertex_improved}, the new classification introduced here allows for a finer control of the behaviour of the infection,  combining a local and global element of the analysis. This can be observed when comparing the inequalities within the statements of both theorems: while inequalities~\eqref{IMI_quick} and~\eqref{IMI_slow} have analogous counterparts~\eqref{OMIstrongquick} and~\eqref{OMIstrongslow}, in Theorem~\ref{teoupper_optimal} we also find inequality~\eqref{OMIweak} for weak vertices. Most noticeably, there is no 
term~$\Tloc$ in~\eqref{OMIweak}, which is in line with the idea that local survival does not apply to weak vertices, and hence imposing this inequality follows closer the expected behaviour of the infection on weak vertices. With the new classification comes a small trade-off: a slight decrease in precision concerning the constants within the statement. Unlike in Theorem~\ref{teoupper_vertex_improved}, where the constant~$\CMI$ was explicit, here we need it to be sufficiently large for the proof to work.\smallskip

 The stochastic upper bound introduced here has several differences with the one introduced in the proof of Theorem~\ref{teoupper_vertex_improved} (although many similarities as well), and to motivate its definition we point out the following observations: 
 \begin{itemize}[leftmargin=*]
	\item 
	 Providing an upper bound for $I_N(t)$ translates  by~\eqref{ineq:I_N_byduality}  into bounding from above the extinction probability
	 of the contact process starting from a single infected non-star vertex. The infection should not be able to spread and survive for too long, unless it infects a star on its way. 	Hence we should not observe vertices with unusually large degree, short cycles, or other structures that are rare in the local limit of the (static) network. If we do, we stop the  process. Part of the proof consists in showing that by that time the infection has already either died out or infected a star.\smallskip
	\item 
	The stochastic upper bound for $(\mathscr G^{\ssup N}_t,X_t)$ is an infection process (akin to the contact process) running on top of an environment modelled as a growing oriented network $(G_t)_{t\geq0}$ representing the environment seen by the infection. The growth of the network 
	is following the oriented infection paths
	using updating rules similar to those in $\mathscr G^{\ssup N}_t$, with one significant deviation from the previous model: connections are never destroyed but merely blocked for the infection. Hence the growing network serves not only to carry  the infection process, but also to keep track of all the vertices and connections that have been revealed in the past. This information 
	can be used for the \emph{global analysis} in a similar way as in the previous section.	\smallskip
	\item 
	The growing network is set up to have a tree structure,
	which we exploit in order to control the behaviour of the infection on weak vertices. For these vertices, the scoring function $M(t)$ is defined for the \emph{local analysis} similarly as in \cite{MVY13}, where the authors study the contact process on static trees with degrees bounded by $1/(8\lambda^2)$. 	Note that weak vertices have expected degree less than $1/(10\lambda^2)$, and thus are quite likely to have degree bounded by $1/(8\lambda^2)$. At the early stage of the infection, we will typically never visit a weak vertex with a degree higher than this threshold. \smallskip
	\item For our methods to work we need to control the degree of weak vertices, which, as seen in Table~\ref{table:comparison} can be slow. In particular, it is not viable to use the clogged/clear scheme introduced in the proof of Theorem~\ref{teoupper_vertex_improved}, {where vertices steadily increase their degree until the next updating event}. Instead, we are forced to reveal more of the network structure, which implies several {complications}. But this network structure will also imply  {simplifications} in our computations as compared to the previous section.
	\end{itemize}

\subsection*{The stochastic upper bound} We now define more precisely the upper bound consisting of an infection process running on top of a growing oriented network $(G_t)_{t\geq0}$. {Despite the complexity of the process it will always be clear from the construction that  up to the time when the process is stopped it represents a stochastic upper bound on the infection.}%
\smallskip

{Let $\mathcal{E}$ the set of all {oriented} edges of the vertex set $\{1,2,\ldots,N\}$ excluding loops.  The growing network is initialized at time $t=0$ as some $G_0=(V_0,E_0)$ to be described later, with vertex set $V_0\subseteq \{1,\ldots,N\}$ and $E_0\subseteq \mathcal{E}$. The network then evolves by connecting to new vertices from this set through edges in $\mathcal{E}$. Together with the evolution of the network $(G_t)_{t\geq0}$ we also keep track of an accompanying process $(\mathfrak{A}_t)_{t\geq0}$ of \emph{available edges}, where $\mathfrak{A}_t\subseteq \mathcal{E}\setminus E_t$ at each time $t$ consists of all the {oriented edges that can be used for the growth of the network. Each edge can only be used once and in only one of the two possible orientations} in the construction of $G_t$, so the set of available edges decreases over time and the set inequality $\mathfrak{A}_t\subseteq \mathcal{E}\setminus E_t$ is, in general, strict because as the network grows, several unused edges will be discarded.}\smallskip

We maintain the 
classification of vertices in $\{1,\ldots,N\}$ distinguishing between 
strong slow, strong quick, and weak. Vertices can be either \emph{healthy} or \emph{infected}, and strong vertices (both quick and slow) can also be \emph{saturated}, which is a state analogous to the one defined in the proof of Theorem~\ref{teoupper_vertex_improved}. We also classify each oriented edge $(i,j)$ 
as either
\begin{itemize}
	\item \emph{weak}, if $j$ is weak,
	\item \emph{strong slow} if both $i$ and $j$ are strong slow,
	\item or \emph{strong quick} if $j$ is strong (either quick or slow) but $(i,j)$ is not strong slow,
\end{itemize}
{see Table~\ref{tab:edgestypes2}.} Observe that in general it is the head of the oriented edge $(i,j)$ which determines its classification, except  for strong slow edges. Also note that in the case $\eta\geq-1/2$, strong quick edges only occur if the head is strong slow and the tail is weak. In that case the term ``strong" is a bit awkward, but, as will become clear below, it points towards the idea that the infection uses such an edge to discover a new strong vertex.\smallskip\pagebreak[3]

Besides their classification, edges also have a \emph{state} which depends on their type and whether they belong to the network at a given time:

\begin{itemize}[leftmargin=*]
		\item Edges that belong to $G_t$ are either \emph{present} or \emph{absent}, regardless of their classification. The reader might think of this in the same way as in the original dynamic network, where present edges are used for the spread of the infection, while absent edges are not.
		\item Strong slow edges that belong to $G_t$ (and only these edges) can also become saturated, which is a state analogous to the one defined for slow edges in the proof of Theorem~\ref{teoupper_vertex_improved}. Again, the intuition is that saturated edges connect saturated vertices, and hence become useless (we provide the exact conditions for saturation later).
		\item Available edges that are strong quick are 
		\emph{clogged}, which, as in Theorem~\ref{teoupper_vertex_improved}, can be thought of as constantly available channels for the infection, traversed at a much smaller~rate.
		\item Available edges that are either weak or strong slow are 
		\emph{unseen}, which is a state that reflects the fact that ``these connections have not occurred so far". 
\end{itemize}

\begin{table}[h]
	\centering
	\begin{tabular}{|c||c|c|c|}
		\hline
		\diagbox{Type of $i$} {Type of $j$}   & Strong slow                      & Strong quick                    & \multicolumn{1}{c|}{ Weak}                             \\ \hline\hline
		Strong slow     & \multicolumn{1}{d|}{unseen/absent/}                &    \multicolumn{1}{b|}{}             & \multicolumn{1}{a|}{ }                         \\ 
		     & \multicolumn{1}{d|}{present/saturated}                 &      \multicolumn{1}{b|}{}              &       \multicolumn{1}{a|}{unseen/absent/present} \\ \cline{1-2}
		Strong quick             & \multicolumn{2}{b|}{\multirow{2}{*}{} }& \multicolumn{1}{a|}{}
		\\ \cline{1-1} 
		Weak  & \multicolumn{2}{b|}{\qquad clogged/absent/present} &  \multicolumn{1}{a|}{}
		\\ \hline               
	\end{tabular}
	\caption{The possible states  of an edge $(i,j)$. Edges are classified \emph{strong slow} in the green, \emph{strong quick} in the blue and \emph{weak} in the yellow field.} 
\label{tab:edgestypes2}
\end{table}
The network $G_t$ represents the graph structure that has been observed by the infection by time $t$. In this sense, the reader may find it useful to think of $G_t$ as a window in which we (somewhat closely) track the structure belonging to $\mathscr{G}^{\ssup N}_t$, and hence it makes intuitive sense for this window to grow over time, representing the growth of the region ``explored" by the infection. This intuition may contrast with the intuition about the \emph{actual} structure in $\mathscr{G}^{\ssup N}_t$, where connections are destroyed over time. We reconcile both intuitions by pointing out that while $G_t$ grows, most of its edges are labeled \emph{absent} over time. From this point onwards, all comments about the growth of the network refer to $G_t$.\smallskip

We now define the following operations:
\begin{enumerate}[leftmargin=*]
	\item Let $(i,j)$ in $\mathfrak{A}_t$ be an available edge at time $t$. We say that we \emph{incorporate} the edge into the network, by:
	\begin{itemize}[leftmargin=*]
		\item Adding the vertex $j$ to $G_t$. If $j$ was already in $G_t$, we ignore this step.
		\item Adding the edge $(i,j)$ to $G_t$, changing its state to \emph{present}.
		\item Removing both $(i,j)$ and $(j,i)$ from $\mathfrak{A}_t$.
	\end{itemize}
	At times we abuse notation and say that we incorporate a vertex $j$. This amounts to saying that we incorporate the edge $(i,j)$ if $i$ is clear from context.\smallskip
	\item Let $i$ be a vertex in $G_t$.  We say that we \emph{reveal new neighbours of $i$} by:
	\begin{itemize}[leftmargin=*]
		\item Tossing an independent coin with probability $p_{i,j}$ for each \emph{unseen} edge $(i,j)$.
		\item Incorporating to the network each edge whose toss turned up heads
	\end{itemize}
	We frequently specify some subset of the neighbours of $i$ to reveal. This means that we only toss coins for the edges $(i,j)$ with $j$ belonging to the specified subset. \smallskip
	\item Let $i\in\{1,2,\ldots,N\}$. We say that we \emph{reveal $i$} by:
	\begin{itemize}[leftmargin=*]
		\item Tossing an independent coin with probability $p_{j,i}$ for each \emph{unseen} edge $(j,i)$, $j\in G_t$.
		\item Incorporating to the network each edge whose toss turned up heads.
	\end{itemize}
		As in the previous operation, we might specify that $i$ is revealed \emph{to a particular subset of vertices}, which implies that we only toss coins for the edges $(j,i)$ with $j$ belonging to the specified subset.
		\end{enumerate}	
	\begin{figure}[h]
	\begin{center}
		\includegraphics[scale=1]{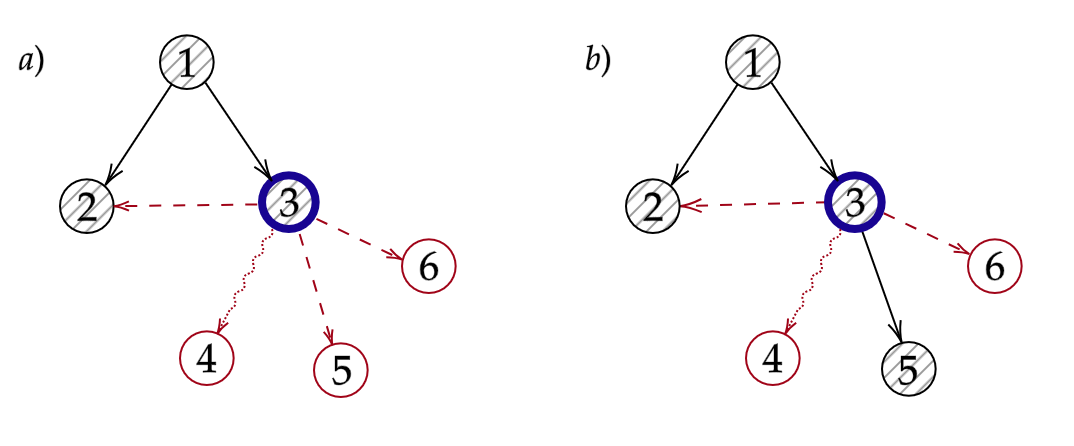}
		\caption{\small $a)$ An example where the vertex set of $G_t$ is $\{1,2,3\}$, while $4$, $5$, and~$6$ have not been added to the network yet. Vertex~$3$ reveals new neighbours so we require to sample all \emph{unseen} available edges leaving $3$. In this case, this includes the edge $(3,2)$, even though $2$ already belongs to $G_t$. We also exclude the edge $(3,4)$ which is \emph{clogged}, represented through a squiggly arrow.  $b)$ After sampling all edges, only edge $(3,5)$ is incorporated to the network, so in particular the new vertex set of $G_t$ is $\{1,2,3,5\}$.}
		\label{fig1thm4}
	\end{center}
	\end{figure}
	\begin{figure}[h]
		\begin{center}
			\includegraphics[scale=1]{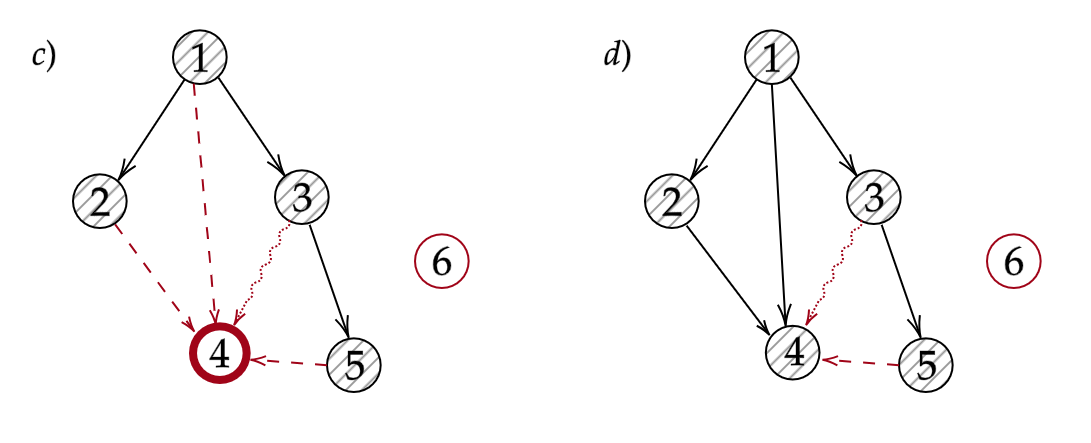}
			\caption{\small $c)$ In the picture we see the network from Figure~\ref{fig1thm4} where now it is vertex $4$ that is being revealed. This requires to sample all \emph{unseen} available edges entering $4$ from vertices in $G_t$. Again, this excludes edge $(3,4)$ which is \emph{clogged}. $d)$ After sampling all edges, both $(1,4)$ and $(2,4)$ are incorporated to the network. The creation of a cycle then stops the process.}\label{fig2thm4}
		\end{center}
	\end{figure}
The following observations are key about the newly defined operations:
\begin{enumerate}[leftmargin=*, label=(\alph*)]
	\item Incorporating an edge means that the network grows by adding the edge (with the corresponding endpoint) to its structure, and since edges can only be incorporated once it must become unavailable afterwards along with its inverse.
	\item When revealing new neighbours of a vertex $i$, we only reveal \emph{unseen} edges, thus leaving out strong quick edges. The reason behind this is that in the new process, strong slow and weak edges behave similarly to the edges in the original dynamic network $\mathscr G^{\ssup N}_t$, while strong quick edges behave similarly to the quick edges in the previous proof, see Figure~\ref{fig1thm4} for an example.
	\item Typically, when we reveal new neighbours of $i \in G_t$, these will not be in  $G_t$ and when we reveal a vertex $i$ to a vertex $j\in G_t$, we have $i\notin G_t$. It is still possible that {neighbours $i,j\in G_t$ are revealed (and the tree structure is broken), see for example Figure~\ref{fig2thm4},} but this is an unlikely and undesirable scenario upon which we stop the process, as we will explain later.
	\item The classification of edges is such that strong quick vertices can never be revealed, and strong slow ones can only be revealed to strong slow vertices in~$G_t$.
\end{enumerate}

We say a vertex $i\in\{1,2,\ldots,N\}$ has been \emph{visited} by time $t$, if it has become infected or saturated at least once by that time. Since $G_t$ represents the {part of the environment observed by the infection by time $t$,} visited vertices always belong to the network. The reader might think of non-visited vertices in $G_t$ as places yet to be explored by the infection.%
\smallskip  

Using the new definitions, we can now begin to describe the evolution of the process and the scoring associated with the configurations it can reach. 

\subsubsection*{Evolution of the process within $G_t$}
 Setting aside the growth of the network, the rules followed by the process are analogous to the original contact process on $\mathscr{G}^{\ssup N}_t$:
\begin{enumerate}[leftmargin=*, label=Rule (\arabic*):, series=processrules]
	\item 
	Each vertex $i\in\{1,2,\ldots,N\}$ updates at rate $\kappa_i$. If $i$ belongs to $G_{t^-}$, then (among other effects we describe later) all edges $(i,j)$ and $(j,i)$ in $G_{t^-}$, whether previously present, absent, or saturated, are resampled as in $\mathscr{G}^{\ssup N}_t$: for each such edge, we flip a coin with probability $p_{i,j}$ and label \emph{present} all the ones for which the coin turned up heads. We label \emph{absent} all the rest.
	\item Within $(G_t)_{t\geq0}$ the infection evolves as a regular contact process with infection rate $\lambda$ (except for saturated vertices which we cover later). The edges used to spread the infection must always be \emph{present} (not absent or saturated) and can be used regardless of their orientation.
\end{enumerate}
Note that within $G_t$ all edges and vertices are treated the same, regardless of their type, and oriented edges are treated as if they were unoriented. 
The classification of vertices has more impact on the network growth than on the changes and infections occurring within~it.%

\subsubsection*{Evolution of the process by inherited rules} 

Several of the evolution rules and scoring functions are adaptations of what was defined in the previous section. Before giving these rules we provide a brief (and incomplete) summary of the process described in Section~\ref{sec_upper_new}:
\begin{itemize}[leftmargin=*]
	\item Vertices were divided into slow and quick, where each vertex $i$ updated at rate~$\kappa_i$.
	\item When a slow vertex $i$ updated, it lost its set of slow neighbours (which became \emph{unrevealed}). However, if at that time $i$ was infected, we sampled a new set of slow neighbours. Also, while $i$ was infected (or saturated), every time another slow vertex $j$ updated, its connection to $i$ was resampled.
	\item When a slow vertex $i$ transitioned from \emph{healthy} to \emph{infected} or \emph{saturated}, we added new slow neighbours among the ones that were \emph{unrevealed} at that time. 
	\item Edges $\{i,j\}$ with at least one quick vertex were classified as either \emph{clogged} or \emph{clear}. Clogged edges transmitted the infection at rate $\lambda p_{i,j}$, and at the infection event, the edge was cleared. Clear edges transmitted the infection at rate $\lambda$, and upon updating, become clogged again.
	\end{itemize}
	We refer the reader to Section~\ref{sec:Thm3} for the intuition behind this construction. We adapt these rules to the current scenario, with strong slow and strong quick vertices playing the role of slow and quick vertices, respectively:
\begin{enumerate}[leftmargin=*, label=Rule (\arabic*):, resume=processrules]
	\item A strong slow vertex $i\in G_t$ reveals new strong slow neighbours if either,
	\begin{itemize}[leftmargin=*]
		\item it updates while infected or saturated ($i$ must be \emph{visited} at time~$t^-$), or
		\item it transitions from \emph{healthy} to \emph{infected} or \emph{saturated}.
	\end{itemize}
	\item When a strong slow vertex $j\in\{1,2,\ldots,N\}$ updates at a time $t$, it is revealed to the strong slow vertices in $G_t$ that are infected.
	\item An infected vertex $i$ transmits the infection through the clogged edges of the form $(i,j)$ at rate $\lambda p_{i,j}$. When an infection of this sort takes place, $j$ becomes infected, $(i,j)$ is incorporated to the network and $j$ becomes instantly infected (and thus visited), {see for example Figure~\ref{fig:clogged}.}
\end{enumerate}

\begin{figure}[h]
	\begin{center}
		\includegraphics[scale=1.2]{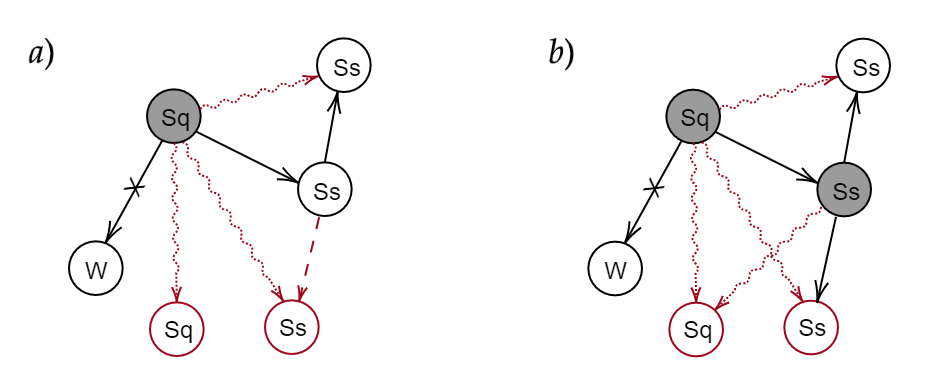}
		\caption{\small $a)$ In the picture $G_t$ consists of the four vertices joined through solid arrows, with a single vertex infected, depicted as grey. Each vertex is classified as either weak (W), strong slow (Ss), or strong quick (Sq), and solid arrows marked by  an $X$ are \emph{absent}. The infected vertex can infect its strong slow neighbour (but not the weak one) or transmit the infection through one of the \emph{clogged} edges, depicted as squiggly arrows, at rate $\lambda p_{i,j}$. $b)$ The infection is transmitted through the present edge. The newly infected strong slow vertex reveals new strong slow neighbours, which results in the incorporation of a new edge to the network.}
		\label{fig:clogged}
	\end{center}
\end{figure}

{A small difference between the rules given here, and the description of the process in Section~\ref{sec_upper_new} is that strong quick edges belonging to $G_t$ do not become clogged again upon updating, and instead evolve as any other edge. In that sense, the current process is closer to the original dynamic network $\mathscr{G}^{\ssup N}_t$, because it has ``better memory".}

\subsubsection*{Evolution of the process at weak vertices and edges}
As pointed out after the presentation of Theorem~\ref{teoupper_vertex_improved} at the beginning of Section~\ref{sec: supermart}, the main improvement of Theorem~\ref{teoupper_optimal} over Theorem~\ref{teoupper_vertex_improved} is related to the treatment of weak vertices. In this treatment, we need to control the degrees of these vertices at all times, which forces us to endow weak vertices and edges with a set of rules that closely resembles those in the original network $\mathscr{G}^{\ssup N}_t$: 
\begin{enumerate}[leftmargin=*, label=Rule (\arabic*):, resume=processrules]
	\item When a visited vertex $i\in G_t$ updates, it reveals new weak neighbours. This also occurs when $i$ first becomes visited.
	\item When a weak vertex $i\in\{1,2,\ldots,N\}$ updates, we reveal it to all visited vertices in $G_t$.
\end{enumerate}
Note that Rule $(7)$ introduces a steady growth of the network, while Rule $(6)$ introduces jumps to this growth. Also new weak edges are only appended to visited vertices, since the structure around non-visited vertices has not yet been observed by the infection.\smallskip
\pagebreak[3]

All the rules described so far act either simultaneously, or following a sequential intuitive ordering. For example, {see Figure~\ref{fig:exa}, } when a visited weak vertex within $G_t$ updates, it resamples the state of the connections within $G_t$ (following Rule $(1)$), reveals new weak neighbours (Rule $(6)$), and reveals itself to non-neighbouring vertices in $G_t$ (Rule $(7)$), all at the same time.%

\begin{figure}[h]
	\begin{center}
		\includegraphics[scale=1.2]{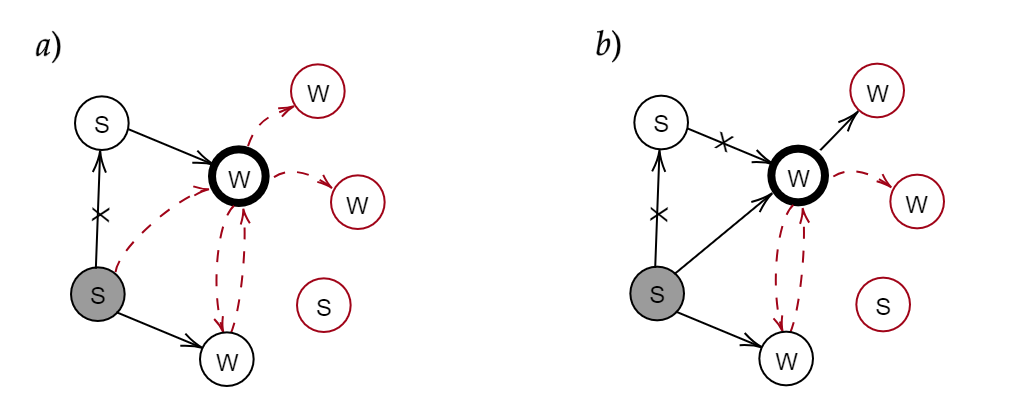}
		\caption{\small 
		$a)$ The visited weak vertex (identified with a wide border) updates, so we require to resample its connection to the strong vertex at the top left, and both reveal new weak neighbours, and be revealed to the vertices in the network. $b)$ After resampling its connection to the strong vertex, the edge becomes absent. Also, two new edges are incorporated to the network.}\label{fig:exa}
	\end{center}
\end{figure}

As a further example, when a vertex $i\in G_t$ infects some $j\notin G_t$ through a clogged edge, $j$ is added to the network and becomes visited. Consequently, according to Rule $(6)$, it reveals new weak neighbours. If $j$ is a strong slow vertex, then according to Rule $(3)$, it also reveals new strong slow neighbours. Observe that we have not provided the mechanisms involved in the saturation of vertices and edges. As 
some of the conditions for saturation require us to describe certain elements of the scoring function we will postpone it.

\subsubsection*{Starting and stopping the process}

As pointed out before, the goal of the theorem is to upper bound the extinction probability $\P_{i_0}(t<T_{\rm{ext}})$ of the contact process starting from a single infected non-star vertex. We fix such a vertex $i_0$ and {initialize $G_0$} as~follows:
\begin{itemize}[leftmargin=*]
	\item 
	{Initially, we consider all possible oriented edges $(i,j)$  in $\mathcal{E}$ as \emph{available}. }
	\item We 
start by adding $i_0$ to the network and infecting it, thus declaring it \emph{visited}.
	\item Following Rules $(3)$ and $(6)$, we reveal new weak and (possibly) strong slow neighbours~of~$i_0$, {thereby defining $G_0$ as well as $\mathfrak A_0$.}
\end{itemize}\pagebreak[3]

As mentioned at the beginning of this section, for our methods to work, we require the network to satisfy a few conditions that are likely to hold for some time. If one or more of these conditions are not met, we stop the process. Let $s$ be the function in the statement of Theorem~\ref{teoupper_optimal}, and introduce the stopping times:
\begin{itemize}
	\item[$T_{\rm{ext}}$] the extinction time of the infection process,\smallskip
	\item[$T_{\rm{hit}}$] 
	the smallest time $t$ such that
	$M(t)\ge s(a)$,\smallskip
	\item[$T_{\rm{bad}}^{\rm{tree}}$] the smallest time $t$ such that either $G_t$ contains an (undirected) cycle, or an edge $(i,j)\in G_{t^-}$ is resampled as \emph{present} (following Rule $(1)$),
	\smallskip
	\item[$T_{\rm{bad}}^{\rm{deg}}$] the smallest time $t$ such that a weak vertex $i$ is connected to at least $1/(8\lambda^2)$ weak vertices in $G_t$ through present edges. 
\end{itemize}

We also advance that there is an additional stopping time  $T_{\rm{bad}}^{\rm{conn}}$, which is far more technical and {will be defined in the next section, see~\eqref{def: Tconn}}, 
 since it requires several elements which will be introduced when defining the scoring function. The process stops at time $T$, which is defined as
\[T=T_{\rm{bad}}\wedge T_{\rm{hit}} \wedge T_{\rm{ext}},\qquad\text{ with }\qquad T_{\rm{bad}}=T_{\rm{bad}}^{\rm{tree}}\wedge T_{\rm{bad}}^{\rm{deg}}\wedge T_{\rm{bad}}^{\rm{conn}}.\]
The intuition behind the stopping time $T$ is that we run the process until either the network stops behaving nicely ($T_{\rm{bad}}$), the infection goes extinct ($T_{\rm{ext}}$), or the score becomes too large ($T_{\rm{hit}}$). 
{As will be seen later, the latter condition includes the critical event in which the infection reaches a star, or a star vertex is incorporated to $G_t$. In particular $G_t$ contains only nonstar vertices at times $t<T_{\rm{hit}}$.
Moreover, at time $t<T_{\rm{bad}}^{\rm{tree}}$, the network $G_t$ consists of a tree in which all edges point outwards} {from the root, which is} { the starting infected vertex $i_0$. Once again, we stress that the orientation of the edges is used only in the construction of $G_t$ {and the scoring}, not for the propagation of the epidemics inside $G_t$.} 
\smallskip

At this point it is worth mentioning that once a vertex $i$ becomes visited, we begin tracking its weak neighbourhood (following Rules $(6)$ {and $(7)$}). After an update of $i$, with large probability all edges incident to $i$ in $G_t$ become absent. Thus, if $i$ was healthy at the time of the update, it becomes inaccessible for the infection, never to be reached again. One might question whether tracking the weak neighbours of $i$ could prematurely halt the process by reintroducing an absent edge or creating a loop in $G_t$. However, as we demonstrate in the last section, even in such cases, the infection is more likely to reach a star or extinction before such events occur. 

\subsubsection*{Score attached to weak vertices}

We begin describing the score given to configurations as long as the network behaves nicely, that is, for all $t< T_{\rm{bad}}$. The general score $M(t)$ is divided as \[M(t)=M_{\rm{strong}}(t)+M_{\rm{weak}}(t),\] 
with $M_{\rm{strong}}(t)$ and $M_{\rm{weak}}(t)$ the scores attached to strong and weak vertices. In order to define $M_{\rm{weak}}(t)$ we {use that locally around a weak vertex the graph has a tree structure and methods from \cite{MVY13} developed for the contact process on static trees can be adapted to our situation.}
We  
begin by recalling \cite{MVY13}, where the contact process was studied on a 
tree~$\tree$ with degrees bounded by $1/(8\lambda^2)$. In this work the authors associate a scoring $W^{\rm{static}}_j(t)$ to each vertex $j$, defined as
\[W^{\rm{static}}_j(t)=\sum_{\substack{i\in \tree\\ i\text{ infected}}} (2\lambda)^{d_\tree(i,j)},\]
where $d_\tree$ is the graph distance within $\tree$. In order to prove that this scoring is a supermartingale, it is shown that each \smash{$W^{\rm{static}}_j(t)$} has average infinitesimal increase bounded by $-W^{\rm{static}}_j(t)/4$. It follows that for any choice of nonnegative weights $w_j$ on $\tree$, the process
\begin{equation}\label{eq: staticsupermartingale}
\sum_{j\in\tree} w_j W^{\rm{static}}_j(t)=\sum_{i,j\in\tree} (2\lambda)^{d_\tree(i,j)} w_j \1_{i \text{ infected at time }t}
\end{equation}
is a non-negative supermartingale. 
Observe that the process in \eqref{eq: staticsupermartingale}, being a double sum, can be written as
\[
\sum_{j\in\tree} w_j W^{\rm{static}}_j(t)=\sum_{\substack{i\in\tree\\i\text{ infected}}}M^{\rm{static}}_i(t)\]
where, for each $i\in\tree$,
\begin{equation}\label{def: dualstatic}
	M^{\rm{static}}_i(t)=\sum_{j\in\tree}(2\lambda)^{d_\tree(i,j)}w_j
\end{equation}
serves as a dual scoring function. Maybe surprisingly, we can use the same functions in the context of evolving graphs with only minor modifications. 
\smallskip

Let $t\geq0$ and $i\in G_t$ a weak vertex. We define $\tree(t)\subseteq G_t$ as the subgraph induced by the weak vertices, where we only keep the present edges (which for all purposes will be taken unoriented). The resulting network corresponds to a collection of trees with maximum degree at most $1/(8\lambda^2)$, since we assume that \smash{$t<T_{\rm{bad}}^{\rm{deg}}$},  where we define the usual graph distance $d_{\tree}(i,j)$, with the convention $d(i,j)=\infty$ if $i$ and $j$ are not connected.
\smallskip\pagebreak[3]

Observe that vertices in $\tree$ that are not labelled \emph{visited}, have been revealed as neighbours of visited vertices, but have not revealed their own neighbourhoods yet. These vertices, which we call  {\emph{pending leaves}}, prevent us from choosing {$w_j=s(j/N)$ in \eqref{eq: staticsupermartingale}}, since upon being infected they would reveal new neighbours (following Rule $(6)$) and thus increase the score in average. To prevent this, we replace $w_j$ with a slightly different scoring whose role is to overestimate the possible scoring obtained after revealing a new part of the structure. {Let $\CMI\ge 4$, which will be further specified later,  and assume that $s$ satisfies the assumptions of Theorem~\ref{teoupper_optimal} for this value of $\CMI$.}
\smallskip

We define the function
\[\tilde s_i:=\begin{cases}s_i&\mbox{ if }i\text{ is visited,}\\
\frac{\CMI}{\CMI-2}s_i&\mbox{ if }i\text{ is a pending leaf,}\end{cases}\]
where $s_i=s(i/N)$. {From the constraint $\CMI\ge 4$, we thus have $s_i< \tilde s_i\le 2s_i$.} We now define for each $j$ the individual scoring $M_j$ similarly to \eqref{def: dualstatic} as
\begin{equation}\label{eq: weakscoring}M_i(t)=\sum_{j\in \tree} (2\lambda)^{d(i,j)}\tilde s_j\end{equation}
under the assumption that $\lambda<1/2$. Observe that from this assumption and the choice $d(i,j)=\infty$ for non-connected vertices, the sum can be restricted to the connected component of $i$. We thus define
\[M_{\rm{weak}}(t)=\sum_{\substack{i\in\tree\\i\text{ infected}}}M_i(t).\]

\subsubsection*{Evolution of the process by saturation of vertices and edges}

As in the process introduced in the previous section, vertex saturation is a state akin to being infected, but in which vertices cannot recover. When a vertex becomes saturated, it remains in this state, acting as a source of infection, until it updates. Even though the intuition is the same as in the proof of Theorem~\ref{teoupper_vertex_improved}, the conditions for saturation differ, and not all vertices can become saturated. We also maintain the intuition for edge saturation, which occurs when an edge connects two saturated vertices, allowing the use of the edge for infection transmission to be ignored. Before defining the exact conditions for saturation we introduce more notation.\smallskip

{Recall the definition of $\tree$ as introduced in the previous subsection and let $\tilde\tree$ be an arbitrary subset of $\tree$. We write \begin{equation}\label{def: subscoreweak}M_{\tilde\tree}(t):= \sum_{\substack{k \in \tilde\tree\\k\text{ infected}}} M_k(t)\end{equation}
to refer to the total scoring of $\tilde\tree$. Define also $\tree_j$ as the connected component of $j$ in $\tree$ (for $j\in G_t$ weak).} With this notation, for any strong vertex $i\in G_t$, define $R_{ij}(t)$, for $j\in G_t$,~as:
\[R_{ij}(t):=\begin{cases}\;\;0&\mbox{ if }j\mbox{ is strong slow,}\\[0ex]\frac{\lambda}{\kappa_i+\kappa_j}&\mbox{ if }j\mbox{ is strong quick,}\\[1ex] \frac {16 \lambda M_{\tree_j}(t)}{s_j}&\mbox{ if }j\mbox{ is weak.}\end{cases}\]
Using these functions define
\begin{equation}\label{def: R_i(t)}R_i(t):=R_i^{\rm{strong}}(t)+R_i^{\rm{weak}}(t):=\sum_{j\,\rm{ strong }\,\sim i}R_{ij}(t)+\sum_{j\,\rm{ weak }\,\sim i}R_{ij}(t)\end{equation}
where $j\sim i$ means, as usual, that $i$ and $j$ are neighbours in $G_t$ (regardless of orientation). Finally, for a strong slow vertex $i\in G_t$, let $\Nss_i$ be the number of neighbours of $i$ in the network that are strong slow. We can now introduce the evolution rules concerning saturated vertices and edges.

\begin{enumerate}[leftmargin=*, label=Rule (\arabic*):, resume=processrules]
	\item A strong vertex $i\in G_t$ is instantly declared \emph{infected and saturated} as soon as any one of the following two conditions is met:
	\begin{itemize}
		\item[$(i)$] $R_i(t)\ge 1/4$.
		\item[$(ii)$] $i$ is incident to a weak vertex  $j$ {(through a present edge)}, and some vertex $k\in \tree_j$ infects another strong vertex $i'$.	\end{itemize}
	If $i$ is strong slow, we also declare it infected and saturated if
	\begin{itemize}
		\item[$(iii)$]  since its last update it has infected another strong slow vertex or has been infected by another strong slow vertex,
		\item[$(iv)$] $\Nss_i(t) \ge \frac {3\kappa_i \Tloc_i} {8\lambda}$, where $\Tloc_i$ is as in Definition~\ref{def: piandTloc}.
	\end{itemize} 
	\pagebreak[3]
	\smallskip
	\item Saturated vertices do not recover, and can only desaturate when updating. When a saturated vertex $i$ updates, if none of the conditions introduced in Rule $(8)$ hold, then the vertex is no longer saturated and becomes just infected.\smallskip
	\item When a strong slow vertex $i$ infects another strong slow vertex $j$, the edge between them becomes \emph{saturated}.
\end{enumerate}

	The reader might wonder why, as opposed to the proof of Theorem~\ref{teoupper_vertex_improved} we do not need to ``split" the slow strong edges within $G_t$. To answer this question recall that the splitting of edges was used to handle the possibility of reinfection between slow vertices, which for this process is impossible. We argue as follows: suppose that the strong slow edge $(i,j)$ is incorporated to the network at some time $t$, at which time $j$ must have been revealed as a (healthy) neighbour of $i$. Since the network is a tree, for $j$ to become infected for the first time, it must receive the infection from $i$. At that point, both $i$ and $j$ become saturated (following Condition $(iii)$ in Rule $(8)$). For $j$ to reinfect $i$ later on, $i$ must be healthy, and hence must have updated at some point. However, in that case, the edge $(i, j)$ would have become absent, thereby disconnecting $i$ and $j$.\smallskip

Observe that the definition of saturation in Theorem~\ref{teoupper_vertex_improved} involves conditions analogous to $(i)$, $(iii)$, and $(iv)$ in Rule $(8)$. Condition $(ii)$, however, is new and might be quite astonishing with its long distance interaction, which does not even require $i$ to be infected in the first place, { see for example Figure~\ref{fig:ii}.} It is however extremely useful, as it ensures that for a strong unsaturated vertex $i$, the weak trees $\tree_j$ attached to it are not connected to any other strong vertex from which it could receive infections. 
\begin{figure}[h]
	\begin{center}
		\includegraphics[scale=1.2]{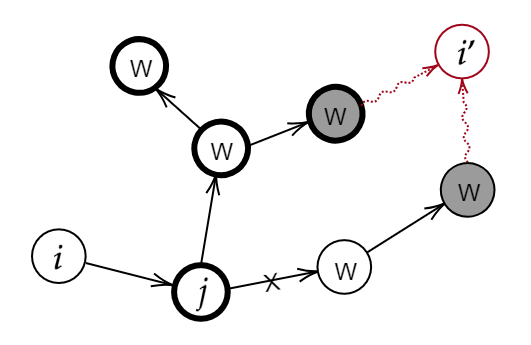}
		\caption{\small In the picture, $i$ is connected to the weak vertex $j$, and the vertices in $\tree_j$ are depicted with a wide border. There is one infected vertex in $\tree_j$ attempting to infect the strong vertex $i'$ through a \emph{clogged} edge. If it were to succeed, then $i$ would immediately become infected and saturated. Observe that the vertex at the right also could infect $i'$, but in that case nothing would happen to $i$, because of the absent edge (with an $X$) on its path to $i$.}
		\label{fig:ii}
	\end{center}
\end{figure}

A first application of Condition~$(ii)$ is the justification of the new choice for the function $R_i(t)$ in this proof. Observe that it is similar to the one introduced in the proof of Theorem~\ref{teoupper_vertex_improved}, the main difference being the introduction of the term $R_{ij}(t)=16 \lambda M_{\tree_j}(t)/s_j$ for weak vertices. Recall that  the role of $R_{ij}$ is to provide an upper bound for the probability that the vertex $j$ will later infect $i$ before the connection is refreshed. With this in mind, \begin{itemize}[leftmargin=*]
	\item suppose that $i$ is unsaturated and connected to $j$, and in what follows conditions $(i)-(iv)$ are never met {(since upon becoming saturated, infections from $j$ to $i$ are pointless)}. {Because the event in condition $(ii)$ does not occur}, the tree $\tree_j$ cannot be directly connected to any other strong vertex \smash{$i'$}, and hence in our interpretation of $R_{ij}$ we assume that the infection is ``trapped within $\tree_j$".\smallskip
	\item  In the next section we show that, {when neglecting infections between strong and weak vertices,  $M_{\rm{weak}}$ is a supermartingale, and in fact,
	\[\frac 1 {\de t}  \E[M_{\tree_j(t+\de t)}(t+\de t) -M_{\tree_j(t)}(t)\vert \F_t]\leq-\frac{1}{8}M_{\tree_j(t)}(t),\]
	where $\tree_j(t)$ and $\tree_j(t+\de t)$ are the connected components of $j$ in $\tree$ at times $t$ and $t+\de t$, respectively. Using almost identical computations it can be proved that if we allow for infections between $i$ and $j$, a similar upper bound is obtained if we replace $M_{\tree_j(t)}(t)$ by
	\[\tilde{M}(t):=M_{\tree_j(t)}(t)+ \1_{i \text{ infected at time }t}\sum_{l\in \tree} (2\lambda)^{d(l,j)-1}\tilde s_l.\]
	If we were to start the infection process at time $t$ from a given configuration in which $i$ is healthy, then, using the stopping time $\tilde{T}\geq t$ in which either the infection dies out within $\tree_j$, or $i$ becomes infected, we arrive at 
	\[\frac{s_j}{2\lambda}\P(j\text{ infects }i\,|\,\mathcal{F}_t)\leq \E(\tilde{M}(\tilde{T})\,|\,\mathcal{F}_t)\leq \tilde{M}(t)=M_{\tree_j(t)}(t),\]
	thus showing that $R_{ij}$ fits the intended role.}\smallskip
	\item When $i$ first infects $j$, the value of $R_{ij}(t)$ will be of order $\lambda$, and the same will be typically true when the infection propagates throughout $\tree_j$. This is an improvement over the expression $\lambda/(\kappa_{i}+\kappa_{j})$ {used for strong quick vertices or in the proof of Theorem~\ref{teoupper_vertex_improved}.} 
\end{itemize}

\subsubsection*{Score attached to strong vertices and edges}

We now address the remaining expression $M_{\rm{strong}}(t)$ from the decomposition $M(t)=M_{\rm{weak}}(t)+M_{\rm{strong}}(t)$, which is analogous to the scoring introduced in the proof of Theorem~\ref{teoupper_vertex_improved}, so in particular we use the same notation: We recall the functions $\pi$ and $\Tloc$ introduced in Definition~\ref{def: piandTloc},
	\begin{eqnarray*}
		\pi(x)&=& \int_0^1 \frac {p(x,y)} {\kappa(x)+\kappa(y)} \, \mathrm d y,\\
		T^{\rm{loc}}(x)&=&\max\left\{8,\frac 8 {3\kappa(x)},16 \lambda^2 \frac {\pi(x)}{\kappa(x)}\right\},
	\end{eqnarray*}
and take $a(\lambda)$ and $s(x)$ satisfying the hypotheses of Theorem~\ref{teoupper_optimal}, that is, \eqref{OMIweak}, \eqref{OMIstrongquick}, and \eqref{OMIstrongslow}, as well as assumptions (H1) and (H2). As in the previous proof we use the subscript notation when considering the corresponding functions defined on $\{1,\ldots, N\}$. We now define the scoring $M_{\rm{strong}}$ as
\begin{equation}\label{eq: strongscoring}
M_{\rm{strong}}(t):= \sum_{i\in G_t \text{ strong}}m_i(t)+ {\color{orange}\sum_{\heap{(i,j)\text{ present}}{\text{strong slow}}} s_j+t_j+u_j},
\end{equation}
with
\[
m_i(t):=\left\{ \begin{array} {ll}
	s_i+t_i  &\mbox{if } i \mbox{ is \color{magenta} infected saturated},\\
	s_i+ 4 R_i(t) t_i &\mbox{if } i \mbox{ is \color{red} infected unsaturated},\\
	2R_i(t) (s_i+t_i) &\mbox{if }i\mbox{ is \color{mygreen} healthy},
\end{array}\right.
\]
and again, {for some $r\ge3/2$ to be fixed later,}
\[
t_i:= \frac {s_i}{\ratiost \kappa_i T^{\rm{loc}}_i}, \qquad u_i=\frac {t_i} 3.
\]
Regarding the scoring we point out the following useful facts:
\begin{itemize}[leftmargin=*]
	
	\item Since $\Tloc_j\geq \frac{8}{3\kappa_i}$ and $\Tloc_j\geq 8$, it easily follows from the definition of $t_i$ that inequalities
	\begin{equation}\label{eq: adapt1}t_i\leq \frac{s_i}{4}\qquad\text{and}\qquad\kappa_it_i\leq \frac{s_i}{8r}\end{equation}
	hold as in the previous section. Similarly, as in \eqref{ineq:resaturationbyN}, for any strong slow vertex $i$,
		\begin{equation}\label{eq: adapt2}\frac {8\lambda}{3\kappa_i T^{\rm{loc}}_i} \sum_{j\le \assl N} p_{ij} \le \frac 13,\end{equation}
		by exploiting the fact that $\assl\leq \asl$. As a last adaptation of previous computations, we can require that $\CMI\geq 4r$ (by fixing $\CMI$ after fixing $\ratiost$), which gives
		\begin{equation}\label{eq: adapt3}
			\sum_{j<\assl N} p_{ij} (s_j+t_j+u_j)\le u_i,
		\end{equation}
		analogously to \eqref{eq:boundnewneigh}.
	\item As in the previous proof, the score related to the connection between strong slow vertices is attached to the strong slow edges rather than the vertices (because $R_{ij}=0$ if both vertices are strong slow). Note that this is one of the few instances where the orientation of an edge $(i,j)\in G_t$ does matter, since the score value depends on the endpoint $j$ alone. Again, this scoring reflects the idea that $j$ is very likely to get infected by $i$.
	{\item A star vertex $j$ can be incorporated to $G_t$ without being infected only through revealing a present strong slow edge $(i,j)$, with associated score $s_j+t_j+u_j>s(a)$. If this happens we instantly reach $T_{\rm{hit}}$.} 
	\item The scoring of strong vertices only depend on the weak vertices through $R^{\rm{weak}}$, while the scoring of weak vertices do not depend on the strong vertices.
	\item Non visited vertices have null score. This follows since the only non visited vertices that can be found in $G_t$ are either healthy weak vertices (for which the scoring is null), or strong slow vertices with a single neighbour $j\in G_t$, which is also strong slow (so the score is null since $R_{ij}=0$ if $j$ is strong slow). A similar analysis shows that when a strong vertex $i$ reveals new neighbours, it does not increase its score. 
	\end{itemize}

\subsection*{Supermartingale property of $M(t)$}

In this section we compute the average infinitesimal change of $M(t)$ due to the evolution of the process and show that it is a nonnegative supermartingale. Before the computations we point out two important facts:
\begin{enumerate}[leftmargin=*]
	\item The scoring $M(t)$ has only been defined at times $t$ satisfying $t<T_{\rm{bad}}$, so for the scoring to be useful, we need to define what the scoring should be when reaching a configuration giving $t=T_{\rm{bad}}$. Fortunately, since $M(t)$ is nonnegative, it is enough to set $M(T_{\rm{bad}})=0$ without any risk of losing the supermartingale property.
	\item We have not yet defined the stopping time $T_{\rm{bad}}^{\rm{conn}}$. The definition will be given at the end of this section, but we advance that it comprises the scenarios in which strong vertices are connected to too many weak trees where the infection has proliferated.
\end{enumerate} 

As in the proof of Theorem~\ref{teoupper_vertex_improved}, we show that $M(t)$ is a supermartingale by fixing $t>0$, and writing the expected infinitesimal change of the score
\begin{equation}\label{eq: infinitscore}
	\frac 1 {\de t}  \E[M(t+\de t) -M(t)\vert \F_t]
\end{equation}
as a large sum of terms associated with the effects that {a single transition event of the process have on the scoring of the vertices and edges. In what follows we assume that:
\begin{itemize}[leftmargin=*]
	\item At time $t$ we have not yet stopped the process. Even though considering times $t\geq T$ is harmless for the supermartingale property (because in that case $M(t+\de t)-M(t)=0$), it is important for our computations to obtain strictly negative upper bounds for~\eqref{eq: infinitscore}.
	\item The transition events considered preserve the tree event (that is $t+\de t<T_{\rm{bad}}$). Again, this is a harmless assumption since by definition $M(T_{\rm{bad}})=0$.
\end{itemize}}
 Analogously to the proof of Theorem~\ref{teoupper_vertex_improved}, we group these terms into \emph{contributions}, many of them identical to the ones presented in the previous section. We divide these contributions according to the types of vertices involved in the events taking place. In order to avoid cluttered notation, we will drop the dependence on time for all functions for the rest of this section, provided there is no risk of confusion.
\smallskip

$\mathbf{ (i)\;\;\text{\bf Contributions from events that only involve weak vertices}}.$

\noindent
For any fixed $t\geq0$ let $\tilde\tree$ be a connected component of $\tree$ at time $t$. Our main goal will be to show that the expressions appearing in
{\begin{equation}\label{eq: upperboundweak}
\frac 1 {\de t}  \E[M_{\rm{weak}}(t+\de t) -M_{\rm{weak}}(t)\vert \F_t]\end{equation}
as a result of events that involve vertices in $\tilde\tree$, and their interactions with other weak vertices,} can be upper bounded by $-\frac{1}{8}M_{\tilde\tree}(t)$, {where $M_{\tilde\tree}$ was defined in \eqref{def: subscoreweak}. We first observe that because we are assuming $t+\de t<T_{\rm{bad}}$, we exclude the possibility of connecting vertices in $\tilde\tree$ to other vertices in $G_t\setminus\tilde\tree$ and hence the score of the latter vertices is unaffected. On the other hand, weak vertices that are added to $\tilde\tree$ are healthy, so their score is not counted when computing $M_{\rm{weak}}(t+\de t)$. Hence, for these events it is enough to upper bound \[\frac 1 {\de t}  \E[M_{\tilde\tree}(t+\de t) -M_{\tilde\tree}(t)\vert \F_t],\]
where it is important to bear in mind that $M_{\tilde\tree}(t+\de t)$ stands for the collective score at time $t+\de t$ of the vertices that belonged to the connected component at time $t$ (the connected component might have changed at time $t+\de t$).} As with \eqref{eq: infinitscore}, the expression above corresponds to a large sum of terms, which we group depending on the type of event they originate from:\smallskip\pagebreak[3]

{$\bullet$ \emph{Infections and recoveries within $\tilde\tree$.} In this contribution we bound the average infinitesimal changes in score due to the recoveries of vertices in $\tilde\tree$, as well as the infections between these vertices. Notice that this must take into account the possible reveal of new neighbours as a result of infecting a previously non visited vertex (following Rule $(6)$). 
Fix $j\in\tilde\tree$ with $j$ infected, and consider the following events:
\begin{itemize}[leftmargin=*]
\item[-] $j$ recovers. This event occurs at rate $1$, and it results in the exclusion of $M_{j}$ from $M_{\tilde\tree}$ (all other terms remain unchanged). It follows that this event yields a total contribution
\begin{equation}\label{eq: weakcontribution1}
-\sum_{i\in\tilde\tree}\tilde{s}_i(2\lambda)^{d(i,j)}.
\end{equation}
\item[-] $j$ infects an already visited neighbour $j'\in\tilde\tree$. This event results in the inclusion of $M_{j'}$ into $M_{\tilde\tree}$, while all other terms remain unchanged. Since for each $j'\sim j$, this occurs at rate $\lambda$, the event results in the total contribution
\begin{equation}\label{eq: weakcontribution2}
	\lambda\sum_{{\substack{ { j'\sim j}\\{j'\text{ visited}}}}}\sum_{i\in\tilde\tree}\tilde{s}_i(2\lambda)^{d(i,j')}.
\end{equation}
\item[-] $j$ infects a pending leaf $j'\in\tilde\tree$. This is the most delicate event, since it has the added effect of changing the value of $\tilde{s}_{j'}$ and connecting new vertices to $\tilde\tree$.  We break down the changes in score as follows:
\begin{itemize}[leftmargin=*]
	\item[$*$] For every $k\in\tilde\tree$, the change in value of $\tilde{s}_{j'}$ implies a decrease of $M_k$ by $\frac{2}{\CMI-2}s_{j'}(2\lambda)^{d(k,j')}$.
	\item[$*$] For every $k\in\tilde\tree$, the connection of a vertex $l$ to $j'$ implies an increase of $M_k$ by
	\[\tilde{s}_{l}(2\lambda)^{d(k,l)}=\frac{2\lambda\CMI}{\CMI-2}s_l(2\lambda)^{d(k,j')}\]
	where we have used that $d(k,l)=d(k,j')+1$ and the fact that $l$ is not visited (for the definition of $\tilde{s}_l$).
	\item[$*$] Since every weak vertex $l\notin G_t$ becomes a neighbour of $j'$ with probability $p_{j',l}$, it follows that the average change of $M_k$ is upper bounded by
	\[-\frac{2}{\CMI-2}s_{j'}(2\lambda)^{d(k,j')}+\frac{2}{\CMI-2}(2\lambda)^{d(k,j')}\sum_{l\text{ weak}}\lambda\CMI p_{j',l}s_l\leq 0\]
	where we have used \eqref{OMIweak} to deduce that $\sum\lambda \CMI p_{j',l} s_l\leq s_{j'}$.
\end{itemize}
	Since the effect of revealing new vertices decreases all the $M_k$ in average, it follows that the change of $M_{\tilde\tree}$ by infecting a pending leaf can be upper bounded by the one given by the infection of an already visited vertex. We conclude that the total contribution is bounded by
\begin{equation}\label{eq: weakcontribution3}
\lambda \sum_{{\substack{ { j'\sim j}\\{j'\text{ non visited}}}}}\sum_{i\in\tilde\tree}\tilde{s}_i(2\lambda)^{d(i,j')}.
\end{equation}
\end{itemize}}

$\bullet$ \emph{Network dynamics.} In this contribution we upper bound the average changes of $M_{\tilde\tree}$ that follow from the updates of weak vertices (not necessarily in $\tilde\tree$). Observe that on these events the infection does not spread so $M_{\tilde\tree}$ does not increase by adding new terms $M_j$, but rather by modifying the values $M_i$ of all currently infected vertices in $\tilde\tree$. Fix then any such infected $i$ and  observe that as long as $t+\de t<T_{\rm{bad}}$, the updating of another $i'\in\tree$ can only decrease $M_{i}$ (because the connected component of $i$ in $\tree$ is reduced) and hence the only transitions of the network that may increase $M_i(t)$ are:
\begin{itemize}
	\item[$*$] the update of $i$, or
	\item[$*$] the update of an unseen vertex $k$ that then connects to some visited vertex $j\in \tilde\tree$, becoming a pending leaf.
\end{itemize}
In the first case, the resulting tree $\tree_i(t+\de t)$ containing $i$ will just consist of $i$ and some pending leaves $j$. The average value of $M_i(t+\de t)$ will then be
\begin{equation}\label{eq: averageM}M_i(t+\de t)\leq s_i+ \sum_{j \text{ weak}}  \frac{2\lambda\CMI}{\CMI-2} p_{i,j}s_j\leq s_i+\frac{2}{\CMI-2}s_i=\frac{\CMI}{\CMI-2}s_i\end{equation}
where in the second inequality we have used \eqref{OMIweak}. In the second case, the unseen vertex $k$ that connects to $j\in \tilde\tree$ will add to $M_i$ the quantity  $(2\lambda)^{d(i,j)+1} \frac{\CMI}{\CMI-2}s_k$. Since the update of $i$ occurs at rate $\kappa_i\leq\kappa_0$, and each $k$ attaches to a visited vertex $j$ with rate $\kappa_kp_{j,k}\leq \kappa_0p_{j,k}$ we deduce that the average change of $M_i$ due to these events is upper bounded by
\begin{align*}
	\kappa_0  & \left(\frac{\CMI}{\CMI-2}s_i-M_i(t)\right)+ \sum_{j\in \tilde\tree} \sum_{k\text{ weak}} \kappa_{0} p_{j,k} (2\lambda)^{d(i,j)+1} \frac{\CMI}{\CMI-2}s_k \\
	\le &\kappa_0 \left(\frac{\CMI}{\CMI-2}M_i-M_i\right)+ \sum_{j\in \tilde\tree} \frac{2\kappa_0}{\CMI-2}(2\lambda)^{d(i,j)}s_j=\frac{4\kappa_0}{\CMI-2}M_i,
\end{align*}
where we have used that $s_i\leq M_i$ (since $i$ is infected), as well as~\eqref{OMIweak}, and in the equality we have used the definition of $M_i$. Adding all these average changes, and taking $\CMI$ sufficiently large we obtain that the total contribution given by network dynamics events is
\begin{equation}\label{eq: contributiondynamicweak}
\sum_{\substack{i\in\tilde\tree\\i\text{ infected}}}\frac{4\kappa_0}{\CMI-2}M_i\leq \frac{1}{8}M_{\tilde\tree}.
\end{equation}
By adding the terms in \eqref{eq: weakcontribution1}, \eqref{eq: weakcontribution2}, and \eqref{eq: weakcontribution3}, for all infected vertices $j\in\tilde\tree$, as well as the contribution obtained in \eqref{eq: contributiondynamicweak}, we obtain that the expressions appearing in $\frac 1 {\de t}  \E[M_{\tilde\tree}(t+\de t) -M_{\tilde\tree}(t)\vert \F_t]$ can be upper bounded by
\begin{align*} \sum_{\substack{j\in\tilde\tree\\j\text{ infected}}} & \sum_{i\in\tilde\tree}\tilde{s}_i\left(\sum_{j'\sim j}\lambda(2\lambda)^{d(i,j')}-(2\lambda)^{d(i,j)}\right)+\frac{1}{8}M_{\tilde\tree}\\&\leq \sum_{\substack{j\in\tilde\tree\\j\text{ infected}}}\sum_{i\in\tilde\tree}\tilde{s}_i\left(\lambda(2\lambda)^{d(i,j)-1}+\frac{1}{8\lambda^2}(2\lambda)^{d(i,j)+1}-(2\lambda)^{d(i,j)}\right)+\frac{1}{8}M_{\tilde\tree}
\end{align*}
\begin{align*} \phantom{\sum_{\substack{j\in\tilde\tree\\j\text{ infected}}}} 
&= -\frac{1}{4}\sum_{\substack{j\in\tilde\tree\\j\text{ infected}}}\sum_{i\in\tilde\tree}\tilde{s}_i(2\lambda)^{d(i,j)}+\frac{1}{8}M_{\tilde\tree}=-\frac{1}{8}M_{\tilde\tree},\end{align*}
where in the inequality we have used that for any pair of vertices $i,j\in\tree$, there is at most one vertex $j'\sim j$ satisfying $d(i,j')=d(i,j)-1$, while the rest of the (at most) $\frac{1}{8\lambda^2}$ neighbours of $j$ satisfy $d(i,j')=d(i,j)+1$.\smallskip

{Now that we have shown that the events involving vertices in $\tilde\tree$ and their interactions with other weak vertices induce an average change of at most $-\frac{1}{8}M_{\tilde\tree}$ in $M_{\rm{weak}}$, we can consider all events involving only weak vertices (such that $t+\de t<T_{\rm{bad}}$) by adding all these contributions to obtain that their contribution to \eqref{eq: upperboundweak} is upper bounded by
\[-\frac{1}{8}M_{\rm{weak}}(t).\]
Similarly, it can be shown that the contribution of said events to
\[\frac 1 {\de t}  \E[M_{\tree_j(t+\de t)}(t+\de t) -M_{\tree_j(t)}(t)\vert \F_t]\]
is upper bounded by $-\frac{1}{8}M_{\tree_j(t)}(t)$. Indeed, we can repeat the exact same computations as for $M_{\tilde\tree}$, with the only exception of the case in which a vertex $j'\neq j$ in $\tree_j$ updates, which then becomes detached from $j$. We exploit this upper bound, together with the fact that for each strong vertex $i$, its scoring depends on the weak vertices only through the terms $R_{ij}$, which depend linearly on the corresponding $M_{\tree_j}$ to finally conclude that the contribution of events involving only weak vertices to the average change in the total score $M(t)$ is upper bounded by}
\begin{equation}\label{eq: totalcontributionweak}
	{\color{blue} -\frac 1 8 \sum_{\heap{i\  \text{weak}}{\text{ infected}}}M_i(t)}+{\color{red} \sum_{\heap{i \text{ strong infected}}{\text{ unsaturated}}} - \frac 4 8 R_i^{\rm{weak}}(t) t_i }+ {\color{mygreen}\sum_{\heap{i \text{ strong}}{\text{ healthy}}} -\frac 2 8 R_i^{\rm{weak}}(t)(s_i+t_i)}.
\end{equation}
\medskip

$\mathbf{ (ii)\;\;\text{\bf Contributions from events that only involve strong vertices}}.$\smallskip

Observe that events only involving strong vertices do not affect $M_{\rm{weak}}$. Indeed, the scoring of each $M_i$ depends on $\tree_i$ which is unaffected by these events, while the possible revealing of new weak vertices (as neighbours of strong ones) also leave $M_{\rm{weak}}$ unaffected, since these new vertices are healthy. It follows that the events involving exclusively strong vertices only impact $M_{\rm{strong}}$.\smallskip

 Noting the inherited rules concerning strong vertices, the almost identical score function to the one in the proof of Theorem~\ref{teoupper_vertex_improved}, as well as the bonds \eqref{eq: adapt1}, \eqref{eq: adapt2}, and \eqref{eq: adapt3}, it follows that the computations arising from these events are analogous to the ones in the previous proof, so we refrain from repeating them, although we point out that:
\begin{itemize}[leftmargin=*]
	\item When vertices update, the already revealed edges turn absent instead of unrevealed. This does not affect our computations since we always use worst-case scenario bounds. This actually improves some of the negative terms.\smallskip
	\item The contributions obtained in \eqref{neg4} and \eqref{pos2}, {which depended on clear edges, now depend on present edges $(i,j)$ where both $i$ and $j$ are strong, but where at least one of them is strong quick.} The term \smash{$\Ncl_i(t)$} appearing in both terms is now replaced by the number of all such edges incident to $i$.\smallskip
	\item The contribution computed in \eqref{pos1} still relates to the event in which an infected unsaturated strong slow vertex updates, thus revealing new neighbours. In that case only new strong slow neighbours affect the scoring.\smallskip
	\item The contribution computed in \eqref{pos3} now relates to the transmission of the infection through a strong slow edge.
\end{itemize}
Following the same computations leading to \eqref{finalboundM} we obtain a total contribution of
\begin{equation}\label{ineq: strongstrongbound}
	\begin{split}
	&{\color{red}-\sum_{{i \text{ infected}}\atop{\text{ unsaturated}}} \left(4\ratiost-\tfrac {17} {3} -\tfrac {4\ratiost} {3\CMI} \right)\tfrac {s_i} {8\ratiost}}-{\color{magenta}\sum_{{i \text{ infected}}\atop{\text{saturated}}} \left(\tfrac{1}{3}-\tfrac {4\ratiost} {3\CMI}\right)\kappa_i t_i}\\&\hspace{1cm}- {\color{mygreen}\sum_{{i \text{ healthy}}\atop{\text{strong}}} \lambda 
	\Nsq_i(t)(s_i+t_i)}-{\color{orange}\sum_{{(i,j) \text{ present}
		}\atop{\text{strong slow}}} \tfrac 2 3 (\kappa_{i}+\kappa_j)(s_j+t_j+u_j)},
	\end{split}
\end{equation}
where $\Nsq_i$ is the number of non strong slow present edges connecting $i$ to strong vertices.
\pagebreak[3]

\medskip
$\mathbf{ (iii)\;\;\text{\bf Contributions from events that involve both strong and weak vertices}}.$\medskip

The remaining transitions to consider are events that follow from infections of weak vertices by strong ones, and vice-versa. 
Moreover, we emphasize that the infection of a strong vertex by a weak one along a clogged edge can also lead to infection and saturation of other strong vertices by the ``long distance saturation'' rule. \medskip

$\bullet$ \emph{Infection of weak vertices by strong neighbours.} The infection of a weak vertex $j$ by a neighbouring strong infected vertex $i$ happens at rate $\lambda$. Following similar computations to the ones in \eqref{eq: averageM}, we deduce that $j$ increases its score in average by $M_j(t)$ (even if infecting $j$ resulted in revealing new neighbours). It follows that if $i$ is unsaturated, then it increases its score by at most $64 \lambda M_j(t) t_i/s_j$. This yields the contribution
\begin{equation}
\label{strong_weak}
{\color{red} \sum_{\heap{i \text{ strong infected}}{\text{unsaturated}}} 64\lambda^2  \sum_{\substack{j \text{ weak}\sim i\\j\text{ healthy}}} \frac {M_j(t)}{s_j} t_i} +
{\color{purple}\sum_{\heap{i \text{ strong}} {\text{infected}}}\ 
\sum_{\substack{j {\textrm{ weak}} \sim  i\\j\text{ healthy}}} \lambda M_j(t)}.
\end{equation}

$\bullet$ \emph{Infection of a strong healthy vertex by a connected weak vertex}. 
The infection of a strong vertex $i$ by a weak infected vertex $j$ through a present edge holds at rate $\lambda$, leading to the contribution
\[
\sum_{\heap{i \text{ strong}}{\text{ healthy}}}\ \ 
\sum_{j \text{ weak infected }\sim i} \lambda (s_i+t_i).
\]
Observing that $j$ is infected, we have $M_{\tree_j}\ge s_j$ and thus $R_{ij}(t)\ge 16\lambda$, we can further bound this contribution by
\begin{equation}
\label{strong_used_edge}
{\color{mygreen}
\sum_{\heap{i \text{ strong}}{\text{ healthy}}} \frac {R_i^{\rm{weak}}(t)} {16} (s_i+t_i)}.
\end{equation}

$\bullet$ \emph{Infection of a strong healthy vertex by a weak one along a clogged edge}.
The infection of a strong vertex $j$ by a weak infected vertex $i$ holds at rate $\lambda p_{ij}$ when the edge is clogged. Apart from a possible long distance effect, which is taken into account separately below, this transmission creates an average score associated to $j$ and to the strong slow vertices pending to $j$ (when $j$ is slow) at most of
\[(s_j+t_j)+u_j
\le \tfrac 4 3 s_j.
\]
Summing over $i$ and $j$ and using  the inequality~\eqref{OMIweak} available for the weak vertex $i$, we obtain a contribution bounded by
\begin{equation}
\label{infection_weak_strong}
\sum_{\heap{i \text { weak}}{\text{ infected}}} \frac 4 {3\CMI} s_i\le 
{\color{blue}
\sum_{\heap{i \text { weak}}{\text{ infected}}} \frac 4 {3\CMI} M_i(t)
}.
\end{equation}

$\bullet$ \emph{Long distance saturation.} For each strong unsaturated $i$  and each weak neighbour $j$ of $i$, a vertex $k\in \tree_j$ may infect another strong vertex (through a clogged edge), leading to infection and saturation of $i$, and thus a score increase of at most $t_i$ if $i$ infected, or $(s_i+t_i)+u_i$ if $i$ healthy (taking into account the possible revealing of strong slow vertices pending to $i$ in that case).
For each such $i$ and $j$, this holds at rate
\[
\sum_{\substack{k\in \tree_j \\ \text{infected}}} \sum_{i' \text{ strong}} \lambda p_{k,i'}
\le \sum_{\substack{k\in \tree_j \\ \text{infected}}} \lambda \int_{x'<\astr} p(\tfrac{k}N,x') \mathrm d x' \le \frac \lambda {s(\astr)} \sum_{\substack{k\in \tree_j \\ \text{infected}}} s_k \le \frac {\lambda M_{\tree_j}}{s_j}= \frac {R_{ij}}{16},\] 
where the third inequality follows 
from~(H1) 
and the last inequality from the definition of $M_{\tree_j}$ and $s_j\le s(\astr)$.
Summing over the choices of $i$ and $j$, we  obtain a contribution bounded by
\begin{equation}
\label{long_distance_saturation}
{\color{red}\sum_{\heap{i \text{ strong infected}}{\text{unsaturated}}} \frac 1 {16} R_i^{\rm{weak}}(t) t_i }+ 
{\color{mygreen}\sum_{\heap{i \text{ strong}}{\text{ healthy}}} \frac 1 8 R_i^{\rm{weak}}(t)s_i}.
\end{equation}

\medskip
$\mathbf{ (iv)\;\;\text{\bf Total contribution}}.$\medskip

Grouping the contributions obtained in \eqref{eq: totalcontributionweak}, \eqref{ineq: strongstrongbound}, \eqref{strong_weak}, \eqref{strong_used_edge}, \eqref{infection_weak_strong}, and \eqref{long_distance_saturation}, we can upper bound the infinitesimal increase of $M$ as
\begin{align*} 
&\frac 1 {\de t}\E[\left(M(t+\de t)-M(t) \right)
| \F_t] \\
&\le {\color{blue}\sum_{\heap{i \text { weak}}{\text{infected}}}I_i(t)} +
{\color{orange}\sum_{\heap{i \text{ strong}}{\text{\rm{pending}}}} I_i^{\rm{pend}}(t)}
+  \!\!\!
{\color{red}\sum_{\heap{i \text{ strong infected}}{ \text{unsaturated}}}I_i^{\rm{inf}}(t) }+
{\color{mygreen} 
\sum_{\heap{i \text{ strong}}{\text{healthy}}}I_i^{\rm{hea}}(t)}
+ \!\!\!
{\color{magenta}\sum_{\heap{i \text{ strong infected}}{\text{saturated}}} I_i^{\rm{sat}}(t) 
}, 
\end{align*}
where the expressions of the terms inside the sums are given by
\begin{align*}
{\color{blue}I_i(t)}&= -\left(\tfrac 1 8- \tfrac 4 {3\CMI} \right) M_i(t)\\
{\color{orange}I_i^{\rm{pend}}(t)}&= - \tfrac 2 3 \kappa_i (s_i+t_i+u_i).\\
{\color{red}
	I_i^{\rm{inf}}(t)}&=- \left(4\ratiost-\tfrac {17} {3} -\tfrac {4\ratiost} {3\CMI} \right)\tfrac {s_i} {8\ratiost} 
+\sum_{j {\text{ weak}}\sim i} \lambda M_j(t) 
+ 64\lambda^2t_i \sum_{j {\text{ weak}}\sim i} \tfrac {M_j(t)}{s_j} - \tfrac 7 {16} R^{\rm{weak}}_i(t) t_i,\\
{\color{mygreen}I_i^{\rm{hea}}(t)}&=-\left(\lambda \Nsq_i(t) +  \tfrac 1 {16} R^{\rm{weak}}_i(t)\right) (s_i + t_i),\\
{\color{magenta}I_i^{\rm{sat}}(t)}&=-\left(\tfrac 1 3 - \tfrac {4\ratiost} {3 \CmMI}\right) \kappa_i t_i + \sum_{j {\text{ weak}}\sim i} \lambda M_j(t).
\end{align*}
Taking $\ratiost$ and $\CmMI$ large with still $\CmMI> 4\ratiost,$ we have that the terms $\sfrac 1 8-\sfrac 4 {3\CMI}$, $\sfrac 1 3-\sfrac {4\ratiost}{3\CMI}$ and $4\ratiost-\sfrac {17} 3-\sfrac {4\ratiost}{3\CMI}$ are all positive, as well as 
$4\ratiost-6- \sfrac {4\ratiost}{3\CMI}$. Define the expressions,
\begin{align*}
J_i(t):= \sum_{j {\text{ weak}}\sim i} \lambda M_j(t), \qquad K_i(t):=\sum_{j {\text{ weak}}\sim i} \frac  {M_j(t)}{s_j}
\end{align*}
and note that as soon as there is some $\eps>0$ such that the inequalities
\begin{align}
\label{LD_for_J} J_i(t)&\le \left( \tfrac 1 3 - \tfrac {4\ratiost} {3\CmMI}-\eps \right) \kappa_i t_i, \\
\label{LD_for_K} K_i(t)&\le \left(4 \ratiost - 6 - \tfrac {4\ratiost} {3\CmMI}-\eps \right) \tfrac {\pi_i} {32},
\end{align}
hold, then \eqref{LD_for_J} ensures that $I_i^{\rm{sat}}(t)\le -\eps \kappa_i t_i$ as well as 
$$J_i(t)\le \tfrac {\kappa_i t_i} 3=\tfrac {s_i}{3\ratiost \Tloc_i}\le \tfrac 1 3 \tfrac {s_i}{8\ratiost},$$ 
while \eqref{LD_for_K} ensures that 
\smash{$64\lambda^2 t_i K_i(t)\le (4\ratiost-6 -\sfrac {4\ratiost} {3\CMI}-\eps)\sfrac {s_i} {8\ratiost}$} and further $I_i^{\rm{inf}}(t)\le -\eps s_i/8\ratiost$, so in particular all the contributions become negative. With this in mind we define at last the missing stopping time $T_{\rm{bad}}^{\rm{conn}}$ as
\begin{equation}\label{def: Tconn}T_{\rm{bad}}^{\rm{conn}}=\inf\left\{\begin{array}{c}t\geq 0,\; \text{\eqref{LD_for_J} or~\eqref{LD_for_K}
	fails to hold for at least}\\ \text{one vertex $i$ amongst the visited strong vertices}\end{array}\right\},\end{equation}
{for some values of $\ratiost$, $\CMI$ and $\eps>0$ that are still not fixed, but satisfy all previously mentioned constraints}. We then have:
\begin{equation}\label{eq:negative_infinitesimal_increase}
\frac 1 {\de t}  \E[M_{t+\de t} -M_t\vert \F_t]\le - c_\eqref{eq:negative_infinitesimal_increase} \one_{t<T},
\end{equation}
for some $c_\eqref{eq:negative_infinitesimal_increase}>0$ independent of $i_0$, $t$, and $N$. We therefore conclude that both $(M(t\wedge T))_{t\ge 0}$ and $(M(t\wedge T)+c_\eqref{eq:negative_infinitesimal_increase}(t\wedge T))_{t\ge 0}$ are supermartingales. 

\subsection*{Exploiting the supermartingale}

As in the previous proof, our aim is to exploit the supermartingale property to prove Theorem~\ref{teoupper_optimal}. First, we bound $\P_i(t<T_{\rm{ext}})$ as follows:
\begin{align*}
\P_i(t<T_{\rm{ext}})&\le \P_i(t\le T)+ \P_i(T= T_{\rm{hit}})+\P_i(t\le T_{ext},T<t,T\ne T_{\rm{hit}})\\
&\le \P_i(t\le T)+ \P_i(T= T_{\rm{hit}})+\P_i(T_{\rm{bad}}< T_{\rm{hit}}\wedge t).
\end{align*}
As in the previous proof, $M(0)$ is not deterministic under $\P_i$ but satisfies
\smash{$s_i\le \E_i[M(0)]\le \frac 76 s_i$},
so the optional stopping theorem applied to $M(\cdot \wedge T)$ implies
$
\P_i(T= T_{\rm{hit}})\le \frac {7 s_i}{6s(a)}.
$
Moreover, with the  optional stopping theorem applied to $M(\cdot \wedge T)+ c_\eqref{eq:negative_infinitesimal_increase} (\cdot \wedge T)$, we get
\[
\P_i(t<T)\le \frac 1 t \E_i[T]\le \frac 1 {c_\eqref{eq:negative_infinitesimal_increase} t} \E_i[M_0]\le \frac {7 s_i} {6c_\eqref{eq:negative_infinitesimal_increase} t}.
\]
Using~\eqref{ineq:I_N_byduality}, we now obtain
\begin{equation}\label{ineq:fullboundIN}
I_N(t)\le a+\tfrac 1 N+\tfrac 7 6 \Big(\tfrac 1 {s(a)}+\tfrac 1 {c_\eqref{eq:negative_infinitesimal_increase} t}  \Big)\int_a^1 s(y)dy+\frac 1 N \sum_{i=\lceil aN\rceil+1}^N \P_i(T_{\rm{bad}}< T_{\rm{hit}} \wedge t).
\end{equation}
This expression resembles~\eqref{boundIn}, except for the last term. In Subsection~\ref{sec:stopping_process}, we will prove that we can bound this term at time $t=t_0:=7s(a)/c_\eqref{eq:negative_infinitesimal_increase}$ as
\begin{equation}\label{Bound_on_bad_stopping}
\frac 1 N \sum_{i=\lceil aN\rceil+1}^N \P_i(T_{\rm{bad}}< T_{\rm{hit}} \wedge t_0)\le a+ \frac {\omega'}N,
\end{equation}
for some choice of $\omega'=\omega'(\lambda)$, and small $\lambda$.
Taking this for granted, we continue as follows:
\begin{itemize}[leftmargin=*]
	\item For $t\le t_0$, we bound $\P_i(T_{\rm{bad}}< T_{\rm{hit}} \wedge t)$ by $\P_i(T_{\rm{bad}}< T_{\rm{hit}} \wedge t_0)$ to obtain
	\[
	I_N(t)\le 2a+ \frac 7 {6 s(a)}\int_a^1 s(y)dy+ \frac {\omega}{t} +\frac {1+\omega'}N,
	\]
	with $\omega=\omega(\lambda)=\tfrac 7 {6c_\eqref{eq:negative_infinitesimal_increase}} \int_a^1 s(y)  \mathrm dy$.
	\item For $t\ge t_0$, we use that $I_N$ is nonincreasing to bound $I_N(t)$ by $I_N(t_0)$, which, using again~\eqref{Bound_on_bad_stopping} and the equality $\frac 76(\frac 1 {s(a)}+\frac 1 {c_\eqref{eq:negative_infinitesimal_increase} t_0})=\frac 4 {3s(a)}$, yields
	\[
	I_N(t)\le 2a+ \frac 4 {3 s(a)}\int_a^1 s(y)dy+\frac {1+\omega'(\lambda)} N.
	\]
\end{itemize}
Thus \eqref{boundIn} is satisfied for all $t\ge 0$, showing Theorem~\ref{teoupper_optimal}. It only remains to prove~\eqref{Bound_on_bad_stopping}.



\subsection{Bound on the probability of stopping the exploration early.}\label{sec:stopping_process}

We will provide a bound of the form
\begin{equation}
\label{Bound_on_bad_stopping_unif}
\P_i(T_{\rm{bad}}< T_{\rm{hit}} \wedge t_0)\le o(a(\lambda))+ \frac {\omega'(\lambda)}N,
\end{equation}
valid for all small $\lambda$ and large $N$, and uniformly over $i\in \{\lceil aN\rceil+1,\ldots, N\}$, from which~\eqref{Bound_on_bad_stopping} naturally follows. Recall the definition of $t_0= t_0(\lambda)= 7s(a)/c_\eqref{eq:negative_infinitesimal_increase}$ and \smash{$T_{\rm{bad}}=T_{\rm{bad}}^{\rm{tree}}\wedge T_{\rm{bad}}^{\rm{deg}}\wedge T_{\rm{bad}}^{\rm{conn}}$}, and introduce further the following stopping times:
\smallskip

\begin{itemize}[leftmargin=*]
	\item $T_\ell$, for $\ell\ge1$, the time of the $\ell$-th infection, where an infection means a vertex transitioning from the state healthy to infected (and possibly saturated). Note that two infections can occur simultaneously (by the saturation rules $(8)(i)$ or $(8)(ii)$), {in which case we count the two infections, so we might have $T_\ell=T_{\ell+1}$ for example}. A single vertex can also be infected multiple times. In any case, we can roughly bound the number of visited vertices at times $t<T_\ell$ by $\ell$.\smallskip
	
	\item $\tilde T_{\rm{deg}}$, the smallest time when a visited strong vertex is connected to more than $2 c_2 a^{-\gamma}$ weak vertices in $G_t$. Note that when incorporated to the network, a strong vertex $j$ was possibly connected to a weak vertex $i$ that has infected it (so that $(i,j)$ has been incorporated to $G_t$). However, its connections to the other weak vertices follow the same dynamics as in the original stationary dynamic network (by rules (6) and (7)). It follows that at any given time $t$, we can use~\eqref{condp} to bound the average number of present weak edges $(j,k)$ by $c_2 a^{-\gamma}$. Being connected to more than $2 c_2 a^{-\gamma}$ is thus an unlikely event, which will help us to control the more involved quantities $J_j(t)$ and $K_j(t)$ appearing in~\eqref{LD_for_J} and~\eqref{LD_for_K}.
\end{itemize}
We have
\begin{align*}
\P_i(T_{\rm{bad}}<T_{\rm{hit}}\wedge t_0)&\le \P_i(T_\ell\wedge T_{\rm{bad}}^{\rm{tree}}\wedge T_{\rm{bad}}^{\rm{deg}}\wedge \tilde T_{\rm{deg}}\wedge T_{\rm{bad}}^{\rm{conn}}< T_{\rm{hit}}\wedge t_0)\\
&\le P^{\rm{infected}} +P^{\rm{tree}}+P^{\rm{deg}}\ +P^{\rm{conn}},
\end{align*}
where the second line is obtained by looking at which of the stopping times is smallest, and the terms are defined as
\begin{align*}
P^{\rm{infected}}&:= \P_i(T_\ell\le T_{\rm{bad}}\wedge T_{\rm{hit}}\wedge t_0),\\
P^{\rm{tree}}&:=\P_i(T_{\rm{bad}}^{\rm{tree}}<T_\ell\wedge T_{\rm{hit}}\wedge t_0),\\
P^{\rm{deg}}&:= \P_i(T_{\rm{bad}}^{\rm{deg}}\wedge \tilde T_{\rm{deg}}<T_\ell\wedge T_{\rm{bad}}^{\rm{tree}}\wedge T_{\rm{hit}} \wedge t_0), \\
P^{\rm{conn}}&
:=\P_i(T_{\rm{bad}}^{\rm{conn}}<T_\ell\wedge T_{\rm{bad}}^{\rm{tree}} \wedge T_{\rm{bad}}^{\rm{deg}}\wedge \tilde T_{\rm{deg}}\wedge T_{\rm{hit}}\wedge t_0).
\end{align*}
We now prove that
\begin{gather*}
P^{\rm{tree}}\le \tfrac {\omega'(\lambda)} N,\\
P^{\rm{infected}}+ P^{\rm{deg}}+ P^{\rm{conn}}= o(a(\lambda)),
\end{gather*}
when $\ell$ is chosen as 
\begin{equation}\label{ineq:k_limit}
\ell=\ell(\lambda)=\lambda^{-C_{\eqref{ineq:k_limit}}}
\end{equation}
with $C_{\eqref{ineq:k_limit}}>0$ defined below.

\pagebreak[3]

\subsection*{Bounding the total number of infections}

We prove
\begin{equation}\label{ineq:Pinfected}
P^{\rm{infected}}\le \frac {2+4s(a)^2}{\ell}.
\end{equation}
From now on, we fix a sufficiently large value of \smash{$C_{\eqref{ineq:k_limit}}$} to ensure this is $o(a)$. This is possible as $a^{-1}$ and $s(a(\lambda))$ are growing polynomially in $\lambda^{-1}$. 
\medskip

\noindent\emph{Proof.}
	 The idea is to associate each infection with a positive increase of value at least $1/2$ for 
		the supermartingale $M(t\wedge T)$, and then use standard martingale theory to obtain the desired upper bound. Indeed, the infection of a weak vertex $i$ implies a score increase of at least $s_i\ge 1$, while the infection of a strong vertex $i$ implies an increase of its associated score $m_i$ of at least $s_i/2\ge 1/2$. This simple idea however has to take into consideration several subtleties.
		\begin{itemize}[leftmargin=*]
			\item The infection of a strong pending vertex possibly implies a \emph{decay of the total score}. This is due to the fact that while the vertex becomes infected (and saturated), the edge simultaneously also becomes saturated, which implies a decay of its associated score. For this reason, in this proof, we will consider a strong pending vertex as already infected when counting the number of ``infections''.
			\item By the previous consideration, when a strong vertex $i$ reveals a strong pending leaf~$j$, we count this as an infection, which we also want to associate to an increase of score. However, it might be the case that $i$ was itself a strong pending leaf, that was just infected by another strong slow edge $(k,i)$. In that case, the following events occur simultaneously:
			\begin{itemize}[leftmargin=*]
				\item $i$ gets infected and saturated, with associated score $s_i+t_i$.
				\item $k$ also gets saturated (if it were not before), but for simplicity we neglect this effect from this discussion.
				\item The edge $(k,i)$ becomes saturated, implying a score decrease of $s_i+t_i+u_i$. These first events together thus imply possibly a score decrease of $u_i$.
				\item The vertex $i$ has revealed the pending strong slow vertex $j$ (and possibly other strong slow and weak vertices as well). The  increase in score associated to this newly revealed strong slow vertex is $s_j+t_j+u_j$.
			\end{itemize}
			The problem is that the increase $s_j+t_j+u_j$ might be compensated by the decrease $u_i$, so it is not clear whether we can associate a score increase to the revelation of~$j$.%
			\smallskip%
			\end{itemize}
			
			In order to overcome this difficulty, we define a discrete-time supermartingale $(\Md_n)_{n\ge 0}$, that evolves as the jump process associated to the supermartingale $M_t$, except for the particular case in which a strong slow vertex $i$ is infected by another strong slow vertex~$k$. In that case, the discrete-time process $\Md$ decomposes the jump $\Delta M_t$ into several jumps as follows:
			\begin{itemize}[leftmargin=*]
				\item As a first jump of $\Md$, it increases due to the change in score given by the saturation of $i$, $k$, and $(k,i)$, the revelation of new weak neighbours of $i$, and an additional term $\sum_{j\text{ strong slow}}p_{ij}(s_j+t_j+u)$ representing the average increase of $M_t$ resulting from the revelation of new strong slow neighbours. Equivalently, this term can be seen as the result of assigning a temporary score $p_{i,j} (s_j+t_j+u_j)$ to each strong slow edge $(i,j)$, before attempting to reveal it as in Rule $(3)$. 
				\item For the remaining steps associated to this event, we successively attempt to reveal each strong slow edge $(i,j)$ 
				(starting from the smallest $j$ or in other arbitrary order), and replace the score $p_{i,j} (s_j+t_j+u_j)$ by either $(s_j+t_j+u_j)$ or $0$, according to whether the edge was revealed or not. 
			\end{itemize}
			Once the vertex $i$ has gone through all the strong slow vertices, the two processes coincide again (via a time-change that we do not write down explicitly).	This provides the construction of $(\Md_n)_{n\ge 0}$. Note that:
			\begin{itemize}[leftmargin=*]
				\item When we successfully reveal a strong slow neighbour $j$, this implies a positive jump $\Delta \Md_n=(1-p_{i,j})(s_j+t_j+u_j)\ge 1/2$, using that $p_{i,j}\le p(a,a)/N\le 1/2$ for large $N$. Actually it is now true that any ``infection'' that we count in this proof is associated to an increase $\Delta \Md_n\ge 1/2$. We might also have two simultaneous infections, or reveal several strong pending vertices when a strong slow vertex updates, but in each case the total number of simultaneous ``infections'' is bounded by $2\Delta \Md_n$. 
				\item The first jump of $\Md$ is on average equal to $\Delta M_t$, while the rest of the jumps have mean $0$. Doing this, we maintain the supermartingale property of the process, so $\Md$ is indeed a supermartingale.
				\item The process $\Md$ takes nonnegative values, and it can reach a value larger than $s(a)$ only at a moment corresponding to the time $T_{\rm{hit}}$ when $M_t$ also reaches such a value.\smallskip
			\end{itemize}
			Writing $N_{\rm{hit}}:=\inf\{n\ge 0, \Md_n>s(a)\}$ and \smash{$Z_n:=(s(a)-\Md_{n\wedge N_{\rm{hit}}})_+$}, we have that $Z_n$ is a submartingale with values in $[0,s(a)]$. The total number of infections occurring before time $T$ can now be bounded (using also $T\le T_{\rm{hit}}$), by
			\[
			\sum_{n<N_{\rm{hit}}} 2 \Delta \Md_n \1_{\Delta \Md_n\ge 1/2}\le 4 \sum_{n<N_{\rm{hit}}}  \left(\Delta \Md_n\right)^2.
			\]
			Noting that we have $\Delta Z_n=-\Delta \Md_n$ for $n<N_{\rm{hit}}$, we can further bound the expected number of infections occurring before time $T$ by
			\[
			4\E\Big[\sum_{n<N_{\rm{hit}}} \left(\Delta Z_n\right)^2\Big]\le 4\E[Z_\infty^2-Z_0^2]\le 4s(a)^2.
			\]
			We also have at most two additional infections occurring at time $T$.
			Then~\eqref{ineq:Pinfected} follows by the Markov inequality.\hfill\qed

\subsection*{Bounding the probability of breaking the tree structure}

We now prove 
\[
P^{\rm{tree}}
\le 8\ell^2
(2\kappa_0t_0+\lambda t_0+\ell+1)^2\frac{p(a,a)^2}{N}=: \frac {\omega'(\lambda)}N.
\]

\begin{proof}
	Introduce, in this proof only, the notation $\tilde T:=T_k\wedge T_{\rm{hit}}\wedge t_0$, {and the quantities
	\begin{align*}
		E^{\rm{tree}}_{\rm{rep}}&=\bigg\{\parbox{10cm}{$(i,j)\in \{1,2,,\ldots,N\}^2$, such that $(i,j)$ becomes present at least twice before time $\tilde{T}$}\bigg\}\\[3pt]
		V^{\rm{tree}}_{\rm{cycle}}&=\bigg\{\parbox{10cm}{$j\in \{1,2,,\ldots,N\}$, there are $i,i'\in\{1,2,\ldots N\}$ such that both $(i,j)$ and $(i',j)$ become present 
		before time $\tilde{T}$}\bigg\}
	\end{align*}
	It follows that $T_{\rm{bad}}^{\rm{tree}}<\tilde T$ if and only if $|E^{\rm{tree}}_{\rm{rep}}|+|V^{\rm{tree}}_{\rm{cycle}}|\geq 1$ so from the Markov inequality, we can bound
	\[P^{\rm{tree}}\leq\E[|E^{\rm{tree}}_{\rm{rep}}|]+\E[|V^{\rm{tree}}_{\rm{cycle}}|].\]
	The quantities in both expectations involve counting edges that become present, so we stress that in order for an edge $(i,j)$ to become present at time $t<\tilde{T}$:
	\begin{itemize}[leftmargin=*]
		\item The edge $i$ must be visited at time $t$.
		\item It must be the case that at time $t$ either:
		\begin{itemize}
			\item $(i,j)$ is a strong quick edge and $i$ is infecting $j$ (which was previously not in the network), or
			\item $i$ or $j$ updates, or
			\item $i$ is a strong slow neighbour which becomes saturated due to Rule $(8.iii)$.
		\end{itemize}
	\end{itemize}\pagebreak[3]
	
	Denote by  $\tau_i$ the time at which vertex $i$ becomes visited, and $N_{\rm{visited}}$ the number of visited vertices by time $\tilde{T}$. We now bound $|E^{\rm{tree}}_{\rm{rep}}|$ and $|V^{\rm{tree}}_{\rm{cycle}}|$ separately:
	\begin{itemize}[leftmargin=*]
		\item \emph{Repeated edges}: Start from the observation
		\[\E[|E^{\rm{tree}}_{\rm{rep}}|]=\sum_{\text{ edges }(i,j)}\P(\tau_i<\tilde{T}\text{ and $(i,j)$ becomes present at least twice before time $\tilde{T}$})\]
		and observe that, conditionally on $\tau_i$:
		\begin{itemize}
			\item The number of updates of $(i,j)$ in $(\tau_i,\tilde{T}]$ that result in the edge being present is bounded by an independent Poisson random variable with rate $(\kappa_i+\kappa_j)(\tilde{T}-\tau_i)p_{i,j}\leq 2\kappa_0t_0p_{i,j}$.
			\item Since $\tilde{T}\leq T_\ell$, there are at most $\ell$ infections within $[\tau_i,\tilde{T}]$ which could reveal $(i,j)$ due to Rule $(8.iii)$, and at each such infection an independent coin with probability $p_{i,j}$ is tossed to decide whether the edge becomes present.
			\item The number of possible events in which $(i,j)$ is infecting $j$ through a clogged edge is bounded by a Poisson random variable with rate $\lambda t_0p_{i,j}$.
		\end{itemize}
		We thus deduce that, conditionally on $\tau_i$, the probability of $(i,j)$ becoming present at least twice is bounded by
		\[2(2\kappa_0t_0+\lambda t_0+\ell+1)^2p_{i,j}^2\leq 2(2\kappa_0t_0+\lambda t_0+\ell+1)^2\frac{p(a,a)^2}{N^2},\]
		and as there are at most $\ell$ infections before time $\tilde{T}$, it must be that $N_{\rm{visited}}\leq \ell$,~so
		\begin{align*}\E[|E^{\rm{tree}}_{\rm{rep}}|]&\leq \sum_{\text{ edges }(i,j)}\P(\tau_i<\tilde{T})2(2\kappa_0t_0+\lambda t_0+\ell+1)^2\frac{p(a,a)^2}{N^2}\\[1pt]&= \sum_{i=1}^N\P(\tau_i<\tilde{T})2(2\kappa_0t_0+\lambda t_0+\ell+1)^2\frac{p(a,a)^2}{N}\\[1pt]&\leq 2\ell(2\kappa_0t_0+\lambda t_0+\ell+1)^2\frac{p(a,a)^2}{N}\end{align*}
	\item \emph{Cycles}: This case is similar to the previous one, where we can write
	\[\E[|V^{\rm{tree}}_{\rm{cycle}}|]=\sum_{i,j,l}\P(\tau_i,\tau_j<\tilde{T},\text{ and both $(i,l)$ and $(j,l)$ become present before time $\tilde{T}$}).\]
	Following the same line of thought as in the previous case, we deduce that, conditionally on $\tau_i$ and $\tau_j$, the probability of both $(i,l)$ and $(j,l)$ becoming present is bounded by
	\[4(2\kappa_0t_0+\lambda t_0+\ell+1)^2p_{i,l}p_{j,l}\leq 4
	(2\kappa_0t_0+\lambda t_0+\ell+1)^2\frac{p(a,a)^2}{N^2},\]
	and hence we obtain
	\begin{align*}
		\E[|V^{\rm{tree}}_{\rm{cycle}}|]&\leq \sum_{i,j,l}\P(\tau_i,\tau_j<\tilde{T})4
		(2\kappa_0t_0+\lambda t_0+\ell+1)^2\frac{p(a,a)^2}{N^2}\\[1pt]&=\sum_{i,j}\P(\tau_i,\tau_j<\tilde{T})4
		(2\kappa_0t_0+\lambda t_0+\ell+1)^2\frac{p(a,a)^2}{N}\\[1pt]&\leq 4\ell^2
		(2\kappa_0t_0+\lambda t_0+\ell+1)^2\frac{p(a,a)^2}{N}.
	\end{align*}
	\end{itemize}
The conclusion follows by adding both bounds.}
\end{proof}


\subsection*{Bounding the degrees} 

We prove below the upper bounds
\begin{align}
\label{Pdegree} P^{\rm{deg}}&\le  \ell t_0 e^{-c_\eqref{Pdegree}\lambda^{-2}},\end{align}
with $c_\eqref{Pdegree}>0$ a constant independent of $\lambda$ or $N$. As $\ell$ and $t_0$ are polynomial in $\lambda$, this term is exponentially small in $\lambda$, and in particular negligible compared to $a$ as $\lambda \to 0$.

\begin{proof}
	At times $t<T_{\rm{bad}}^{\rm{tree}}$, each vertex $j$ in $G_t$ (except $i_0$) is the head of one edge $(i,j)\in G_t$, with possibly $i$ weak. Once $j$ is infected, every other weak edge $(j,k)$ is revealed and evolves as in the true stationary dynamic network. Therefore, conditionally given $j$ is in $G_t$, the number $N_j^{\rm w}(t)$ of present weak edges $(j,k)$ in $G_t$, follows a Poisson binomial distribution, with expectation bounded by $1/10\lambda^2$ if $j$ is itself a weak vertex. 	Therefore, for given $t$, the probability of having $N_j^{\rm w}(t)\ge 1/9\lambda^2$ 
	is uniformly bounded by $e^{-c\lambda^{-2}}$ for some constant $c>0$. Moreover, the updating rates of the vertices are upper bounded, therefore one can define a time 
	step~$\tau$ such that for any given $s$,
	\[
	\P\left(\min \{N_j^{\rm w}(u) \colon s\le u\le s+\tau\}\ge 1/9\lambda^2 \ \Big|\ N_j^{\rm w}(s)\ge 1/8\lambda^2\right)\ge \frac 1 2.
	\]
	It follows that we have 
	\[
	\P\left(\max \{N_j^{\rm w}(s) \colon s\le t\}\ge 1/8\lambda^2\ \Big|\ j\in G_t \right)\le \frac {2 t}{\tau} e^{-c\lambda^{-2}}.
	\]
	Now, we inspect the successively infected weak vertices one by one up until time $T_\ell\wedge T_{\rm{bad}}^{\rm{tree}}\wedge T_{\rm{hit}}\wedge t_0$. We find at most $\ell$ such vertices, for which we can use the previous bound with $t=t_0$. Thus the probability of finding a weak vertex with at least $1+\lfloor 1/8\lambda^2 \rfloor$ weak neighbours is bounded by \smash{$\frac {2 \ell t_0}{\tau} e^{-c\lambda^{-2}}$}.
	Similarly, for a strong vertex $j>aN$ and conditionally given  $j\in G_t$, the number of present weak edges $(j,k)$ in $G_t$ follows a Poisson binomial distribution with expectation bounded by~$c_2 a^{-\gamma}$, and we deduce that the probability of finding a strong vertex connected to more than $2 c_2 a^{-\gamma}$ weak vertices is also bounded by  \smash{$\frac {2 \ell t_0}{\tau} e^{-c\lambda^{-2}}$} (with a possibly different value of $c>0$). We deduce~\eqref{Pdegree} by decreasing the value of $c>0$ if necessary. \end{proof}

\subsection*{Bounding the connectivities of visited strong vertices}

It finally remains to bound the probability $P^{conn}$ to visit a strong vertex $i$ that fails to satisfy~\eqref{LD_for_J} or \eqref{LD_for_K} by having a too large value for $J_i(t)$ or $K_i(t)$. This last bound is more technical than the previous ones, and before tackling it, we first provide an upper bound for quantities that can be thought of at upper bounds for $J_i(t)$ and $K_i(t)$ in a fixed idealized environment.

\subsubsection*{A large deviation upper bound in an idealized environment.}

Fix a strong vertex $i$, and introduce a family $(C_{j_1,\ldots,j_n})_{n\ge 1, j_1,\ldots,j_n \text{ weak}}$ of independent random variables, indexed by the sequences of distinct weak vertices, such that $C_{j_1}$ is Bernoulli distributed with parameter $p_{i,j_1}$ and,  for $n\ge2$,  $C_{j_1,\ldots,j_n}$ is Bernoulli distributed with parameter $p_{j_{n-1},j_n}$. Define
\[\bar J_i = \sum_{j\text{ weak}} \lambda  C_j \bar M_j, \qquad \bar K_i = \sum_{j\text{ weak}} \frac {C_j}{s_j} \bar M_j,
\]
where $\bar M_j$ is defined by
\[
\bar M_j:= 2 s_j + 2 \sum_{n\ge 1}\sum_{j_1,\ldots,j_n \text{ weak}}\left(\prod_{k=1}^n 2\lambda C_{j,j_1,\ldots,j_k} \right)s_{j_n}.
\]

\begin{proposition}\label{exponential_moments}
	For $\CMI$ sufficiently large, there exists a finite constant $C_{\ref{exponential_moments}}$ such that the following results hold:
	\begin{enumerate}
		\item 
		for any $j$ weak and any $\theta\ge 0$ satisfying $\theta\le 1/(2\lambda s(\astr))$ and $\theta \le s_j^{-1}$, we have:
		\[
		\E[e^{\theta  \bar M_j}]\le 1 + C_{\ref{exponential_moments}} \theta s_j.
		\]
		\item 
		for any $j$ weak and any $\theta\ge 0$ satisfying $\theta\le 1/(2\lambda s(\astr))$, we have 
		\[
		\E[e^{\theta \bar M_j}]\le e^{ C_{\ref{exponential_moments}} \theta s_j}.
		\]
		\item 
		for any $i$ strong and any $\theta\ge 0$ satisfying $\theta\le  1/(\lambda s(\astr))$, we have 
		\[
		\E[e^{\theta \bar J_i}]\le e^{ C_{\ref{exponential_moments}} \theta \sigma_i}\le e^{\frac {C_{\ref{exponential_moments}} \theta} \CMI \,  \frac{ s_i}{T^{\rm{loc}}_i}  },
		\]
		where  $\sigma_i=\lambda \sum_{j \text{weak}} p_{ij} s_j$. 
		\item 
		for any $i$ strong and any $\theta\ge 0$ satisfying $\theta \le 1$ and $\theta\le  1/(2\lambda s(\astr))$, we have
		\[
		\E[e^{\theta \bar K_i}]\le e^{ \frac {C_{\ref{exponential_moments}}}{2 \kappa_0} \theta \tilde \sigma_i}\le e^{C_{\ref{exponential_moments}} \theta \pi_i},
		\]
		where $\tilde \sigma_i= \sum_{j \text{weak}} p_{ij}$.
		\end{enumerate}
\end{proposition}

\begin{proof}[Proof of Proposition \ref{exponential_moments}]
We focus on the first point of the proposition. For $j$ weak and $n\ge0$, we introduce
\[
\bar M_j^{(n)}:= 2 s_j + 2 \sum_{k=1}^n \sum_{j_1,\ldots,j_k \text{ weak}}\left(\prod_{l=1}^k 2\lambda C_{j,j_1,\ldots,j_l} \right) s_{j_k},
\]
so that  $\bar M_j^{(0)}= 2 s_j$ and $\bar M_j =\lim \bar M_j^{(n)}$. We prove recursively on $n$ the existence of a constant $C_{\ref{exponential_moments}}>0$ independent of $n$ such that, for any $j$ weak and any $\theta\ge 0$ satisfying $\theta\le 1/(2\lambda s(\astr))$ and $\theta \le s_j^{-1}$, we have
\[
\E[e^{\theta  \bar M_j^{(n)}}]\le 1 + C_{\ref{exponential_moments}} \theta s_j.
\]
For $n=0$, using the convexity of the exponential function and $\theta s_j\le 1$, we have
\[e^{\theta \bar M_j^{(0)}}=
e^{2\theta s_j}\le 1 + (e^2-1)
\theta s_j,
\]
whence the result holds if $C_{\ref{exponential_moments}}\ge e^2-1$. 
For the induction step, write
\[
\bar M_j^{(n)}= 2 s_j+ \sum_{j_1 \text{ weak}} 2\lambda C_{j,j_1} \bar M_{j,j_1}^{(n)},
\]
where the random variables 
\[
\bar M_{j,j_1}^{(n)}:=2 s_{j_1} +2\sum_{k=2}^n \sum_{j_2,\ldots,j_n} \left(\prod_{l=2}^n 2\lambda C_{j,j_1, \ldots,j_l}\right) s_{j_k}
\]
are independent, independent from the variables $C_{j,j_1}$, and have the same laws as $\bar M_{j_1}^{(n-1)}$. Further,
\begin{align*}
	\E[e^{\theta \bar M_j^{(n)}}]&=e^{2\theta s_j } \prod_{j_1 \text{ weak} }\E[e^{2\lambda \theta C_{j,j_1} M_{j,j_1}^{(n)} }] 
	= e^{2\theta s_j} \prod_{j_1 \text{ weak} }\left( 1 + p_{jj_1}\left( \E[e^{2\lambda \theta \bar M_{j_1}^{(n-1)}}]-1\right)\right).
\end{align*}
From the hypothesis $\theta\le 1/(2\lambda s(\astr))$, it follows $2\lambda \theta \le 1/s(\astr)\le 1/s_{j_1}$ for all $j_1$ weak, as well as  $2\lambda \theta\le \theta\le 1/(2\lambda s(\astr))$, since $\lambda\le 1/2$. We can thus apply the induction hypothesis to 
obtain
\begin{align*}
	\E[e^{\theta \bar M_j^{(n)}}]&\le e^{2\theta s_j} \prod_{j_1 \text{ weak} }\left( 1 + 2\lambda \theta C_{\ref{exponential_moments}} p_{jj_1} s_{j_1}\right)\\
	&\le \exp\left(2\theta s_j+2  C_{\ref{exponential_moments}} \theta \sum_{j_1 \text{ weak}} \lambda p_{jj_1} s_{j_1}\right) 
	\le \exp\left(2\theta s_j\big(1 + \frac {C_{\ref{exponential_moments}}}{\CMI}\big)\right),
\end{align*}
where we used~\eqref{OMIweak}.
Using now the convexity of the exponential function and $2 \theta s_j(1 + C_{\ref{exponential_moments}}/\CMI)\le 2(1+C_{\ref{exponential_moments}}/\CMI)$, we obtain
\[
\E[e^{\theta \bar M_j^{(n)}}]\le 1 + \left(e^{2(1 + C_{\ref{exponential_moments}}/\CMI)}-1\right)\ \theta s_j.
\]
To perform the recurrence step, we need $e^{2(1 +C_{\ref{exponential_moments}}/\CMI)}-1\le C_{\ref{exponential_moments}}$, which we can ensure by choosing $C_{\ref{exponential_moments}}>e^2-1$ and $\CMI$ sufficiently large. This concludes the proof of the first part of the proposition.%
\smallskip

For the second part, we proceed exactly as above to bound $\E[e^{\theta \bar M_j^{(n)}}]$ by the expression $\exp\left(2\theta (1+C_{\ref{exponential_moments}}/\CMI) s_j\right)$. This is the required upper bound, by increasing the value of $C_{\ref{exponential_moments}}$ if necessary. Note also that we cannot further bound this without the hypothesis $\theta s_j \le 1$.
\medskip

For the third and fourth parts, the computations are similar. For $\theta\le 1/\lambda s(\astr)$, which ensures $\lambda \theta \le \min(1/2\lambda s(\astr), 1/s_j)$, we obtain
\begin{align*}
	\E[e^{\theta \bar J_i}]&\le \prod_{j\text{ weak}} \left(1+p_{ij}\left(\E[e^{\lambda \theta  \bar M_j}] - 1\right)\right) 
	\le e^{C_{\ref{exponential_moments}}\theta \sigma_i}.
\end{align*}
For  $\theta\le \min(1, 1/(2\lambda s(\astr)))$, which ensures $\theta/s_j\le \min(1/2\lambda s(\astr), 1/s_j)$, we obtain
\[
\E[e^{\theta \bar K_i}]\le e^{C_{\ref{exponential_moments}}\theta \sum_{j \text{weak}} p_{ij}} \le e^{2C_{\ref{exponential_moments}} \theta \kappa_0 \pi_i}.
\]
\\[-11mm]
\end{proof}
\smallskip

\begin{corollary}\label{LDBoundsFixedTime0}
	Under the same conditions, choosing a value $C_{\ref{exponential_moments}}>2$ for which Proposition~\ref{exponential_moments}(2) holds, we also have, for any $j$ weak,
	\begin{equation}
	\label{LD_coro_sy}
	\P(\bar M_j\ge 2 C_{\ref{exponential_moments}} s(\astr))\le e^{-\lambda^{-1}}.
	\end{equation}
\end{corollary}

\begin{proof}
	For $\theta=1/(2\lambda s(\astr))$, Markov inequality and Proposition~\ref{exponential_moments}\emph{(2)} yield
	\begin{align*}
	\P\left(\bar M_j\ge 2 C_{\ref{exponential_moments}} s(\astr)\right)&\le e^{-2C_{\ref{exponential_moments}} \theta s(\astr)}\E[e^{\theta \bar M_j}]
	\le e^{-2 C_{\ref{exponential_moments}} \theta s(\astr)} e^{C_{\ref{exponential_moments}}\theta s_j} \\
	&\le e^{- C_{\ref{exponential_moments}} \theta s(\astr)}=e^{-C_{\ref{exponential_moments}}/2\lambda},
	\end{align*}
	using also that the score $s_j$ associated to the weak vertex $j$ is smaller than $s(\astr)$.
\end{proof}

\begin{corollary}\label{LDBoundsFixedTime} 
Assume that the additional technical condition $(H2)$ holds, yielding some $\alpha\in (0,1]$, and $c_{\eqref{cond:H2}}>0$, not depending on $\lambda$, such that
\begin{align*}
	s(\astr)&\le c_{\eqref{cond:H2}} \lambda^{-3+\alpha}, \\
	s(\astr)&\le c_{\eqref{cond:H2}} \lambda^{-1+\alpha} \inf \big\{ \tfrac{s(x)}{T^{\rm{loc}}(x)} \colon x<\astr\big\}.
\end{align*}
For any $c_\eqref{LD_coro_J}>0$ and large $\ratiost>0$, taking $\CMI$ sufficiently large, there exists a constant $c_{\ref{LDBoundsFixedTime}}$ such that for small $\lambda$, for any strong vertex $i$ and weak vertex $j$, we have
\begin{align}
	\label{LD_coro_sy2}
	\P(\bar M_j\ge \pi_{\astr} s_j)&\le e^{-c_{\ref{LDBoundsFixedTime}} \lambda^{-\alpha}},\\
	\label{LD_coro_J}
	\P(\bar J_i \ge c_\eqref{LD_coro_J} \tfrac {s_i}{T^{\rm{loc}}_i})&\le e^{-c_{\ref{LDBoundsFixedTime}} \lambda^{-\alpha}}, \\ 
	\label{LD_coro_K}
	\P(\bar K_i \ge \tfrac {\ratiost}{10} \pi_i)&\le e^{-c_{\ref{LDBoundsFixedTime}} \lambda^{-\alpha}}.
\end{align}
\end{corollary}

\begin{proof}
To obtain $\eqref{LD_coro_sy2}$, proceed as in the proof of~\eqref{LD_coro_sy}  to get
\[
\P\left(\bar M_j\ge \pi_{\astr} s_j\right)\le e^{- \frac  {\pi_{\astr}-C_{\ref{exponential_moments}}}{\AL{2}\lambda s(\astr)}},
\]
and conclude by observing that $\lambda^2 \pi_{\astr}$ and $\lambda^{\alpha-3}/s(\astr)$ are both lower bounded by positive constants.
To obtain \eqref{LD_coro_J}, take $\theta= 1/(\lambda s(\astr))$ to get
\begin{align*}
	\P\left(\bar J_i\ge c_\eqref{LD_coro_J} \tfrac {s_i}{T^{\rm{loc}}_i}\right)
	&\le e^{-c_\eqref{LD_coro_J} \theta \frac {s_i}{T^{\rm{loc}}_i}} \E[e^{\theta \bar J_i}] 
	\le e^{(-c_\eqref{LD_coro_J} +C_{\ref{exponential_moments}}/\CMI) \frac {s_i}{\lambda s(\astr) T^{\rm{loc}}_i}} 
	\le e^{ \frac {-c_\eqref{LD_coro_J} +c/\CMI} {c_{\eqref{cond:H2}}}\lambda^{-\alpha}},
\end{align*}
using $(H2)$, provided the constant \smash{$-c_\eqref{LD_coro_J}+C_{\ref{exponential_moments}}/\CMI$} is negative, which we can ensure by choosing $\CMI$ large.
Finally, to obtain \eqref{LD_coro_K}, take $\theta=\lambda^{2-\alpha}/2\le \min(1, 1/2\lambda s(\astr))$ and $\ratiost/10$ larger than $C_{\ref{exponential_moments}}$ given by Proposition~\ref{exponential_moments}\emph{(4)}. Then,
\[
\P\left(\bar K_i\ge \frac \ratiost{10} \pi_x\right) \le e^{-(\frac \ratiost{10}-C_{\ref{exponential_moments}}) \theta \pi_i}\le e^{-c  \lambda^{2-\alpha} \pi_{\astr}}\]
	for some small constant $c$. We conclude again by observing that $\lambda^2 \pi_{\astr}$ is lower bounded by a positive constant.
\end{proof}

\subsubsection*{Exploring and bounding the connectivities of the visited strong vertices}

We now turn to the true environment encountered by the infection. From now on, we assume that condition $(H2)$ holds, so in particular we can apply Corollary~\ref{LDBoundsFixedTime}. For a strong vertex $i\in G_t$, recall the definition of $J_i(t)$ and decompose it as $J_i^{\rm{out}}(t)+ J_i^{\rm{in}}(t)$ with
\[
J_i^{\rm{out}}(t):= \sum_{\substack{j \text{ weak}\\ (i,j) \text{ present}}}\lambda M_j(t),\qquad J_i^{\rm{in}}(t):= \sum_{\substack{j \text{ weak}\\ (j,i) \text{ present}}}\lambda M_j(t).
\]
Recalling that there is (at most) one edge $(j,i)$ in $G_t$ pointing to $i$, the sum in the definition of $J_i^{\rm{in}}$  contains only one or zero term. Similarly, we write $K_i(t)=K_i^{\rm{out}}(t)+ K_i^{\rm{in}}(t)$ with
\[
K_i^{\rm{out}}(t):= \sum_{\substack{j \text{ weak}\\ (i,j) \text{ present}}}\frac { M_j(t)}{s_j},\quad K_i^{\rm{in}}(t):= \sum_{\substack{j \text{ weak}\\ (j,i) \text{ present}}}\frac {M_j(t)}{s_j}.
\]
For a weak vertex $i$, we also decompose $M_i(t)=M_i^{\rm{out}}(t)+M_i^{\rm{in}}(t)$ with 
\begin{align*}
	M_i^{\rm{out}}(t)&:= \tilde s_i+ \sum_{\substack{j \text{ weak}\\ (i,j) \text{ present}}}\ \sum_{\substack{k\in \tree\\ d(i,k)=d(j,k)+1}} (2\lambda)^{d(i,k)} \tilde s_k,\\
	M_i^{\rm{in}}(t)&:= \sum_{\substack{j \text{ weak}\\ (j,i)\text{ present}}}\
	\sum_{\substack{k\in \tree\\ d(i,k)=d(j,k)+1}} (2\lambda)^{d(i,k)} \tilde s_k.
\end{align*}
By the following lemma, the idealized environment provides an upper bound for the quantities $M_i^{\rm{out}}$, $J_i^{\rm{out}}$ and $K_i^{\rm{out}}$.
\begin{lemma}\label{lem:coupling_exploredtree}
	\ \\[-3mm]
	\begin{itemize}[leftmargin=*]
		\item If $j$ is a weak vertex and $t\ge 0$, then, conditionally on $j$ being visited, the random variable $M_j^{\rm{out}}(t)\one_{t< T_{\rm{bad}}^{\rm{tree}}}$ is stochastically bounded by $\bar M_j$.
		\item If $i$ is a strong vertex and $t\ge 0$, then, conditionally on $i$ being visited, the random variables \smash{$J_i^{\rm{out}}(t) \one_{t< T_{\rm{bad}}^{\rm{tree}}}$} and \smash{$K_i^{\rm{out}}(t)\one_{t< T_{\rm{bad}}^{\rm{tree}}}$} are stochastically bounded by $\bar J_i$ and $\bar K_i$ respectively.\\[-9mm]
	\end{itemize},
\end{lemma}
\begin{proof}
	For a visited weak vertex, we can {write} 
	$M_j^{\rm{out}}(t)$ on the event $\{t< T_{\rm{bad}}^{\rm{tree}}\}$ as
	\[
	\tilde s_j +  \sum_{n\ge 1}(2\lambda)^n \sum_{\substack{j_1,\ldots,j_n \text{ distinct}\\\text{weak vertices}}}\one_{\{\text{the edges }(j,j_1), \ldots, (j_{n-1},j_n) \text{ are in } G_t\}}\tilde s_{j_n}.
	\]
	Recalling the bound $\tilde s_k\le 2s_k$, it remains to couple the graph and the random variables $C$ in such a way that for any directed path $(j,j_1,\ldots,j_n)$ in
	$G_t$, we have $C_{j,\ldots,j_n}=1$. We proceed by induction on $n$. If $j_{n-1}$ is first visited at time $s\le t$, then for every weak vertex $j_n$ which is not in $G_{s-}$, the edge $(j_{n-1},j_n)$ is present in 
	$G_t$ with the stationary probability $p_{j_{n-1},j_n}$. We then couple the connection at time $t$ with $C_{j,\ldots,j_n}$. Note that we consider times smaller than $T_{\rm{bad}}^{\rm{tree}}$ {so the connection of an edge $(j_{n-1},j_n)$ is coupled to at most one variable $C_{j,\ldots,j_n}$, thus maintaining the independence between the $C$ variables. }
	The second part of the lemma with strong vertices is similar.
\end{proof}\pagebreak[3]

In the rest of this proof specifically, we write $\tilde T$ for $T_\ell\wedge T_{\rm{bad}}^{\rm{tree}} \wedge T_{\rm{bad}}^{\rm{deg}}\wedge \tilde T_{\rm{deg}}\wedge T_{\rm{hit}}\wedge t_0$.
\begin{corollary}\label{LD_outcontributions}
	Considering the same settings as in Corollary~\ref{LDBoundsFixedTime}, say a weak vertex $j$ is \emph{bad} at time $t$ if it is visited and satisfies 
	\[
	M_j^{\rm{out}}(t)\ge 2 C_{\ref{exponential_moments}} s(\astr)\wedge \pi_{\astr} s_j.
	\] 
	Similarly, say a strong vertex $i$ is bad at time $t$ if it is visited and satisfies any of the following inequalities:
	\begin{align*}
		J_i^{\rm{out}}(t)&\ge c_\eqref{LD_coro_J} \frac {s_i}{T_i^{\rm{loc}}},\\
		K_i^{\rm{out}}(t)&\ge \frac r {10} \pi_i. 
	\end{align*}
	Finally, say a vertex is bad if it is bad at some time $t<\tilde T$. 
	%
	Then the probability of finding any bad vertex is bounded by $C_{\ref{LD_outcontributions}} \ell^3 a^{-2\gamma} t_0 e^{-c_{\ref{LDBoundsFixedTime}}\lambda^{-\alpha}}$, for some finite constant $C_{\ref{LD_outcontributions}}$.
\end{corollary}
We stress that the definition of bad vertices in this corollary actually depends on the parameters $c_\eqref{LD_coro_J}$ and $r$ (that we have still not fixed), and also requires a constraint of large~$\CMI$ to apply Corollary~\ref{LDBoundsFixedTime}. We discuss later how all these parameters are chosen.
\begin{proof}
	Note that at times $t<\tilde T$, the number of visited vertices is bounded by $\ell$, while the number of nonvisited weak vertices in $G_t$ is bounded by $(2c_2 a^{-\gamma} \vee 1/8\lambda^2)\ell$. Further, let $\ell'=\ell+(2c_2 a^{-\gamma} \vee 1/8\lambda^2)\ell$, and choose a vertex $i$ conditioned to be in $G_{\tilde T}$. Suppose that the vertex turns bad at some time \smash{$t<\tilde T$}. Then, with probability at least $1/e$, the vertex $i$ and the weak vertices in $G_t$ do not update until time $t+\tau$, with $\tau=1/\kappa_0 \ell'^2$. Therefore,  with probability at least $e^{-1}$, we have $t+\tau<T_{\rm{bad}}^{\rm{tree}}$ and the value of $M_i$ (if $i$ is weak) or of $J_i$ and $K_i$ (if $i$ is strong) do not decrease on this time interval. In particular, there exists a time $s\in \tau \N \cap [0,\tilde T+\tau)$ which is less than $T_{\rm{bad}}^{\rm{tree}}$ and at which we observe a large value for $M_i(s)$ (or $J_i(s)$ or $K_i(s)$). By Lemma~\ref{lem:coupling_exploredtree} and Corollaries~\ref{LDBoundsFixedTime0} and~\ref{LDBoundsFixedTime}, the probability that such a time $s$ exists, is bounded by $(t_0/\tau+1)e^{-c_{\ref{LDBoundsFixedTime}} \lambda^{-\alpha}}.$ We deduce successively a bound of
	\smash{$e(t_0/\tau+1)e^{-c_{\ref{LDBoundsFixedTime}} \lambda^{-\alpha}}$} for the probability that a visited vertex turns bad, and a bound of 
	\smash{$\ell e(t_0/\tau+1)e^{-c_{\ref{LDBoundsFixedTime}} \lambda^{-\alpha}}$} for the probability of finding any bad vertex.
\end{proof}
We now suppose that we do not encounter any bad vertex. We will then  prove recursively 
that all the visited weak vertices $j$ satisfy
\begin{equation}\label{ineq:good_weak_vertex}
	M_j(t)< 3 C_{\ref{exponential_moments}} s(\astr)\wedge 2\pi_{\astr} s_j,
\end{equation}
at all times $t<\tilde T$.
For the first visited weak vertex, we have $M_i(t)=M_i^{\rm{out}}(t)$, so the bound follows from the fact that this vertex is good. Now, for a newly visited weak vertex $i$, we can bound
$M_i^{\rm{out}}$ by $2 C_{\ref{exponential_moments}} s(\astr)\wedge \pi_{\astr} s_i$, so it suffices to bound $M_i^{\rm{in}}$ by $C_{\ref{exponential_moments}} s(\astr)\wedge \pi_{\astr} s_i$. But the sum on $j$ in the definition of $M_i^{\rm{in}}$ contains at most one term, which corresponds to a previously visited weak vertex $j$ and can be bounded by $2\lambda M_j$. By the recurrence hypothesis, we can further bound this term by $6\lambda C_{\ref{exponential_moments}} s(\astr)$, which is smaller than $C_{\ref{exponential_moments}} s(\astr)$ if $\lambda\le 1/6$. For small $\lambda$, we also obtain the bound $2\lambda M_j\le \pi_{\astr}\le  \pi_{\astr} s_i$ by observing that we have $ \pi_{\astr}\ge 1/20\kappa_0 \lambda^{-2}$ and recalling that by condition $(H2)$ we have $s(\astr)\le c_{\eqref{cond:H2}} \lambda^{-3+\alpha}$.\\
Having proved that all visited weak vertices satisfy~\eqref{ineq:good_weak_vertex}, we deduce the following bound for any visited strong vertex $i$. 
\begin{align*}
	J_i^{\rm{in}}&\le 3\lambda C_{\ref{exponential_moments}} s(\astr)
	\le 3\lambda^\alpha C_{\ref{exponential_moments}} c_{\eqref{cond:H2}} \frac {s_i}{T_i^{\rm{loc}}},
\end{align*}
where we used condition $(H2)$ to obtain the second inequality. Taking $\lambda$ small, we can further bound this by $c_\eqref{LD_coro_J} s_i/T_i^{\rm{loc}}$. For the term $K_i^{\rm{in}}$ we obtain similarly a simple bound, namely $2\pi_{\astr}\le 2\pi_i$. Combining this with the fact that all visited strong vertices are good, we obtain that they all satisfy, at all times $t<\tilde T$,
\begin{align}
	\label{LD_for_J_effective}	
	J_i(t)&\le \frac {2c_\eqref{LD_coro_J} s_i}{T_i^{\rm{loc}}}=2c_\eqref{LD_coro_J} r \kappa_i t_i,\\
	\label{LD_for_K_effective}
	K_i(t)&\le \left(2+ \frac r {10}\right) \pi_i.
\end{align}
We thus finally obtained bounds that resemble~\eqref{LD_for_J} and~\eqref{LD_for_K}, and it only remains to explain in detail how we fix the values of $r$, $\CMI$, $\eps$, $c_\eqref{LD_coro_J}$.\smallskip\pagebreak[3]

First, recall we have to fix the values $\CMI$, $r$ and $\eps$, so that if the scoring function $s$ is assumed to satisfy Inequalities~\eqref{OMIweak},~\eqref{OMIstrongquick} and~\eqref{OMIstrongslow} with this value of $\CMI$, and if we use this value of $r$ in the definition the vertex scores $t_i$ (and thus of the configuration score $M(t)$), then we can ensure that Inequalities~\eqref{LD_for_J} and~\eqref{LD_for_K} are likely to be satisfied up until time $\tilde T$. In order to do so, we first request $\CMI>8\ratiost$ and fix  a small value of $\eps$ so as to have $4\ratiost/3\CMI+\eps\le 1/6.$ Then, choosing $r$ large, we can ensure that we have \smash{$(4 \ratiost - 6 - \tfrac {4\ratiost} {3\CmMI}-\eps)/32\ge 2+r/10$}, so that~\eqref{LD_for_J} is implied by~\eqref{LD_for_J_effective}. Now, choosing $c_\eqref{LD_coro_J}=1/12 r$, we also have \smash{$(\tfrac 1 3- \tfrac {4\ratiost} {3\CmMI}-\eps)\ge \tfrac 1 6=2c_\eqref{LD_coro_J} r$}, so that~\eqref{LD_for_K} is implied by~\eqref{LD_for_K_effective}.\smallskip

At this point, we are still free to increase the value of $r$ if we keep the request $\CMI\ge 8r$ and $c_\eqref{LD_coro_J}=1/12r$. We choose the (large) value of $r$ and $c_\eqref{LD_coro_J}=1/12r$, so as to apply Corollary~\ref{LDBoundsFixedTime}. Eventually, we fix the large value of $\CMI$ necessary to apply Corollaries~\ref{LDBoundsFixedTime0} and~\ref{LDBoundsFixedTime}, fixing at the same time the value $C_{\ref{exponential_moments}}$. We then have a proper definition of bad vertices and can apply Corollary~\ref{LD_outcontributions} to deduce that each visited strong vertex satisfies~\eqref{LD_for_J_effective} and~\eqref{LD_for_K_effective} and thus~\eqref{LD_for_J} and~\eqref{LD_for_K}, with failure probability bounded by $C_{\ref{LD_outcontributions}} \ell^3 a^{-2\gamma} t_0 e^{-c_{\ref{LDBoundsFixedTime}}\lambda^{-\alpha}}$. This bound on $P^{\rm{conn}}$ is stretched exponentially small in $\lambda$, and hence negligible compared to $a$ as $\lambda\to 0$.
\bigskip

\noindent {\bf Acknowledgements:} AL was partially supported by the CONICYT-PCHA/Doctorado nacional/2014-21141160 scholarship, grant GrHyDy ANR-20-CE40-0002, and by 
FONDECYT grant 11221346. PM was partially supported by DFG project 444092244 ``Condensation in random geometric graphs" within the priority programme SPP~2265.


\bigskip%


{\footnotesize
\noindent
{\bf Emmanuel Jacob}, Ecole Normale Sup\'erieure de Lyon, Unit\'e de Math\'ematiques Pures et Appliqu\'ees,  UMR CNRS 5669 , 46, All\'ee d'Italie, 69364 Lyon Cedex 07, France.}

\medskip%

{\footnotesize
	\noindent
	{\bf Amitai Linker}, Departamento de Matem\'aticas, Facultad de Ciencias Exactas, Universidad Andr\'es Bello, \linebreak Sazi\'e 2212, Santiago, Chile.}

\medskip%

{\footnotesize
\noindent
{\bf Peter M\"orters}, Universit\"at zu K\"oln,  Mathematisches Institut, Weyertal 86--90, 50931~K\"oln, Germany.}

\end{document}